\documentclass[11pt]{article}

\title{Proof of the Arnold chord conjecture in three dimensions II}
\author{Michael Hutchings and Clifford Henry Taubes}
\date{}

\usepackage{amssymb}
\usepackage{latexsym}
\usepackage{amsmath}
\usepackage{amsthm}
\usepackage{amscd}

\newcommand{\mc}[1]{{\mathcal #1}}

\numberwithin{equation}{section}

\newtheorem{theorem}{Theorem}[section]
\newtheorem{proposition}[theorem]{Proposition}
\newtheorem{corollary}[theorem]{Corollary}
\newtheorem{lemma}[theorem]{Lemma}
\newtheorem{lemma-definition}[theorem]{Lemma-Definition}

\theoremstyle{definition}
\newtheorem{definition}[theorem]{Definition}
\newtheorem{remark}[theorem]{Remark}

\newtheorem{example}[theorem]{Example}

\newcommand{\energy}{\textsc{E}}

\newcommand{\eqdef}{\;{:=}\;}

\renewcommand{\frak}{\mathfrak}

\newcommand{\C}{{\mathbb C}}

\newcommand{\R}{{\mathbb R}}

\newcommand{\Z}{{\mathbb Z}}
\newcommand{\Sp}{{\mathbb S}}
\newcommand{\A}{{\mathbb A}}
\newcommand{\op}{\operatorname}
\newcommand{\dbar}{\overline{\partial}}
\newcommand{\zbar}{\overline{z}}
\newcommand{\wbar}{\overline{w}}

\newcommand{\Spinc}{\op{Spin}^c}

\newcommand{\Spin}{\op{Spin}}
\newcommand{\End}{\op{End}}

\newcommand{\Ker}{\op{Ker}}

\newcommand{\tensor}{\otimes}

\newcommand{\vu}{\nu}
\newcommand{\neta}{\eta}

\newcommand{\union}{\bigcup}

\newcommand{\bpm}{\begin{pmatrix}}
\newcommand{\epm}{\end{pmatrix}}

\renewcommand{\epsilon}{\varepsilon}

\begin{document}

\setcounter{tocdepth}{2}

\maketitle

\begin{abstract}
In ``Proof of the Arnold chord conjecture in three dimensions I'', we
deduced the Arnold chord conjecture in three dimensions from another
result, which asserts that an exact symplectic cobordism between
contact three-manifolds induces a map on (filtered) embedded contact
homology satisfying certain axioms.  The present paper proves the
latter result, thus completing the proof of the three-dimensional
chord conjecture.  We also prove that filtered embedded contact
homology does not depend on the choice of almost complex structure
used to define it. 
\end{abstract}

\tableofcontents

\section{Introduction}

The main goal of this paper is to prove that an exact symplectic
cobordism between contact 3-manifolds induces a map on (filtered)
embedded contact homology (ECH) satisfying certain axioms.  This
result appears here as Theorem~\ref{thm:cob}, and was previously
stated in \cite[Thm.\ 2.4]{cc}, where it was used to prove the Arnold
chord conjecture in three dimensions.  This result also has additional
applications, for example it gives rise to new obstructions to
symplectic embeddings in four dimensions, see \cite{qech}.  Along the
way to proving Theorem~\ref{thm:cob}, we will also prove that filtered
ECH does not depend on the choice of almost complex structure used to
define it (Theorem~\ref{thm:FECH} below).  Although this paper is a
sequel to \cite{cc}, we will not use anything from the latter paper
except for some basic definitions.  We begin by briefly reviewing
these definitions.  For more about ECH, see \cite{icm,bn} and the
references therein.

\subsection{Embedded contact homology}
\label{sec:ech}

Let $Y$ be a closed oriented $3$-manifold. (For simplicity, all $3$-manifolds in this
paper are assumed connected except where otherwise stated.)  Let
$\lambda$ be a contact form on $Y$, let $R$ denote the associated Reeb
vector field, and let $\xi=\Ker(\lambda)$ denote the associated
contact structure.  Assume that $\lambda$ is nondegenerate, i.e.\ all
Reeb orbits are nondegenerate.

Let $J$ be an almost complex structure on $\R\times Y$ such that $J$
is $\R$-invariant, $J(\partial_s)=R$ where $s$ denotes the $\R$
coordinate, and $J$ sends $\xi$ to itself, rotating $\xi$ positively
with respect to the orientation on $\xi$ given by $d\lambda$.  We call
such an almost complex structure {\em symplectization-admissible\/}.
The reason for the terminology is that the noncompact symplectic
manifold $(\R\times Y,d(e^s\lambda))$ is called the
``symplectization'' of $(Y,\lambda)$.  Note that a
symplectization-admissible almost complex structure is equivalent to
an almost complex structure $J$ on $\xi$ which rotates positively with
respect to $d\lambda$.  In particular, the space of
symplectization-admissible almost complex structures is contractible.

Given a generic symplectization-admissible $J$, and given $\Gamma\in
H_1(Y)$, the {\em embedded contact homology\/}
$ECH_*(Y,\lambda,\Gamma;J)$ is the homology of a chain complex
$ECC_*(Y,\lambda,\Gamma;J)$ defined as follows.  Recall that an {\em
  orbit set\/} is a finite set of pairs $\Theta=\{(\Theta_i,m_i)\}$
where the $\Theta_i$'s are distinct embedded Reeb orbits, and the
$m_i$'s are positive integers.  The homology class of the orbit set
$\Theta$ is defined by
\[
[\Theta] \eqdef \sum_i m_i[\Theta_i] \in H_1(Y).
\]
The orbit set $\Theta=\{(\Theta_i,m_i)\}$ is called {\em admissible\/}
if $m_i=1$ whenever $\Theta_i$ is hyperbolic, i.e.\ the linearized
Reeb flow around $\Theta_i$ has real eigenvalues.  Define
$ECC_*(Y,\lambda,\Gamma;J)$ to be the free $\Z/2$-module generated by
admissible orbit sets $\Theta$ with $[\Theta]=\Gamma$.  Although ECH
can also be defined over $\Z$, see \cite[\S9]{obg2}, in this paper we
always use $\Z/2$ coefficients for simplicity.

To specify the differential $\partial$ on the chain complex, we need
the following:

\begin{definition}
\label{def:Jhol}
Given a symplectization-admissible $J$, and given orbit sets
$\Theta=\{(\Theta_i,m_i)\}$ and $\Theta'=\{(\Theta'_j,m'_j)\}$, define
a ``$J$-holomorphic curve from $\Theta$ to $\Theta'$'' to be a
$J$-holomorphic curve in $\R\times Y$ (whose domain is a possibly
disconnected punctured compact Riemann surface) with positive ends at
covers of $\Theta_i$ with total multiplicity $m_i$, negative ends at
covers of $\Theta'_j$ with total multiplicity $m'_j$, and no other
ends.  Here a {\em positive end\/} of a holomorphic curve at a (not
necessarily embedded) Reeb orbit $\gamma$ is an end which is
asymptotic to the cylinder $\R\times\gamma$ as the $\R$ coordinate
$s\to +\infty$.  A {\em negative end\/} is defined analogously with
$s\to-\infty$.  Let $\mc{M}^J(\Theta,\Theta')$ denote the moduli space
of $J$-holomorphic curves from $\Theta$ to $\Theta'$, where two such
curves are considered equivalent if they represent the same current in
$\R\times Y$, up to translation of the $\R$ coordinate.
\end{definition}

Given admissible orbit sets $\Theta$ and $\Theta'$ with
$[\Theta]=[\Theta']=\Gamma$, the differential coefficient
$\langle\partial\Theta,\Theta'\rangle\in \Z/2$ is defined to be the
mod 2 count of $J$-holomorphic curves in $\mc{M}^J(\Theta,\Theta')$
with ``ECH index'' equal to $1$.  For the definition of the ECH index
see \cite{pfh2,ir}.  If $J$ is generic, then $\partial$ is
well-defined and $\partial^2=0$, as shown in \cite[\S7]{obg1}.  A
symplectization-admissible almost complex structure that is generic in
this sense will be called {\em ECH-generic\/} here.

The ECH index defines a relative $\Z/d(c_1(\xi)+2\op{PD}(\Gamma))$
grading on the chain complex, where $d$ denotes divisibility in
$H^2(Y;\Z)/\op{Torsion}$.  However the grading will not play a major
role in this paper.

It is shown in \cite{e1,e2,e3,e4} that ECH is isomorphic to a version
of Seiberg-Witten Floer cohomology as defined by Kronheimer-Mrowka
\cite{km}.  The precise statement is that there is a canonical
isomorphism of relatively graded $\Z/2$-modules\footnote{Ordinarily
  $\widehat{HM}^*$ is defined over $\Z$, see \cite{km}, and it is
  shown in \cite{e3} that one can lift the isomorphism
  \eqref{eqn:echswfJ} to $\Z$ coefficients.}
\begin{equation}
\label{eqn:echswfJ}
ECH_*(Y,\lambda,\Gamma;J) \simeq \widehat{HM}^{-*}(Y,\frak{s}_{\xi,\Gamma}).
\end{equation}
Here $\widehat{HM}^*$ denotes Seiberg-Witten Floer cohomology with
$\Z/2$ coefficients, and $\frak{s}_{\xi,\Gamma}$ denotes the spin-c
structure $\frak{s}_\xi+\op{PD}(\Gamma)$ on $Y$, where $\frak{s}_\xi$
denotes the spin-c structure determined by oriented $2$-plane field
$\xi$, see Example~\ref{ex:pf}.  

\subsection{Filtered ECH}
\label{sec:filtered}

If $\Theta=\{(\Theta_i,m_i)\}$ is an orbit set, its {\em symplectic
  action\/} or {\em length\/} is defined by
\begin{equation}
\label{eqn:length}
\mc{A}(\Theta) \eqdef \sum_i m_i \int_{\Theta_i}\gamma.
\end{equation}
Since $J$ is symplectization-admissible, it follows that the ECH
differential decreases the action, i.e.\ if
$\langle\partial\Theta,\Theta'\rangle\neq 0$ then
$\mc{A}(\Theta)>\mc{A}(\Theta')$.  Thus for any real number $L$, it
makes sense to define the {\em filtered ECH\/}, denoted by
$ECH_*^{L}(Y,\lambda,\Gamma;J) $, to be the homology of the subcomplex
$ECC_*^L(Y,\lambda,\Gamma;J)$ of the ECH chain complex spanned by ECH
generators with action less than $L$.

There are various natural maps defined on filtered ECH.  First, if
$L<L'$ then there is a map
\begin{equation}
\label{eqn:iLL'J}
\imath^{L,L'}_J:ECH_*^{L}(Y,\lambda,\Gamma;J) \longrightarrow
ECH_*^{L'}(Y,\lambda,\Gamma;J)
\end{equation}
induced by the inclusion of chain complexes.
The usual ECH is recovered as the direct limit
\begin{equation}
\label{eqn:edr}
ECH_*(Y,\lambda,\Gamma;J) = \lim_{L\to\infty}ECH_*^{L}(Y,\lambda,\Gamma;J).
\end{equation}

In addition, if $c$ is a positive constant, then there is a canonical
``scaling'' isomorphism
\begin{equation}
\label{eqn:scalingJ}
s_J:ECH_*^{L}(Y,\lambda,\Gamma;J) \stackrel{\simeq}{\longrightarrow}
 ECH_*^{cL}(Y,c\lambda,\Gamma;J^c),
\end{equation}
where $J^c$ is defined to agree with $J$ when restricted to the
contact planes $\xi$.  This is because the chain complexes on both
sides have the same generators, and the self-diffeomorphism of
$\R\times Y$ sending $(s,y)\mapsto (cs,y)$ induces a bijection between
$J$-holomorphic curves and $J^c$-holomorphic curves.

Note that to define $ECH_*^L(Y,\lambda,\Gamma;J)$, one does not need the
full assumption that $\lambda$ is nondegeneric and $J$ is ECH-generic,
but only the following conditions:

\begin{definition}
\label{def:Lnondeg}
The contact form $\lambda$ is {\em $L$-nondegenerate\/} if all Reeb
orbits of length less than $L$ are nondegenerate, and if there is no
orbit set\footnote{The condition that there is no orbit set of action
  exactly $L$ is not needed to define filtered ECH, but it will be
  convenient to choose $L$ this way when we relate filtered ECH to
  Seiberg-Witten Floer cohomology, starting in Lemma~\ref{lem:diff}.}
of action exactly $L$.  Given an $L$-nondegenerate contact form
$\lambda$, a symplectization-admissible almost complex structure $J$
for $\lambda$ is {\em $ECH^L$-generic\/} if the genericity conditions
from \cite{obg2} hold for orbit sets of action less than $L$ so that
the ECH differential $\partial$ is well-defined on admissible orbit
sets of action less than $L$ and satisfies $\partial^2=0$.
\end{definition}

\subsection{$J$-independence of filtered ECH (statement)}

We now state a theorem asserting that filtered ECH and the various
maps on it do not depend on $J$.  Before stating the result, let us
recall precisely what it means to say that objects or maps between
them are independent of choices.

Let $\{G_i\mid i\in I\}$ be a collection of groups indexed by some
index set $I$.  We say that ``the groups $G_i$ are canonically
isomorphic to each other'', or ``$G_i$ does not depend on $i$'', if
for every pair $i_1,i_2\in I$ there is a canonical isomorphism
$\phi_{i_1,i_2}:G_{i_1}\stackrel{\simeq}{\to} G_{i_2}$, such that
$\phi_{i_2,i_3}\circ \phi_{i_1,i_2}=\phi_{i_1,i_3}$ for every triple
$i_1,i_2,i_3\in I$. In this case all the groups $G_i$ are canonically
isomorphic to a single group $G$. Specifically one can define $G$ to
be the disjoint union of the groups $G_i$, modulo the equivalence
relation that $g\in G_{i_1}$ is equivalent to $\phi_{i_1,i_2}(g)\in
G_{i_2}$, with group operation induced by the operations on the groups
$G_i$.

Now let $\{H_j\mid j\in J\}$ be another such collection of groups
which are canonically isomorphic to a single group $H$ via
isomorphisms $\psi_{j_1,j_2}:H_{j_1}\stackrel{\simeq}{\to} H_{j_2}$ as
above.  Then a collection of maps $\{f_{i,j}:G_i\to H_j\mid i\in I,
j\in J\}$ induces a well-defined map $f:G\to H$ provided that the
diagram
\[
\begin{CD}
G_{i_1} @>{f_{i_1,j_1}}>> H_{j_1} \\
@V{\phi_{i_1,i_2}}V{\simeq}V @V{\psi_{j_1,j_2}}V{\simeq}V \\
G_{i_2} @>{f_{i_2,j_2}}>> H_{j_2}
\end{CD}
\]
commutes for all $i_1,i_2\in I$ and $j_1,j_2\in J$.

With these conventions, we now have:

\begin{theorem}
\label{thm:FECH}
Let $Y$ be a closed oriented connected 3-manifold, and let $\Gamma\in
H_1(Y)$.
\begin{description}
\item{(a)} If $\lambda$ is an $L$-nondegenerate contact form on $Y$,
  then $ECH_*^L(Y,\lambda,\Gamma;J)$ does not depend on the choice of
  $ECH^L$-generic $J$, so we can denote it by
  $ECH_*^L(Y,\lambda,\Gamma)$.
\item{(b)} If $L<L'$ and if $\lambda$ is $L'$-nondegenerate, then the
  maps $i^{L,L'}_J$ in \eqref{eqn:iLL'J} induce a well-defined map
\begin{equation}
\label{eqn:istar}
i^{L,L'}: ECH_*^L(Y,\lambda,\Gamma) \longrightarrow
ECH_*^{L'}(Y,\lambda,\Gamma).
\end{equation}
\item{(c)} If $\lambda$ is a nondegenerate contact form on $Y$, then
  $ECH_*(Y,\lambda,\Gamma;J)$ does not depend on the choice of
  ECH-generic $J$, so we can denote it by $ECH_*(Y,\lambda,\Gamma)$.
\item{(d)}
If $c>0$, then the scaling isomorphisms $s_J$ in \eqref{eqn:scalingJ} induce
a well-defined isomorphism
\begin{equation}
\label{eqn:s}
s:ECH_*^{L}(Y,\lambda,\Gamma) \stackrel{\simeq}{\longrightarrow}
 ECH_*^{cL}(Y,c\lambda,\Gamma).
\end{equation}
\item{(e)}
The isomorphism \eqref{eqn:echswfJ} does not depend on $J$ and so
determines a canonical isomorphism
\begin{equation}
\label{eqn:echswf}
ECH_*(Y,\lambda,\Gamma) \simeq \widehat{HM}^{-*}(Y,\frak{s}_{\xi,\Gamma}).
\end{equation}
\end{description}
\end{theorem}

The proof of Theorem~\ref{thm:FECH} uses Seiberg-Witten theory,
similarly to parts of the proof of the isomorphism
\eqref{eqn:echswfJ}, and is given in \S\ref{sec:iso}.

\begin{remark}
  Although this is not necessary for the proof of the chord
  conjecture, the proof of Theorem~\ref{thm:FECH} works just as well
  with $\Z$ coefficients, by \cite{e3}.  Parts (a)-(d) of
  Theorem~\ref{thm:FECH} also hold for disconnected three-manifolds,
  by a straightforward modification of the proof.
\end{remark}

At times it is convenient to ignore the homology class $\Gamma$ in the
definition of $ECH$, and simply define
\begin{equation}
\label{eqn:gammadecomp}
ECH_*(Y,\lambda) \eqdef \bigoplus_{\Gamma\in
  H_1(Y)}ECH_*(Y,\lambda,\Gamma).
\end{equation}
This is the homology of a chain complex $ECC_*(Y,\lambda;J)$ generated
by all admissible orbit sets, and by \eqref{eqn:echswf} this homology
is canonically isomorphic (as a relatively graded $\Z/2$-module) to
\[
\widehat{HM}^{-*}(Y)\eqdef
\bigoplus_{\frak{s}\in\Spinc(Y)}\widehat{HM}^{-*}(Y,\frak{s}).
\]
Note that while $ECH_*(Y,\lambda)$ is a topological invariant of $Y$, the
filtered version $ECH_*^L(Y,\lambda)$ depends strongly on $\lambda$
and $L$.

\subsection{Exact symplectic cobordisms}
\label{sec:cobech}

If $Y_+$ and $Y_-$ are closed oriented (connected) 3-manifolds, our
convention is that a ``cobordism from $Y_+$ to $Y_-$'' is a compact
oriented smooth $4$-manifold with $\partial X = Y_+ - Y_-$.  Such a
cobordism induces a map of ungraded $\Z/2$-modules
\begin{equation}
\label{eqn:swcob}
\widehat{HM}^*(X): \widehat{HM}^{*}(Y_+) \longrightarrow \widehat{HM}^{*}(Y_-).
\end{equation}
If $\lambda_\pm$ are nondegenerate contact forms on $Y_\pm$, we define
\begin{equation}
\label{eqn:gencob}
\Phi(X):
ECH_*(Y_+,\lambda_+) \longrightarrow ECH_*(Y_-,\lambda_-)
\end{equation}
to be the map on ECH obtained by composing the map \eqref{eqn:swcob}
on Seiberg-Witten Floer cohomology with the canonical isomorphism
\eqref{eqn:echswf} on both sides.

If $(Y_\pm,\lambda_\pm)$ are as above, an {\em exact symplectic
  cobordism\/} from $(Y_+,\lambda_+)$ to $(Y_-,\lambda_-)$ is a pair
$(X,\omega)$, where $X$ is a cobordism from $Y_+$ to $Y_-$, and
$\omega$ is a symplectic form on $X$, such that there exists a
$1$-form $\lambda$ on $X$ with $d\lambda=\omega$ and
$\lambda|_{Y_\pm}=\lambda_\pm$.  A $1$-form with these properties is
called a {\em Liouville form\/} for $(X,\omega)$. When we wish to
specify a Liouville form, we denote the exact symplectic cobordism by
$(X,\lambda)$, and we continue to write $\omega=d\lambda$.

When $(X,\omega)$ is an exact symplectic cobordism as above, we would
like to relate the map \eqref{eqn:gencob} to holomorphic curves.
To prepare for this, let $\lambda$ be a Liouville form.  This
determines a Liouville vector field $V$ characterized by $\imath_V\omega =
\lambda$.  If $\varepsilon>0$ is sufficiently small, then the flow of $V$
starting on $Y_-$ for times in $[0,\varepsilon]$ defines a
diffeomorphism
\begin{equation}
\label{eqn:N-}
N_-\simeq [0,\varepsilon]\times Y_-
\end{equation}
where $N_-$ is (the closure of) a neighborhood of $Y_-$.  If $s$ denotes the
$[0,\varepsilon]$ coordinate in \eqref{eqn:N-}, then
$\lambda=e^s\lambda_-$ on $N_-$.  Likewise we obtain a neighborhood
\begin{equation}
\label{eqn:N+}
N_+\simeq [-\varepsilon,0]\times Y_+
\end{equation}
of $Y_+$ in which $\lambda=e^s\lambda_+$.  Using the identifications
\eqref{eqn:N-} and \eqref{eqn:N+}, one can then
glue symplectization
ends to $X$ to obtain the ``completion''
\begin{equation}
\label{eqn:completion}
\overline{X} \eqdef ((-\infty,0]\times Y_-) \cup_{Y_-} X \cup_{Y_+}
([0,\infty)\times Y_+),
\end{equation}
which is a noncompact symplectic 4-manifold. 

Note that the completion \eqref{eqn:completion} depends on the
Liouville form in the following sense: If $\lambda'$ is another
Liouville form for $\omega$, then the obvious identification between
the completions \eqref{eqn:completion} for $\lambda$ and $\lambda'$
is a homeomorphism, and will be a diffeomorphism if $\lambda$ and $\lambda'$ agree near $\partial X$.

\begin{definition}
\label{def:cobadm}
An almost complex structure $J$ on $\overline{X}$ is {\em
  cobordism-admissible\/} if it is $\omega$-compatible on $X$, and if
it agrees with symplectization-admissible almost complex structures
$J_+$ for $\lambda_+$ on $[0,\infty)\times Y_+$ and $J_-$ for
$\lambda_-$ on $(-\infty,0]\times Y_-$.
\end{definition}

Given a cobordism-admissible $J$, and given (not necessarily
admissible) orbit sets $\Theta^+=\{(\Theta_i^+,m_i^+)\}$ in $Y_+$ and
$\Theta^-=\{(\Theta_j^-,m_j^-)\}$ in $Y_-$, we define a
``$J$-holomorphic curve in $\overline{X}$ from $\Theta^+$ to
$\Theta^-$'' analogously to Definition~\ref{def:Jhol}, and denote the
moduli space of such curves by $\mc{M}^J(\Theta^+,\Theta^-)$, where
two such curves are considered equivalent if they represent the same
current in $\overline{X}$.  More generally, we make the following
definition:

\begin{definition}
\label{def:broken}
Let $J$, $J_\pm$ be as in Definition~\ref{def:cobadm}.  A {\em broken
  $J$-holomorphic curve from $\Theta^+$ to $\Theta^-$\/} is a
collection of holomorphic curves $\{C_k\}_{1\le k\le N}$ called
``levels'', and (not necessarily admissible) orbit sets $\Theta^{k+}$
and $\Theta^{k-}$ for each $k$, such that there exists
$k_0\in\{1,\ldots,N\}$ such that:
\begin{itemize}
\item
$\Theta^{k+}$ is an orbit set in $(Y_+,\lambda_+)$ for each $k\ge k_0$;
$\Theta^{k-}$ is an orbit set in $(Y_-,\lambda_-)$ for each $k\le k_0$; 
$\Theta^{N+}=\Theta^+$; $\Theta^{1-}=\Theta^-$; and
$\Theta^{k-}=\Theta^{k-1,+}$ for each $k>1$.
\item If $k>k_0$ then $C_k\in\mc{M}^{J_+}(\Theta^{k+},\Theta^{k-})$;
  if $k<k_0$ then $C_k\in\mc{M}^{J_-}(\Theta^{k+},\Theta^{k-})$; and
  $C_{k_0}\in\mc{M}^J(\Theta^{k_0,+},\Theta^{k_0,-})$.
\item
If $k\neq k_0$ then $C_k$ is not $\R$-invariant (as a current).
\end{itemize}
Let $\overline{\mc{M}^J(\Theta^+,\Theta^-)}$ denote the moduli space
of broken $J$-holomorphic curves from $\Theta^+$ to $\Theta^-$ as above.
\end{definition}

Note that $\mc{M}^J(\Theta^+,\Theta^-)$ is a subset of
$\overline{\mc{M}^J(\Theta^+,\Theta^-)}$ corresponding to broken
curves as above in which the number of levels $N=1$.  (It is perhaps a
misnomer to use the term ``broken'' when there is just one level.)

We would now like to relate the map \eqref{eqn:gencob} to broken
$J$-holomorphic curves in $\overline{X}$, where $J$ is cobordism-admissible.

\subsection{Statement of the main theorem}
\label{sec:SMT}

Let $(X,\lambda)$ be an exact symplectic cobordism from
$(Y_+,\lambda_+)$ to $(Y_-,\lambda_-)$, and assume that the contact
forms $\lambda_\pm$ are nondegenerate.  Fix a cobordism-admissible
almost complex structure $J$ on $\overline{X}$ which restricts to
symplectization-admissible almost complex structures $J_+$ on
$[0,\infty)\times Y_+$ and $J_-$ on $(-\infty,0]\times Y_-$, as in
Definition~\ref{def:cobadm}.  We now recall some definitions from \cite{cc}.

\paragraph{Product cylinders.}
If the cobordism $(X,\lambda)$ and the almost complex structure $J$ on
$\overline{X}$ are very special, then $X$ may contain regions that
look like pieces of a symplectization, in the following sense:

\begin{definition}
\label{def:PR}
  A {\em product region\/} in $X$ is the image of an embedding
  $[s_-,s_+]\times Z \to X$, where $s_-<s_+$ and $Z$ is an open
  3-manifold, such that:
\begin{itemize}
\item $\{s_\pm\}\times Z$ maps to $Y_\pm$, and $(s_-,s_+)\times Z$
  maps to the interior of $X$.
\item
The pullback of the Liouville form $\lambda$ on $X$ to $[s_-,s_+]\times Z$
has the form $e^s\lambda_0$, where $s$ denotes the $[s_-,s_+]$
coordinate, and $\lambda_0$ is a contact form on $Z$.
\item The pullback of the almost complex structure $J$ on $X$ to
  $[s_-,s_+]\times Z$ has the following properties: 
\begin{itemize}
\item
The restriction of $J$ to $\Ker(\lambda_0)$ is independent of $s$.
\item 
$J(\partial/\partial_s)=f(s)R_0$, where $f$ is a positive function of
$s$ and $R_0$ denotes the Reeb vector field for $\lambda_0$.
\end{itemize}
\end{itemize}
\end{definition}

Given a product region as above, the embedded Reeb orbits of
$\lambda_\pm$ in $\{s_\pm\}\times Z$ are identified with the embedded
Reeb orbits of $\lambda_0$ in $Z$.  If $\gamma$ is such a Reeb orbit,
then we can form a $J$-holomorphic cylinder in $\overline{X}$ by
taking the union of $[s_-,s_+]\times \gamma$ in $[s_-,s_+]\times Z$
with $(-\infty,0]\times \gamma$ in $(-\infty,0]\times Y_-$ and
$[0,\infty)\times\gamma$ in $[0,\infty)\times Y_+$.

\begin{definition}
\label{def:PC}
We call a $J$-holomorphic cylinder as above a {\em product
  cylinder.\/}
\end{definition}

\paragraph{Composition of cobordisms.}
If $(X^+,\lambda^+)$ is an exact symplectic cobordism from
$(Y_+,\lambda_+)$ to $(Y_0,\lambda_0)$, and if $(X^-,\lambda^-)$ is an
exact symplectic cobordism from $(Y_0,\lambda_0)$ to
$(Y_-,\lambda_-)$, then we can compose them to obtain an exact
symplectic cobordism $(X^-\circ X^+,\lambda)$ from $(Y_+,\lambda_+)$
to $(Y_-,\lambda_-)$.  Here $X^-\circ X^+$ is obtained by gluing $X^-$
and $X^+$ along $Y_0$ analogously to \eqref{eqn:completion}, and
$\lambda|_{X^\pm}=\lambda^\pm$.

\paragraph{Homotopy of cobordisms.}
Two exact symplectic cobordisms $(X,\omega_0)$ and $(X,\omega_1)$ from
$(Y_+,\lambda_+)$ to $(Y_-,\lambda_-)$ with the same underlying
four-manifold $X$ are {\em homotopic\/} if there is a smooth
one-parameter family of symplectic forms $\{\omega_t\mid t\in[0,1]\}$
on $X$ such that $(X,\omega_t)$ is an exact symplectic cobordism from
$(Y_+,\lambda_+)$ to $(Y_-,\lambda_-)$ for each $t\in[0,1]$.

\begin{theorem}
\label{thm:cob}
Let $(Y_+,\lambda_+)$ and $(Y_-,\lambda_-)$ be closed oriented
connected 3-manifolds with nondegenerate contact forms.  Let
$(X,\lambda)$ be an exact symplectic cobordism from $(Y_+,\lambda_+)$
to $(Y_-,\lambda_-)$.  Then there exist maps of ungraded
$\Z/2$-modules
\begin{equation}
\label{eqn:PhiL}
\Phi^L(X,\lambda): ECH_*^{L}(Y_+,\lambda_+) \longrightarrow
ECH_*^{L}(Y_-,\lambda_-)
\end{equation}
for each real number $L$, such that:
\begin{description}
\item{(Homotopy Invariance)} The map $\Phi^L(X,\lambda)$ depends only on $L$
  and the homotopy class of $(X,\omega)$.
\item{(Inclusion)} If $L<L'$ then the following diagram commutes:
\[
\begin{CD}
ECH_*^{L}(Y_+,\lambda_+) @>{\Phi^L(X,\lambda)}>> ECH_*^{L}(Y_-,\lambda_-) \\
@VV{\imath^{L,L'}}V @VV{\imath^{L,L'}}V \\
ECH_*^{L'}(Y_+,\lambda_+) @>{\Phi^{L'}(X,\lambda)}>> ECH_*^{L'}(Y_-,\lambda_-).
\end{CD}
\]
\item{(Direct Limit)}
\[
\lim_{L\to\infty}\Phi^L(X,\lambda) = \Phi(X): ECH_*(Y_+,\lambda_+) \longrightarrow
ECH_*(Y_-,\lambda_-),
\]
where $\Phi(X)$ is as in \eqref{eqn:gencob}.
\item{(Composition)}
If $(X,\lambda)$ is the composition of $(X^-,\lambda^-)$ and
$(X^+,\lambda^+)$ as above
with $\lambda_0$ nondegenerate, then
\[
\Phi^L(X^-\circ X^+,\lambda) = \Phi^L(X^-,\lambda^-) \circ
\Phi^L(X^+,\lambda^+).
\] 
\item{(Scaling)} If $c$ is a positive constant then the following
  diagram commutes:
\[
\begin{CD}
ECH_*^{L}(Y_+,\lambda_+) @>{\Phi^L(X,\lambda)}>> ECH_*^{L}(Y_-,\lambda_-) \\
@V{s}V{\simeq}V @V{s}V{\simeq}V \\
ECH_*^{cL}(Y_+,c\lambda_+) @>{\Phi^{cL}(X,c\lambda)}>> ECH_*^{cL}(Y_-,c\lambda_-).
\end{CD}
\]
\item{(Holomorphic Curves)} Let $J$ be a cobordism-admissible almost
  complex structure on $\overline{X}$ such that $J_+$ and $J_-$ are
  $ECH^L$-generic.  Then there exists a (noncanonical) chain map
\[
\hat{\Phi} : ECC_*^L(Y_+,\lambda_+,J_+) \longrightarrow
ECC_*^L(Y_-,\lambda_-,J_-)
\]
inducing $\Phi^L(X,\lambda)$, such that if $\Theta^+$ and $\Theta^-$
are admissible orbit sets for $(Y_+,\lambda_+)$ and $(Y_-,\lambda_-)$
respectively with action less than $L$, then:
\begin{description}
\item{(i)} If there are no broken $J$-holomorphic curves in
  $\overline{X}$ from $\Theta^+$ to $\Theta^-$, then
  $\langle \hat{\Phi}\Theta^+,\Theta^-\rangle=0$.
\item{(ii)} If the only broken $J$-holomorphic curve in $\overline{X}$
  from $\Theta^+$ to $\Theta^-$ is a union of covers of product
  cylinders, then $\langle \hat{\Phi}\Theta^+,\Theta^-\rangle=1$.
\end{description}
\end{description}
\end{theorem}

Our proof of Theorem~\ref{thm:cob} uses Seiberg-Witten theory. While
it would be natural to try to define the maps $\Phi^L(X,\lambda)$ more
directly by counting (broken) holomorphic curves in $\overline{X}$
with ECH index $0$, there are substantial technical difficulties with this
approach; see the discussion in \cite[\S5.5]{bn}.

\begin{remark}
  The maps $\Phi^L(X,\lambda)$ respect the decomposition
  \eqref{eqn:gammadecomp} in the following sense: The image of
  $ECH_*(Y_+,\lambda_+,\Gamma_+)$ has a nonzero component in
  $ECH_*(Y_-,\lambda_-,\Gamma_-)$ only if $\Gamma_+\in H_1(Y_+)$ and
  $\Gamma_-\in H_1(Y_-)$ map to the same class in $H_1(X)$.  This
  follows from part (i) of the Holomorphic Curves axiom (or more
  simply by keeping track of the spin-c structures in the construction
  of $\Phi^L(X,\lambda)$).
\end{remark}

\begin{remark}
  Part (ii) of the Holomorphic Curves axiom includes the case where
  $\Theta^\pm$ and the product region are empty, in which case there
  is a unique holomorphic curve, namely the empty set.  It then
  follows that $\Phi^L(X,\lambda)$ sends the ECH contact invariant for
  $(Y_+,\lambda_+)$ (the class in ECH represented by the empty set of
  Reeb orbits) to the ECH contact invariant for $(Y_-,\lambda_-)$.
\end{remark}

\begin{remark}
  If we allow $Y_+$ and $Y_-$ to be disconnected, then all of
  Theorem~\ref{thm:cob} except for the Direct Limit axiom still holds,
  by a straightforward modification of the proof. (The statement of
  the Direct Limit axiom does not make sense in this case because the
  relevant Seiberg-Witten Floer cohomology needed to define the map
  $\Phi(X)$ has not been defined for disconnected three-manifolds).
\end{remark}

\begin{remark}
  We expect that Theorem~\ref{thm:cob} also holds with $\Z$ coefficients. Note that the
  cobordism maps on Seiberg-Witten Floer homology defined in \cite{km}
  depend on a choice of ``homology orientation'' of the cobordism.
  However we expect to be able to define cobordism maps on ECH without
  choosing a homology orientation, as this works in those cases where
  ECH cobordism maps can be defined by counting holomorphic curves,
  cf.\ \cite[Lem.\ A.14]{algebraic}. The Direct Limit axiom should then hold for a suitable homology orientation.
\end{remark}

\paragraph{Acknowledgments.}
The first author was partially supported by NSF grant DMS-0806037.
The second author was partially supported by the Clay Mathematics
Insitute, the Mathematical Sciences Research Institute, and the NSF.
Both authors thank MSRI for its hospitality.

\section{Seiberg-Witten Floer cohomology and contact forms}
\label{sec:swfcontact}

We now review how to define Seiberg-Witten Floer cohomology, with the
Seiberg-Witten equations perturbed by a contact form.

\subsection{Seiberg-Witten Floer cohomology}
\label{sec:swfc}

We begin by briefly reviewing the relevant parts of the definition of
Seiberg-Witten Floer cohomology.  We follow the conventions in the
book by Kronheimer-Mrowka \cite{km}, which explains the full details
of this theory.

Let $Y$ be a closed oriented (connected) 3-manifold, and let $g$ be a
Riemannian metric on $Y$.  A {\em spin-c structure\/} on $Y$ consists
of a rank $2$ Hermitian vector bundle $\Sp$ over $Y$, whose sections
are called {\em spinors\/}, together with a bundle map
$cl:TY\to\End(\Sp)$, called {\em Clifford multiplication\/}, such that
\begin{equation}
\label{eqn:clifford}
cl(a)cl(b)+cl(b)cl(a)=-2\langle a,b\rangle
\end{equation}
for $a,b\in T_yY$, and
\[
cl(e_1)cl(e_2)cl(e_3)=1
\]
when $(e_1,e_2,e_3)$ is an oriented orthonormal basis for $T_yY$.  We
denote the spin-c structure by $\frak{s}=(\Sp,cl)$.  Two spin-c
structures $(\Sp,cl)$ and $(\Sp',cl')$ are considered equivalent if
there is a bundle isomorphism $\phi:\Sp\stackrel{\simeq}{\to}\Sp'$
respecting the Clifford multiplications, meaning that
$cl'(v)\phi(\psi)=\phi(cl(v)\psi)$ for $v\in T_yY$ and $\psi\in\Sp_y$.
The set of spin-c structures is then an affine space over $H^2(Y;\Z)$.
The definition of the action is that if $e\in H^2(Y;\Z)$, then
\begin{equation}
\label{eqn:affine}
(\Sp,cl) + e \eqdef (\Sp\tensor L_e,cl\tensor 1),
\end{equation}
where $L_e$ denotes the complex line bundle with
$c_1(L_e)=e$.  If $\frak{s}=(\Sp,cl)$ is a spin-c structure,
we define $c_1(\frak{s})\eqdef c_1(\Sp)\in H^2(Y;\Z)$.  

A spin-c structure is also equivalent to a lift of the frame bundle of
$TY$ from a principal $SO(3)$-bundle to a principal $U(2)$-bundle. The
set of spin-c structures on $Y$ does not depend on the metric $g$.

\begin{example}
\label{ex:pf}
An oriented 2-plane field $\xi$ on $Y$ determines a spin-c structure
$\frak{s}_\xi$ as follows.  The spinor bundle is given by
\[
\Sp = \underline{\C}\oplus \xi,
\]
where $\underline{\C}$ denotes the trivial complex line bundle on $Y$,
and $\xi$ is regarded as a Hermitian line bundle using its orientation
and the metric on $Y$.  Clifford multiplication is defined as follows:
if $(e_1,e_2,e_3)$ is an oriented orthonormal basis for $T_yY$ such
that $(e_2,e_3)$ is an oriented orthonormal basis for $\xi_y$, then in
terms of the basis $(1,e_2)$ for $\Sp$,
\[
cl(e_1) = \begin{pmatrix}i & 0 \\ 0 & -i\end{pmatrix}, \quad\quad
cl(e_2) = \begin{pmatrix}0 & -1 \\ 1 & 0 \end{pmatrix}, \quad\quad
cl(e_3) = \begin{pmatrix}0 & i \\ i & 0\end{pmatrix}.
\]
\end{example}

Now fix a spin-c structure $(\Sp,cl)$.  A {\em spin-c
  connection\/} is a connection $\A_\Sp$ on $\Sp$ which is compatible with
Clifford multiplication in the following sense: If $v$ is a section of
$TY$ and $\psi$ is a spinor, then
\begin{equation}
\label{eqn:spincc}
\nabla_{\A_\Sp}(cl(v)\psi) = cl(\nabla v)\psi + cl(v)\nabla_{\A_\Sp}\psi,
\end{equation}
where $\nabla v$ denotes the covariant derivative of $v$ with respect
to the Levi-Civita connection. A spin-c connection $\A_\Sp$ is
equivalent to a (Hermitian) connection $\A$ on the determinant line
bundle $\det(\Sp)$.  Note that adding an imaginary-valued $1$-form $a$
to $\A$ has the effect of adding $a/2$ to $\A_\Sp$.  A spin-c connection
$\A_\Sp$, identified in this way with a connection $\A$ on $\det(\Sp)$,
determines a {\em Dirac operator\/} $D_\A$, which is defined to be the
composition
\begin{equation}
\label{eqn:dirac}
C^\infty(Y;\Sp) \stackrel{\nabla_{\A_\Sp}}{\longrightarrow}
C^\infty(Y;T^*Y\tensor\Sp) \stackrel{cl}{\longrightarrow}
C^\infty(Y;\Sp).
\end{equation}
Here Clifford multiplication is extended to cotangent vectors by using
the metric on $Y$ to identify $T^*Y$ with $TY$.

Now fix an exact 2-form $\eta$ on $Y$.  The {\em Seiberg-Witten
  equations\/} with perturbation $\eta$ concern a pair $(\A,\Psi)$, where
$\A$ is a connection on $\det(\Sp)$ and $\Psi$ is a spinor.  The
equations are
\begin{equation}
\label{eqn:sw3}
\begin{split}
D_\A\Psi &= 0,\\
*F_\A &= \tau(\Psi) + i{*}\eta.
\end{split}
\end{equation}
Here $*$ denotes the Hodge star, $F_\A$ denotes the
curvature of $\A$, and $\tau:\Sp\to iT^*Y$ is a quadratic
bundle map defined by
\[
\tau(\Psi)(v) = \langle cl(v)\Psi,\Psi\rangle
\]
for $\Psi\in\Sp_y$ and $v\in T_yY$.  A pair $(\A,\Psi)$ solves the
Seiberg-Witten equations \eqref{eqn:sw3} if and only if it is a
critical point of the functional $\frak{a}_{\eta}$ on the set of pairs
$(\A,\Psi)$ defined by
\begin{equation}
\label{eqn:aeta}
\frak{a}_{\eta}(\A,\Psi) \eqdef -\frac{1}{8}\int_Y(\A-\A_0)\wedge
(F_\A+F_{\A_0}-2i\eta) + \frac{1}{2}\int_Y\langle D_\A\Psi,\Psi\rangle.
\end{equation}
Here $\A_0$ is any reference connection on $\det(\mathbb{S})$;
changing this reference connection will add a constant to the
functional \eqref{eqn:aeta}.

The {\em gauge group\/} $\mc{G}\eqdef C^\infty(Y;S^1)$ acts on the set
of pairs $(\A,\Psi)$ by
\begin{equation}
\label{eqn:gauge}
u\cdot(\A,\Psi) \eqdef
(\A-2u^{-1}du, u\Psi),
\end{equation}
and this action preserves the set of solutions to the Seiberg-Witten equations.
Two solutions are considered equivalent if one is obtained from the
other by the action of the gauge group. A solution $(\A,\Psi)$ is
called {\em reducible\/} if $\Psi\equiv 0$, and {\em irreducible\/}
otherwise. If the exact 2-form $\eta$ is suitably generic, then there
are only finitely many irreducible solutions to \eqref{eqn:sw3} (modulo
gauge equivalence), each of which is cut out transversely in an
appropriate sense. Fix such a 2-form $\eta$.

The chain complex for defining Seiberg-Witten Floer cohomology (with
$\Z/2$ coefficients) can be decomposed into submodules (not
subcomplexes)
\[
\widehat{CM}^* =
\widehat{CM}^*_{irr} \oplus \widehat{CM}^*_{red}.
\]
Here $\widehat{CM}^*_{irr}$ is the free $\Z/2$-module generated by the
irreducible solutions, while $\widehat{CM}^*_{red}$ is a more
complicated term arising from the reducibles.  The details of the
reducible part $\widehat{CM}^*_{red}$ do not concern us here, because
soon we will be restricting attention to a certain subcomplex of
$\widehat{CM}^*$, for a particular perturbation $\eta$, which lives
entirely within $\widehat{CM}^*_{irr}$ as explained in the proof of
Lemma~\ref{lem:diff} below.

For the same reason, our primarily interest is in the part of the
chain complex differential that maps $\widehat{CM}^*_{irr}$ to itself.
To describe this, let $(\A_+,\Psi_+)$ and $(\A_-,\Psi_-)$ be two
solutions to the equations \eqref{eqn:sw3}. An {\em instanton\/} from
$(\A_-,\Psi_-)$ to $(\A_+,\Psi_+)$ is a smooth one-parameter family of
pairs $(\A(s),\Psi(s))$ parametrized by $s\in\R$, where $\A(s)$ is a
connection on $\det(\Sp)$ and $\Psi(s)$ is a spinor, satisfying the
equations
\begin{equation} \label{eqn:instanton} \begin{split}
    \frac{\partial}{\partial s}\Psi(s) &= - D_{\A(s)}\Psi(s),\\
    \frac{\partial}{\partial s}\A(s) &= - {*}F_{\A(s)} + \tau(\Psi(s)) +
    i{*}\eta,\\
\lim_{s\to\pm\infty}(\A(s),\Psi(s)) &= (\A_\pm,\Psi_\pm).
\end{split} \end{equation} A solution to these equations is a downward
gradient flow line of the functional \eqref{eqn:aeta} from
$(\A_-,\psi_-)$ to $(\A_+,\psi_+)$.  Here the metric on the space of
pairs $(\A,\Psi)$ is induced by the Hermitian inner product on $\Sp$
together with $1/4$ of the $L^2$ inner product on $\Omega^1(Y;i\R)$.
The gauge group $C^\infty(Y;S^1)$ again acts on the space of such
instantons.  Also $\R$ acts on the space of instantons by translating
the $s$ coordinate.  If $(\A_\pm,\Psi_\pm)$ are irreducible, then the
coefficient of $(\A_-,\Psi_-)$ in the differential of $(\A_+,\Psi_+)$
counts index $1$ instantons from $(\A_-,\Psi_-)$ to $(\A_+,\Psi_+)$,
modulo gauge equivalence and translation of $s$.  Here the ``index''
is the local expected dimension of the moduli space of instantons
modulo gauge equivalence.  The index defines a relative
$\Z/d(c_1(\frak{s}))$-grading on the chain complex, such that the
differential increases the grading by $1$.

All we need to know about the rest of the differential is that if
$(\A_+,\Psi_+)$ is irreducible, and if there is no index one instanton
to $(\A_+,\Psi_+)$ from a reducible solution $(\A_-,\Psi_-)$, then the
differential sends $(\A_+,\Psi_+)$ to an element of
$\widehat{CM}^*_{irr}$.  Here when $(\A_-,\Psi_-)$ is reducible, the index
is defined to be the local expected dimension of the moduli space of
instantons modulo gauge equivalence that have the same asymptotic
decay rate as $s\to-\infty$.

In general, to obtain transversality of the moduli spaces of
instantons as needed to define the differential, some ``abstract''
perturbations of equations \eqref{eqn:sw3} and \eqref{eqn:instanton}
are required.  These are described in \cite[Ch.\ 11]{km}, where a
Banach space $\mc{P}$ of such perturbations is defined.  Below, an
{\em abstract\/} perturbation is one from $\mc{P}$, a {\em small\/}
abstract perturbation is one with small $\mc{P}$-norm, and a {\em
  generic\/} abstract perturbation is one from a residual subset
(depending on context) of $\mc{P}$.  As noted previously, if the exact
$2$-form $\eta$ in \eqref{eqn:sw3} is suitably generic, then there are
only finitely many irreducible solutions to \eqref{eqn:sw3}, and these
are all cut out transversely.  For such a choice of $\eta$, the generic
abstract perturbation needed to define the differential can be chosen
to vanish to any given order on the irreducible solutions to
\eqref{eqn:sw3}, and in particular so that the generators of
$\widehat{CM}^*_{irr}$ are unchanged, i.e.\ every solution to the
perturbed version of \eqref{eqn:sw3} is a solution to the unperturbed
version and vice-versa, see \cite[\S3h, Part 5]{e1}.  When $\eta$ is
generic in this sense, we always assume that the abstract
perturbations needed to define the differential (and also the
cobordism maps reviewed in \S\ref{sec:smoothcob}) are chosen this way.
The abstract perturbations then have little conceptual role in the
arguments below, see Proposition~\ref{prop:Liso}(c) and also
\cite[Thm.\ 4.4]{e1}, so we usually suppress them from the notation.

We denote the homology of this chain complex by
$\widehat{HM}^*(Y,\frak{s};g, \eta)$. The homologies for different
choices of $(g,\eta)$ (and abstract perturbations) are canonically
isomorphic to each other.  The isomorphisms between the homologies for
different choices are a special case of the cobordism maps reviewed in
\S\ref{sec:smoothcob}. Thus the homologies for different choices are
canonically isomorphic to a single $\Z/2$-module, which is denoted by
$\widehat{HM}^*(Y,\frak{s})$.

\subsection{Perturbing the equations using a contact form}
\label{sec:pcf}

Now suppose $\lambda$ is a contact form on $Y$.  Choose an almost
complex structure $J$ on the contact planes $\xi$ as needed to
define a symplectization-admissible almost complex structure on $Y$,
see \S\ref{sec:ech}.  The choice of $\lambda$ and $J$ determine a
metric $g$ on $Y$ such that Reeb vector field $R$ has length $1$ and
is orthogonal to the contact plane field $\xi$, on which the metric is
given by
\begin{equation}
\label{eqn:1/2factor}
g(v,w)=\frac{1}{2}d\lambda(v,Jw).
\end{equation}
In this metric one has
\begin{equation}
\label{eqn:2factor}
|\lambda|=1, \quad\quad d\lambda=2{*}\lambda.
\end{equation}

\begin{remark}
\label{rmk:factor2}
The factor of $1/2$ in \eqref{eqn:1/2factor} and the factor of $2$ in
\eqref{eqn:2factor} could probably be dropped, but we have included
these factors for consistency with the papers \cite{tw1,e1} and their
sequels.
\end{remark}

With these choices made, if $\frak{s}=(\Sp,cl)$ is any spin-c
structure, then there is a canonical decomposition
\begin{equation}
\label{eqn:SEKE}
\Sp = E \oplus K^{-1}E
\end{equation}
into eigenbundles of $cl(\lambda)$, where $E$ is the $+i$ eigenbundle,
and $K^{-1}$ denotes the contact structure $\xi$, regarded as a
Hermitian line bundle via $J$. When $E$ is the trivial line bundle
$\underline{\C}$, one recovers Example~\ref{ex:pf}. In this case it
turns out that there is a distinguished connection $A_{K^{-1}}$ on
$K^{-1}$ such that $D_{A_{K^{-1}}}(1,0)=0$. In the general case, a
connection $\A$ on $\det(\Sp)=K^{-1}E^2$ can be written as
\begin{equation}
\label{eqn:connections}
\A=A_{K^{-1}}+2A
\end{equation}
where $A$ is a connection on $E$. Using \eqref{eqn:connections}, we
henceforth identify a spin-c connection with a Hermitian connection
$A$ on $E$ (instead of with a connection $\A$ on $\det(\Sp)$ as in
\S\ref{sec:swfc}), and denote its corresponding Dirac operator by $D_A$.

As in \cite{e1}, given a spin-c structure $\frak{s}$ as above, we
consider, for a connection $A$ on $E$ and a section $\psi$ of $\Sp$, the following version of the Seiberg-Witten equations:
\begin{equation}
\label{eqn:tsw3}
\begin{split}
  D_A\psi &= 0,\\
*F_A &= r(\tau(\psi)-i\lambda) - \frac{1}{2}{*}F_{A_{K^{-1}}} + i{*}\mu.
\end{split}
\end{equation}
Here $r$ is a positive real number (which below we will take to be
very large), and $\mu$ is an exact $2$-form satisfying certain
conditions described in the next paragraph.  Under the identification
\eqref{eqn:connections}, the equations \eqref{eqn:tsw3} are equivalent
to the Seiberg-Witten equations \eqref{eqn:sw3} with perturbation
\begin{equation}
\label{eqn:perturbation}
\eta=-rd\lambda+ 2\mu,
\end{equation}
if we rescale the spinor by
\begin{equation}
\label{eqn:rescaling}
\Psi=\sqrt{2r}\psi.
\end{equation}

The $2$-form $\mu$ above is a suitably generic exact smooth $2$-form
from a certain Banach space $\Omega$ of such forms defined in
\cite{e1}.  The precise details of $\Omega$ are not relevant here; we
just need to recall the following facts.  First, $\Omega$ is dense in
the space of smooth exact $2$-forms.  Also $\Omega$ is a Banach
subspace of $\mc{P}$, meaning that if $\mu\in\Omega$, then the
equations \eqref{eqn:tsw3} without futher perturbation, together with
the corresponding version of \eqref{eqn:instanton}, namely
\eqref{eqn:Tinstanton} below, consitute one of the ``abstract
perturbations'' from $\mc{P}$.  The $\mc{P}$-norm of an element of
$\Omega$ controls its derivatives to all orders.  We always assume
that the form $\mu$ in \eqref{eqn:perturbation} has $\mc{P}$-norm less
than $1$ and $C^0$ norm less than $1/100$.  Also the space $\Omega$
can be chosen so as to contain $d\lambda$, and this is assumed
below\footnote{The fact that $d\lambda\in\Omega$ will be used in the
  proof of Lemma~\ref{lem:cobinv}.}.  Finally, the spaces $\Omega$ and
$\mc{P}$ depend on the metric, and thus on the pair $(\lambda,J)$.
However $\Omega$ and $\mc{P}$ can be chosen so as to define smooth
Banach space bundles over the space of metrics.  We do not indicate
this dependence of $\Omega$ and $\mc{P}$ on the metric in the notation
below.

The version of the Seiberg-Witten Floer chain complex obtained from
solutions to \eqref{eqn:tsw3} for a given data set $(\lambda,J,r,\mu)$
and abstract perturbation from $\mc{P}$ (if necessary to obtain
suitable transversality) is denoted below by
$\widehat{CM}^*(Y,\frak{s};\lambda,J,r)$.  Here we are suppressing
$\mu$ and the abstract perturbation from the notation.  The
corresponding Seiberg-Witten Floer cohomology is denoted by
$\widehat{HM}^*(Y,\frak{s};\lambda,J,r)$.  The irreducible component
$\widehat{CM}^*_{irr}$ of the chain complex is generated by
irreducible solutions to \eqref{eqn:tsw3}.  If
$(A_+,\psi_+)$ and $(A_-,\psi_-)$ are two such irreducible solutions,
then the componenent of $(A_-,\psi_-)$ in the differential of
$(A_+,\psi_+)$ counts index $1$ solutions to a correspondingly
perturbed version of the equations
\begin{equation} \label{eqn:Tinstanton} \begin{split}
    \frac{\partial}{\partial s}\psi(s) &= - D_{A(s)}\psi(s),\\
    \frac{\partial}{\partial s}A(s) &= - {*}F_{A(s)} +
    r(\tau(\psi(s))-i\lambda) - \frac{1}{2}{*}F_{A_{K^{-1}}} +
    i{*}\mu,\\
\lim_{s\to\pm\infty}(A(s),\psi(s)) &= (A_\pm,\psi_\pm),
\end{split} \end{equation}
modulo gauge equivalence and $s$-translation.

\subsection{The energy filtration}
\label{sec:energyfiltration}

When $r$ above is large, the chain complex $\widehat{CM}^*$ has (up to
some level) a filtration analogous to the symplectic action filtration
on ECH.  This works as follows.  If $(A,\psi)$ is a solution to
\eqref{eqn:tsw3}, define the {\em energy\/}
\begin{equation}
\label{eqn:energy}
\energy(A) \eqdef i\int_Y \lambda\wedge F_A.
\end{equation}

The idea is that given an ECH generator $\Theta$, if $r$ is
sufficiently large then there is a corresponding irreducible solution
$(A,\psi)$ to \eqref{eqn:tsw3} for which the zero set of the $E$ component of $\psi$ (see equation \eqref{eqn:SEKE}) is close to the Reeb
orbits in $\Theta$, the curvature $F_A$ is concentrated in a radius $O(r^{-1/2})$
neighborhood of the Reeb orbits in $\Theta$, and the energy
$\energy(A)$ is approximately $2\pi\mc{A}(\Theta)$.

This motivates defining the
following analogue of the filtered ECH chain complex $ECC_*^L$ from
\S\ref{sec:filtered}: If $L$ is a real number, define
$\widehat{CM}^*_L$ to be the submodule of $\widehat{CM}_{irr}^*$
generated by irreducible solutions $(A,\psi)$ to \eqref{eqn:tsw3} with
$\energy(A)<2\pi L$.

\begin{lemma}
\label{lem:diff}
Fix $Y,\lambda,J$ as above and $L\in\R$.  Suppose that $\lambda$ has
no orbit set of action exactly $L$. Fix $r$ sufficiently large, and a
$2$-form $\mu\in\Omega$ with $\mc{P}$-norm less than $1$ so that all
irreducible solutions to \eqref{eqn:tsw3} are cut out transversely.
Then for every $\frak{s}$ and for every sufficiently small generic
abstract perturbation:
\begin{description}
\item{(a)}
$\widehat{CM}_L^*(Y,\frak{s};\lambda,J,r)$ is a subcomplex of
$\widehat{CM}^*(Y,\frak{s};\lambda,J,r)$.
\item{(b)} If $L'<L$ and if there is no orbit set with action in the
  interval $[L',L]$, then the inclusion
\[
\widehat{CM}_{L'}^*(Y,\frak{s};\lambda,J,r) \longrightarrow
\widehat{CM}_L^*(Y,\frak{s};\lambda,J,r)
\]
is an isomorphism of chain complexes.
\end{description}
\end{lemma}

\begin{proof}
First observe that if $r$ is sufficiently large then all solutions
  $(A,\psi)$ to the perturbed Seiberg-Witten equations \eqref{eqn:tsw3}
  with $\energy(A)<2\pi L$ are irreducible, because it follows from
  \eqref{eqn:tsw3} that the energy of a reducible solution $(A,0)$ to
  \eqref{eqn:tsw3} is a linear, increasing function of $r$.

  Now if we fix the spin-c structure $\frak{s}$, then part (a) of the
  lemma follows from the first bullet in \cite[Thm.\ 4.4]{e1}, and
  part (b) of the lemma follows from \cite[Prop.\ 4.12]{e1}.
  Lemma~\ref{lem:diff} (without the spin-c structure fixed) then
  follows from Lemma~\ref{lem:diff2} below.
\end{proof}

\begin{lemma}
\label{lem:diff2}
Given a real number $\mc{E}$, there exists an integer $\kappa$ such that if $r$
is sufficiently large, then at most $\kappa$ spin-c structures admit
solutions $(A,\psi)$ to \eqref{eqn:tsw3} with $\energy(A)<\mc{E}$.
\end{lemma}

\begin{proof}
  Let $(A,\psi)$ be a solution to \eqref{eqn:tsw3} for some spin-c structure $\frak{s}$. It follows from the curvature equation in \eqref{eqn:tsw3} and the a
priori estimates on $\psi$ in \cite[Lem.\ 2.3]{e4} that if $r$ is
sufficiently large, then the $L^1$ norm of $F_A$ over
$Y$ is bounded by $c_0+c_1\energy(A)$, where $c_0$ and $c_1$ are independent of
$r$ and $\frak{s}$.  This implies the lemma,
because a bound on the $L^1$ norm of $F_A$ gives an upper bound on
the absolute values of the pairings of $c_1(\frak{s})$ with a set of
generators for $H_2(Y)$.
\end{proof}

When Lemma~\ref{lem:diff} is applicable, we denote the homology of the
subcomplex $\widehat{CM}_L^*(Y,\frak{s};\lambda,J,r)$ of
$\widehat{CM}^*(Y,\frak{s};\lambda,J,r)$ by
$\widehat{HM}^*_L(Y,\frak{s};\lambda,J,r)$.  If $r$ is larger than
some $(\lambda,J)$-dependent constant, then this homology does not
depend on the $2$-form $\mu$, the small abstract perturbation, or $r$.
This follows from \cite[Lem.\ 4.6]{e1}, and a generalization is proved
in Lemma~\ref{lem:deform} below.  We always assume that $r$ is
sufficiently large as above so that
$\widehat{HM}^*_L(Y,\frak{s};\lambda,J,r)$ is well-defined and
independent of $r$, although we keep $r$ in the notation.  We will see
in \S\ref{sec:iso} that this homology is isomorphic to
$ECH^L_{-*}(Y,\lambda,\op{PD}(\frak{s}-\frak{s}_\xi);J)$.

\section{SWF cohomology and (filtered) ECH}
\label{sec:iso}

We now explain the relation between filtered ECH and the filtered
version of Seiberg-Witten Floer cohomology defined in
\S\ref{sec:energyfiltration}.  Along the way we review the
construction of the isomorphism \eqref{eqn:echswfJ} between ECH and
$\widehat{HM}^*$ and prove Theorem~\ref{thm:FECH} on the
$J$-independence of filtered ECH.

\subsection{L-flat approximations}

In order to define a chain map from the filtered ECH chain complex
to the Seiberg-Witten Floer chain complex, it is convenient to modify
the pair $(\lambda,J)$ so that it has a certain nice form in a tubular
neighborhood of each Reeb orbit of symplectic action less than $L$.
Specifically, we say that the pair $(\lambda,J)$ is {\em $L$-flat\/}
if near each Reeb orbit of length less than $L$ it satisfies the
conditions in \cite[Eq.\ (4.1)]{e1}.  (We do not need to recall these
conditions in detail here.)  The reasons for introducing this
condition are discussed in \cite[\S5c, Part 2]{e1}.  In particular, we
have the following key fact:

\begin{proposition}
\label{prop:Liso}
Fix $Y,\lambda,J$ and $L\in\R$.  Suppose that $\lambda$ is
$L$-nondegenerate, see Definition~\ref{def:Lnondeg}.  Then for all $r$
sufficiently large, and for all $\Gamma\in H_1(Y)$, the following
hold:
\begin{description}
\item{(a)} There is a canonical map from the set of generators of
  $\widehat{CM}^{*}_L(Y,\frak{s}_{\xi,\Gamma};\lambda,J,r)$
  to the set of orbit sets in the homology class $\Gamma$ of length
  less than $L$.
\item{(b)} If $(\lambda,J)$ is $L$-flat, then the generators of
  $\widehat{CM}^{*}_L(Y,\frak{s}_{\xi,\Gamma};\lambda,J,r)$ are cut
  out transversely, and the map in part (a) is a bijection from the
  set of these generators to the set of {\em admissible\/} orbit sets
  in the homology class $\Gamma$ of length less than $L$.
\item{(c)} Suppose $(\lambda,J)$ is $L$-flat and $J$ is
  $ECH^L$-generic.  Fix a $2$-form $\mu$ from $\Omega$ with
  $\mc{P}$-norm less than $1$, and fix a small generic abstract
  perturbation. Then the bijection in part (b) induces an isomorphism
  of relatively graded chain complexes
\begin{equation}
\label{eqn:Liso}
ECC_*^L(Y,\lambda,\Gamma;J) \stackrel{\simeq}{\longrightarrow}
\widehat{CM}_L^{-*}(Y,\frak{s}_{\xi,\Gamma};\lambda,J,r).
\end{equation}
\end{description}
\end{proposition}

\begin{proof}
  Part (a) follows directly from \cite[\S6]{tw1}.  Part (b) follows
  from \cite[Thm.\ 4.2]{e1}, and part (c) follows from the second
  bullet of \cite[Thm.\ 4.4]{e1}.
\end{proof}

The basic picture for part (a) is that when $r$ is large, generators
$(A,\psi)$ of $\widehat{CM}_L^*$ have $F_A$ concentrated near a
collection of Reeb orbits as described preceding Lemma~\ref{lem:diff},
and this defines the map from generators of $\widehat{CM}_L^*$ to
orbit sets with symplectic action less than $L$.  The idea for part
(c) is then that the instantons that define the differential on the
right hand side of \eqref{eqn:Liso} correspond in a similar manner to
the holomorphic curves that define the differential on the left hand
side of \eqref{eqn:Liso}.

To make use of the above proposition, we need to suitably approximate an arbitrary pair $(\lambda,J)$ by an $L$-flat pair.

\begin{definition}
\label{def:lfa}
Let $\lambda$ be an $L$-nondegenerate contact form, and let $J$ be an
$ECH^L$-generic symplectization-admissible almost complex structure
for $\lambda$.  An {\em $L$-flat approximation\/} to $(\lambda,J)$ is
an $L$-flat pair $(\lambda_1,J_1)$ which is the endpoint of a smooth
homotopy $\{(\lambda_t,J_t)\mid t\in[0,1]\}$ with
$(\lambda_0,J_0)=(\lambda,J)$ such that:
\begin{description}
\item{(i)} For each $t\in[0,1]$, $\lambda_t$ is an $L$-nondegenerate
  contact form, and $J_t$ is an $ECH^L$-generic
  symplectization-admissible almost complex structure for $\lambda_t$.
\item{(ii)} The Reeb orbits of $\lambda_t$ with length less than $L$, and
  their lengths, do not depend on $t$.
\end{description}
\end{definition}

We will see in Lemma~\ref{lem:plfa} below that $L$-flat approximations
always exist.  Note that if $\{(\lambda_t,J_t)\}$ is a homotopy as in
Definition~\ref{def:lfa}, then by condition (i) there is a canonical
isomorphism of chain complexes
\begin{equation}
\label{eqn:clf}
ECC_*^L(Y,\lambda,\Gamma;J) \stackrel{\simeq}{\longrightarrow}
ECC_*^L(Y,\lambda_1,\Gamma;J_1),
\end{equation}
induced by the canonical identification of generators from condition
(ii).  Combining this with the isomorphism \eqref{eqn:Liso} for
$(\lambda_1,J_1)$, we conclude that if $(\lambda_1,J_1)$ is an
$L$-flat approximation to $(\lambda,J)$, and if $r$ is sufficiently
large, then there is a canonical isomorphism of chain complexes
\begin{equation}
\label{eqn:Liso2}
ECC_*^L(Y,\lambda,\Gamma;J) \stackrel{\simeq}{\longrightarrow}
\widehat{CM}_L^{-*}(Y,\frak{s}_{\xi,\Gamma};\lambda_1,J_1,r).
\end{equation}

\subsection{Deforming $\lambda$ and $J$}

We now state and prove a key lemma regarding the behavior of
$\widehat{HM}^*_L$ under certain special deformations of $\lambda$ and
$J$.

\begin{definition}
\label{def:ad}
  An {\em admissible deformation\/} is a smooth $1$-parameter family
  $\rho=\{(\lambda_t,L_t,J_t,r_t)\mid t\in[0,1]\}$ such that for all
  $t\in[0,1]$:
\begin{itemize}
\item
$\lambda_t$ is an $L_t$-nondegenerate contact form on $Y$.
\item
$J_t$ is a symplectization-admissible almost complex structure for
$\lambda_t$.
\item $r_t$ is a positive real number.
\end{itemize}
\end{definition}

The following is a slight generalization of \cite[Lemmas 4.6 and 4.16]{e1}.

\begin{lemma}
\label{lem:deform}
Let $\rho=\{(\lambda_t,L_t,J_t,r_t)\mid t\in[0,1]\}$ be an admissible
deformation.  If the real numbers $\{r_t\}$ are sufficiently large,
then for any $\frak{s}\in\Spinc(Y)$, the admissible deformation $\rho$
induces an isomorphism
\begin{equation}
\label{eqn:Phirho}
\Phi_\rho:\widehat{HM}^*_{L_0}(Y,\frak{s};\lambda_0,J_0,r_0)
\stackrel{\simeq}{\longrightarrow}
\widehat{HM}^*_{L_1}(Y,\frak{s};\lambda_1,J_1,r_1)
\end{equation}
with the following properties:
\begin{description}
\item{(a)}
$\Phi_\rho$ is invariant under homotopy of admissible deformations.
\item{(b)}
If $\rho_1$ and $\rho_2$ are composable admissible deformations, then
$\Phi_{\rho_1\circ\rho_2}=\Phi_{\rho_1}\circ\Phi_{\rho_2}$.
\item{(c)}
The diagram
\begin{equation}
\label{eqn:mwh}
\begin{CD}
\widehat{HM}^*_{L_0}(Y,\frak{s};\lambda_0,J_0,r) @>{\Phi_\rho}>>
\widehat{HM}^*_{L_1}(Y,\frak{s};\lambda_1,J_1,r) \\
@VVV @VVV\\
\widehat{HM}^*(Y,\frak{s};\lambda_0,J_0,r) @>>>
\widehat{HM}^*(Y,\frak{s};\lambda_1,J_1,r) \\
\end{CD}
\end{equation}
commutes, where the vertical arrows are induced by the inclusions of
chain complexes, and the bottom arrow is the canonical isomorphism
on Seiberg-Witten Floer cohomology.
\item{(d)} If for all $t\in[0,1]$, the pair $(\lambda_t,J_t)$ is
  $L_t$-flat and $J_t$ is $ECH^{L_t}$-generic, then under the isomorphism
  \eqref{eqn:Liso2}, the map $\Phi_\rho$ is induced by the isomorphism
  of chain complexes
\[
ECC_{-*}^{L_0}(Y,\lambda_0,PD(\frak{s}-\frak{s}_\xi);J_0) \longrightarrow
ECC_{-*}^{L_1}(Y,\lambda_1,PD(\frak{s}-\frak{s}_\xi);J_1)
\]
determined by the canonical bijection on generators.
\end{description}
\end{lemma}

\begin{proof}
  As explained in \cite{km}, the canonical isomorphism on
  Seiberg-Witten Floer cohomology at the bottom of \eqref{eqn:mwh} is
  induced by a chain map which is defined from a suitable 1-parameter
  family of data sets that interpolates between those used to define
  the two chain complexes.  Various relevant aspects of this are
  summarized in \cite[\S3h]{e1}.  In the case at hand, the relevant
  1-parameter family of data sets has the form
  \[
\{D_t=(\lambda_t,J_t,r_t,\mu_t,\frak{p}_t) \mid t\in[0,1]\}.
\]
Here $\{\mu_t\mid t\in[0,1]\}$ is a smooth family of $2$-forms in
$\Omega$ with $\mc{P}$-norm less than $1$; and $\{\frak{p}_t\mid
t\in[0,1]\}$ is a generic smooth family of abstract perturbations with
small $\mc{P}$-norm.  More precisely, recall from \S\ref{sec:pcf} that
$\Omega$ and $\mc{P}$ are smooth Banach space bundles over the space
of metrics on $Y$; the families $\{\mu_t\}$ and $\{\frak{p}_t\}$ are
sections of the restrictions of these bundles to the path of metrics
determined by $\{(\lambda_t,J_t)\}$.  The family $\{\frak{p}_t\}$ can
and should be chosen so that for generic $t\in[0,1]$, the necessary
transversality holds so that the Seiberg-Witten Floer chain complex
$\widehat{CM}^*(Y,\frak{s};\lambda_t,J_t,r_t)$ is defined.

To prove parts (a)--(c), let $N$ be a large positive integer, and
choose numbers $0=t_0<t_1<\cdots<t_N=1$ with $t_i-t_{i-1}<2/N$ for
each $i=1,\ldots,N$, such that the chain complex $\widehat{CM}^*$ is
defined for each data set $D_{t_i}$.  As explained in \cite[\S3h Part
3]{e1}, if $\{\frak{p}_t\mid t\in[0,1]\}$ is generic then for each
$i=1,\ldots,N$, the family of data sets parametrized by
$t\in[t_{i-1},t_i]$ can be used to define a chain map
\begin{equation}
\label{eqn:Ihati}
\widehat{I}_i: \widehat{CM}^*(Y,\frak{s};\lambda_{t_{i-1}},J_{t_{i-1}},r_{t_{i-1}})
\longrightarrow
\widehat{CM}^*(Y,\frak{s};\lambda_{t_i},J_{t_i},r_{t_i}).
\end{equation}
Let $I_i$ denote the map on $\widehat{HM}^*$ induced by $\widehat{I}_i$.
The canonical isomorphism on the bottom of \eqref{eqn:mwh} is then
given by the composition $I_N\circ\cdots\circ I_1$.

Since $L_t$ varies continuously with $t$, it follows from a
compactness argument that there exists $\varepsilon>0$ such that for
each $t\in[0,1]$, the contact form $\lambda_t$ has no orbit set with
action in the interval $[L_t-\varepsilon,L_t+\varepsilon]$.  If $N$ is
sufficiently large, then for each $i$ and for each
$t\in[t_{i-1},t_i]$, we have $|L_t-L_{t_{i-1}}|<\varepsilon$, and in
particular the contact form $\lambda_t$ has no orbit set of action
exactly $L_{t_{i-1}}$.  It then follows from \cite[Lem.\ 4.6]{e1} that if
the numbers $\{r_t\}$ are sufficiently large, then $\widehat{I}_i$
restricts to a chain map $\widehat{CM}^*_{L_{t_{i-1}}}\to
\widehat{CM}^*_{L_{t_{i-1}}}$ which induces an isomorphism
\[
\widehat{HM}^*_{L_{t_{i-1}}}(Y,\frak{s};\lambda_{t_{i-1}},J_{t_{i-1}},r_{t_{i-1}})
\stackrel{\simeq}{\longrightarrow}
\widehat{HM}^*_{L_{t_{i-1}}}(Y,\frak{s};\lambda_{t_i},J_{t_i},r_{t_i}).
\]
Finally, it follows from Lemma~\ref{lem:diff}(b) that, again if the
numbers $\{r_t\}$ are sufficiently large, then there is an isomorphism
\[
\widehat{HM}^*_{L_{t_{i-1}}}(Y,\frak{s};\lambda_{t_i},J_{t_i},r_{t_i})
\stackrel{\simeq}{\longrightarrow}
\widehat{HM}^*_{L_{t_{i}}}(Y,\frak{s};\lambda_{t_i},J_{t_i},r_{t_i}).
\]
induced by the inclusion of one chain complex into the other,
depending on which of $L_{t_{i-1}}$ and $L_{t_i}$ is larger.  We now
define
\[
\Phi_{\rho|_{[t_{i-1},t_i]}}:
\widehat{HM}^*_{L_{t_{i-1}}}(Y,\frak{s};\lambda_{t_{i-1}},J_{t_{i-1}},r_{t_{i-1}})
\stackrel{\simeq}{\longrightarrow}
\widehat{HM}^*_{L_{t_{i}}}(Y,\frak{s};\lambda_{t_i},J_{t_i},r_{t_i})
\]
to be the composition of the above two isomorphisms, and
\[
\Phi_\rho \eqdef \Phi_{\rho|_{[t_{N-1},t_N]}} \circ\cdots\circ
\Phi_{\rho|_{[t_0,t_1]}}.
\]
A two-parameter version of the above subdivision construction, again
using \cite[Lem.\ 4.6]{e1} and assuming that the numbers $\{r_t\}$ are
sufficiently large, shows that the map $\Phi_\rho$ on homology is
independent of the choices made above and satisfies the homotopy
invariance property (a).  Properties (b) and (c) are then immediate
from the construction.

We now show that property (d) holds for a given
$\{(\lambda_t,J_t,L_t)\}$ provided that $\{r_t\}$ is
sufficiently large.  By Lemma~\ref{lem:diff2}, we can fix the spin-c structure $\frak{s}$. Arguing by contradiction, suppose that for each
positive integer $j$ we have a path $\{r_{j,t}\mid t\in[0,1]\}$ for which
property (d) fails, with
$\lim_{j\to\infty}\min_{t\in[0,1]}r_{j,t}=+\infty$.

For each $j$,
for each positive integer $k$, choose a path $\{\frak{p}_{j,k,t}\mid
t\in[0,1]\}$ of abstract perturbations suitable for defining the map
$\Phi_\rho$, such that the following hold for each $j,k,t$:
\begin{description}
\item{(i)}
$\frak{p}_{j,k,t}$ has $\mc{P}$-norm less than $k^{-1}$.
\item{(ii)} There are no negative index $\frak{p}_{j,k,t}$-instantons
  between generators of
  $\widehat{CM}^*_{L_t}(Y,\frak{s};\lambda_t,J_t,r_{j,t})$.  (This can
  be arranged by the Sard-Smale theorem as in \cite[\S7]{tw1}.)
\end{description}

  Now fix $j$ and $k$.  Since property (d) fails for $\{r_{j,t}\}$, it
  follows that if we construct the corresponding map $\Phi_\rho$ using
  $\{\frak{p}_{j,k,t}\}$, then for each $N$ in the construction of
  $\Phi_\rho$, there exists $i\in\{1,\ldots,N\}$ such that the
  corresponding chain map $\widehat{I}_i$ as in \eqref{eqn:Ihati} is
  not the canonical bijection of generators.  Taking $N\to\infty$, a
  compactness argument using (ii) then finds $t_{j,k}\in[0,1]$ and an
  index zero, non-$\R$-invariant $\frak{p}_{j,k,t_{j,k}}$-instanton
  $\frak{d}_{j,k}$ between two generators of
  $\widehat{CM}^*_{L_{t_{j,k}}}(Y,\frak{s};\lambda_{t_{j,k}},J_{t_{j,k}},r_{j,t_{j,k}})$.
  For each $j$, pass to a subsequence of the $k$'s such that the
  sequence $\{t_{j,k}\}$ converges to $t_j\in[0,1]$.  Next pass to a
  subsequence of the $j$'s such that $t_j$ converges to $t_*\in[0,1]$.

  Given the doubly indexed sequence $\{\frak{d}_{j,k}\}$ of
  $\frak{p}_{j,k,t_{j,k}}$-instantons constructed above, the argument
  in \cite[\S8(b)]{e4} can now be repeated almost
  verbatim\footnote{Here one uses the stability condition in
    Remark~\ref{rem:stable} below to deal with the fact that $t_j$
    depends on $j$.} to conclude the following: There exists a broken
  $J_{t_*}$-holomorphic curve in $\R\times Y$ between two generators
  of $ECC_*^{L_{t_*}}(Y,\lambda_{t_*};J_{t_*})$, with each level
  non-$\R$-invariant as in Definition~\ref{def:broken}, and with total
  ECH index zero.  But this contradicts the fact that $J_{t_*}$ is
  $ECH^{L_{t_*}}$-generic, see \cite[Cor.\ 11.5]{t3} or \cite[Prop.\ 3.7]{bn}.
\end{proof}

We can now deduce that $\widehat{HM}^*_L(Y,\frak{s};\lambda,J,r)$
does not depend on $J$ or $r$.

\begin{corollary}
\label{cor:yls}
Suppose $\lambda$ is an $L$-nondegenerate contact form and $\frak{s}$
is a spin-c structure on $Y$.  Then the relatively graded
$\Z/2$-modules $\widehat{HM}^*_L(Y,\frak{s};\lambda,J,r)$ for
different $r$ and $J$ (where $r$ is sufficiently large with respect to
$\lambda,L,J$) are canonically isomorphic to a single relatively
graded $\Z/2$-module $\widehat{HM}^*_L(Y,\lambda,\frak{s})$, with the
following properties:
\begin{description}
\item{(a)}
Inclusion of chain complexes
induces a well-defined map
\[
\widehat{HM}^*_L(Y,\lambda,\frak{s})\longrightarrow
\widehat{HM}^*(Y,\frak{s}).
\]
\item{(b)} If $L<L'$ and if $\lambda$ is also $L'$-nondegenerate, then
  inclusion of chain complexes induces a well-defined map
\[
\widehat{HM}^*_L(Y,\lambda,\frak{s})\longrightarrow
\widehat{HM}^*_{L'}(Y,\lambda,\frak{s}).
\]
\item{(c)}
If $c>0$ then there is a canonical ``scaling'' isomorphism
\begin{equation}
\label{eqn:ylsscaling}
s:\widehat{HM}^*_L(Y,\lambda,\frak{s})\stackrel{\simeq}{\longrightarrow}
\widehat{HM}^*_{cL}(Y,c\lambda,\frak{s}).
\end{equation}
\end{description}
\end{corollary}

\begin{proof}
  Since the space of symplectization-admissible almost complex
  structures for $\lambda$ is contractible, it follows that if $r_i$
  is sufficiently large with respect to $J_i$ for $i=0,1$, then
  Lemma~\ref{lem:deform}(a) provides a well-defined isomorphism
\[
\widehat{HM}^*_L(Y,\frak{s};\lambda,J_0,r_0)\stackrel{\simeq}{\longrightarrow}
\widehat{HM}^*_L(Y,\frak{s};\lambda,J_1,r_1),
\]
induced by an admissible deformation of the form
$\rho=\{(\lambda,L,J_t,r_t)\}$.  By Lemma~\ref{lem:deform}(b), these
isomorphisms satisfy the necessary composition property to identify
the modules $\widehat{HM}^*_L(Y,\frak{s};\lambda,J,r)$ for different $J,r$
with a single relatively graded $\Z/2$-module
$\widehat{HM}^*_L(Y,\lambda,\frak{s})$.

Property (a) now follows immediately from Lemma~\ref{lem:deform}(c).
Property (b) follows similarly from the construction of the maps
$\Phi_\rho$.

To prove property (c), fix $J$ and fix $r$ sufficiently large with
respect to $J$.  Consider the admissible deformation
\begin{equation}
\label{eqn:rhoc}
\rho_c\eqdef\{((1-t+ct)\lambda,(1-t+ct)L,J,r)\}.
\end{equation}
Here we are regarding $J$ as an almost complex structure on $\xi$, so
that the same $J$ can be used for any positive multiple of $\lambda$.
By Lemma~\ref{lem:deform}, the admissible deformation \eqref{eqn:rhoc}
induces an isomorphism
\[
\Phi_{\rho_c}:\widehat{HM}^*_L(Y,\frak{s};\lambda,J,r)
\stackrel{\simeq}{\longrightarrow}
\widehat{HM}^*_{cL}(Y,\frak{s};c\lambda,J,r)
\]
We claim that this isomorphism induces a well-defined isomorphism as
in \eqref{eqn:ylsscaling}.  To prove this, we need to check that given
another pair $(J',r')$, if $\rho_c'$ is the primed analogue of
\eqref{eqn:rhoc}, then the diagram
\[
\begin{CD}
\widehat{HM}^*_L(Y,\frak{s};\lambda,J,r) @>{\Phi_{\rho_c}}>>
\widehat{HM}^*_{cL}(Y,\frak{s};c\lambda,J,r)\\
@VV{\Phi_{\rho_1}}V @VV{\Phi_{\rho_2}}V\\
\widehat{HM}^*_L(Y,\frak{s};\lambda,J',r') @>{\Phi_{\rho_c'}}>>
\widehat{HM}^*_{cL}(Y,\frak{s};c\lambda,J',r')
\end{CD}
\]
commutes.  Here $\rho_1=\{(\lambda,L,J_t,r_t)\}$ and
$\rho_2=\{(c\lambda,cL,J_t,r_t)\}$, where $\{(J_t,r_t)\}$ is a homotopy
from $(J,r)$ to $(J',r')$.  We now observe that both
$\rho_c'\circ\rho_1$ and $\rho_2\circ\rho_c$ are homotopic through
admissible deformations to
\[
\{((1-t+ct)\lambda,(1-t+ct)L,J_t,r_t)\},
\]
and so commutativity of the above diagram follows from
Lemma~\ref{lem:deform}(a),(b).
\end{proof}

Below, when we are not concerned with the spin-c structure, we write
\[
\widehat{HM}^*_L(Y,\lambda) \eqdef
\bigoplus_{\frak{s}\in\Spinc(Y)}\widehat{HM}^*_L(Y,\lambda,\frak{s}).
\]

\subsection{The filtered isomorphism}

We now define an isomorphism from filtered embedded contact homology
to filtered Seiberg-Witten Floer cohomology, and describe how it
behaves under scaling and inclusion of chain complexes.  To obtain a
{\em canonical\/} isomorphism, we will need the following lemma:

\begin{lemma}
\label{lem:plfa}
\cite[Prop.\ B.1]{e1} If $\lambda$ is $L$-nondegenerate and if $J$ is
$ECH^L$-generic, then there exist ``preferred'' $L$-flat
approximations to $(\lambda,J)$, and for each preferred $L$-flat
approximation $(\lambda_1,J_1)$ there exist ``preferred'' homotopies
$\{(\lambda_t,J_t)\mid t\in[0,1]\}$ as in Definition~\ref{def:lfa},
such that:
\begin{description}
\item{(a)}
If $(\lambda_1,J_1)$ is a preferred $L$-flat approximation, then any
two preferred homotopies for $(\lambda_1,J_1)$ are homotopic through
admissible deformations.
\item{(b)}
If $(\lambda_1^0,J_1^0)$ and
  $(\lambda_1^1,J_1^1)$ are two preferred $L$-flat approximations,
  then they are connected by a homotopy of $L$-flat pairs
  $\{(\lambda_1^\nu,J_1^\nu)\mid \nu\in[0,1]\}$ with the following properties:
\begin{description}
\item{(i)} The Reeb orbits of $\lambda_1^\nu$ do not depend on $\nu$.
\item{(ii)} $\{(\lambda_1^\nu,J_1^\nu)\mid \nu\in[0,1]\}$ is homotopic
  through admissible deformations to the composition of a preferred
  homotopy for $(\lambda_1^1,J_1^1)$ with the inverse of a preferred
  homotopy for $(\lambda_1^0,J_1^0)$.
\end{description}
\item{(c)}
For every $\varepsilon>0$, there exists a preferred $L$-flat
approximation $(\lambda_1,J_1)$ with a preferred homotopy
$\{(\lambda_t,J_t)\}$ such that each $(\lambda_t,J_t)$ agrees with
$(\lambda,J)$ except within distance
$\varepsilon$ of the Reeb orbits of action less than $L$.
\end{description}
\end{lemma}

Part (c) of the above lemma will be used in \S\ref{sec:hcg}.

We can now relate filtered ECH to filtered Seiberg-Witten Floer
cohomology:

\begin{lemma}
\label{lem:FI}
Suppose that $\lambda$ is $L$-nondegenerate and $J$ is $ECH^L$-generic.
Then for all $\Gamma\in H_1(Y)$, there is a canonical isomorphism of
relatively graded $\Z/2$-modules
\begin{equation}
\label{eqn:FI}
\Psi^L:
ECH_*^L(Y,\lambda,\Gamma;J) \stackrel{\simeq}{\longrightarrow}
\widehat{HM}^{-*}_L(Y,\lambda,\frak{s}_{\xi,\Gamma})
\end{equation}
with the following properties:
\begin{description}
\item{(a)}
If $L<L'$, if $\lambda$ is $L'$-nondegenerate, and if $J$ is
$ECH^{L'}$-generic, then the diagram
\[
\begin{CD}
ECH_*^L(Y,\lambda,\Gamma;J)@ >{\Psi^L}>>
\widehat{HM}^{-*}_L(Y,\lambda,\frak{s}_{\xi,\Gamma})\\
@V{\imath^{L,L'}_J}VV @VVV \\
ECH_*^{L'}(Y,\lambda,\Gamma;J) @>{\Psi^{L'}}>>
\widehat{HM}^{-*}_{L'}(Y,\lambda,\frak{s}_{\xi,\Gamma})
\end{CD}
\]
commutes, where $\imath^{L,L'}_J$ is the inclusion-induced map
\eqref{eqn:iLL'J}, and the right arrow is the inclusion-induced map in
Corollary~\ref{cor:yls}(b).
\item{(b)}
If $c>0$, then the diagram
\[
\begin{CD}
ECH_*^L(Y,\lambda,\Gamma;J) @>{\Psi^L}>> \widehat{HM}^{-*}_L(Y,\lambda,\frak{s}_{\xi,\Gamma})\\
@V{s_J}VV @V{s}VV \\
ECH_*^{cL}(Y,c\lambda,\Gamma;J) @>{\Psi^{cL}}>>
 \widehat{HM}^{-*}_{cL}(Y,c\lambda,\frak{s}_{\xi,\Gamma})
\end{CD}
\]
commutes, where $s_J$ is the scaling isomorphism \eqref{eqn:s}, and
$s$ is the scaling isomorphism in Corollary~\ref{cor:yls}(c).
\end{description}
\end{lemma}

\begin{proof}
  Let $(\lambda_1,J_1)$ be a preferred $L$-flat approximation to
  $(\lambda,J)$, and let $\{(\lambda_t,J_t)\mid t\in[0,1]\}$ be a
  preferred homotopy from $(\lambda,J)$ to $(\lambda_1,J_1)$.  If $r$
  is sufficiently large, then by \eqref{eqn:Liso2} we have a canonical
  isomorphism
\[
ECH_*^L(Y,\lambda,\Gamma;J) \stackrel{\simeq}{\longrightarrow}
\widehat{HM}^{-*}_L(Y,\frak{s}_{\xi,\Gamma};\lambda_1,J_1,r).
\]
By Lemma~\ref{lem:deform}, the admissible deformation
\begin{equation}
\label{eqn:rho1}
\rho_1=\{(\lambda_{1-t},L,J_{1-t},r)\mid
t\in[0,1]\}
\end{equation}
determines an isomorphism
\begin{equation}
\label{eqn:htud}
\Phi_{\rho_1}:
\widehat{HM}^{-*}_L(Y,\frak{s}_{\xi,\Gamma};\lambda_1,J_1,r)
\stackrel{\simeq}{\longrightarrow}
\widehat{HM}^{-*}_L(Y,\frak{s}_{\xi,\Gamma};\lambda,J,r).
\end{equation}
By Lemmas~\ref{lem:plfa}(a) and \ref{lem:deform}(a), the map
\eqref{eqn:htud} does not depend on the choice of preferred homotopy.
Let
\[
\widetilde{\Psi}^L: ECH_*^L(Y,\lambda,\Gamma;J)
\stackrel{\simeq}{\longrightarrow}
\widehat{HM}^{-*}_L(Y,\frak{s}_{\xi,\Gamma};\lambda,J,r)
\]
denote the composition of the previous two isomorphisms.  We claim that
$\widetilde{\Psi}^L$ induces a well-defined map $\Psi^L$ as in \eqref{eqn:FI}.

We first show that $\widetilde{\Psi}^L$ does not depend on the choice
of preferred $L$-flat approximation.  Given two preferred $L$-flat
approximations $(\lambda_1^0,J_1^0)$ and $(\lambda_1^1,J_1^1)$, let
$\{(\lambda_1^\nu,J_1^\nu)\mid \nu\in[0,1]\}$ be a homotopy of
$L$-flat pairs provided by Lemma~\ref{lem:plfa}(b).  By
Lemma~\ref{lem:deform}(a),(b), the isomorphisms \eqref{eqn:htud} for
the two preferred $L$-flat approximations differ by the isomorphism
\[
\Phi_{\rho_2}:
\widehat{HM}^{-*}_L(Y,\frak{s}_{\xi,\Gamma};\lambda_1^0,J_1^0,r)
\stackrel{\simeq}{\longrightarrow}
\widehat{HM}^{-*}_L(Y,\frak{s}_{\xi,\Gamma};\lambda_1^1,J_1^1,r)
\]
induced by the admissible deformation
\[
\rho_2=\{(\lambda_1^\nu,L,J_1^\nu,r)\mid
\nu\in[0,1]\}.
\]
Applying Lemma~\ref{lem:deform}(d) to the latter path then shows that
the two versions of $\widetilde{\Psi}^L$ defined using the two
preferred $L$-flat approximations agree.

We now show that $\Psi^L$ does not depend on the choice of $r$.
Suppose that $r,r'$ are both sufficiently large to define the
isomorphism $\widetilde{\Psi}_L$.  To prove that the versions of
$\Psi_L$ defined using $r$ and $r'$ agree, it is enough to show that
the following diagram commutes:
\[
\begin{CD}
ECH_*^L(\lambda,J) @>{\simeq}>> \widehat{HM}^{-*}_L(\lambda_1,J_1,r)
@>{\Phi_{\rho_1}}>> \widehat{HM}^{-*}_L(\lambda,J,r) \\
@| @V{\Phi_{\rho_3}}VV @V{\Phi_{\rho_4}}VV \\
ECH_*^L(\lambda,J) @>{\simeq}>> \widehat{HM}^{-*}_L(\lambda_1,J_1,r')
@>{\Phi_{\rho_1'}}>> \widehat{HM}^{-*}_L(\lambda,J,r').
\end{CD}
\]
Here we have dropped $Y$ and $\Gamma$ from the notation; the
horizontal isomorphisms on the left are given by \eqref{eqn:Liso2};
the admissible deformation $\rho_1'$ is defined as in \eqref{eqn:rho1}
but with $r$ replaced by $r'$; and
\[
\begin{split}
\rho_3 &= \{(\lambda_1,L,J_1,(1-t)r+tr')\mid t\in[0,1]\},\\
\rho_4 &= \{(\lambda,L,J,(1-t)r+tr')\mid t\in[0,1]\}.
\end{split}
\]
The left square commutes by Lemma~\ref{lem:deform}(d).  The right
square commutes by Lemma~\ref{lem:deform}(a),(b), because both
$\rho_4\circ\rho_1$ and $\rho_1'\circ\rho_3$ are homotopic through
admissible deformations to
\[
\{(\lambda_{1-t},L,J_{1-t},(1-t)r+tr')\mid t\in[0,1]\}.
\]
This completes the proof that $\Psi^L$ is well-defined.

To prove that $\Psi^L$ satisfies property (a), choose a preferred
$L'$-flat approximation $(\lambda_1,J_1)$ to define $\Psi^{L'}$.  Then
this is also a preferred $L$-flat approximation which can be used to
define $\Psi^L$.  It now suffices to show that the diagram
\[
\begin{CD}
ECH_*^L(\lambda,J)@ >{\simeq}>>
\widehat{HM}^{-*}_L(\lambda_1,J_1,r) @>{\Phi_{\rho_1}}>>
\widehat{HM}^{-*}_L(\lambda,J,r)\\
@V{\imath^{L,L'}_J}VV @VVV @VVV \\
ECH_*^{L'}(\lambda,J) @>{\simeq}>>
\widehat{HM}^{-*}_{L'}(\lambda_1,J_1,r) @>{\Phi_{\rho_1''}}>>
\widehat{HM}^{-*}_{L'}(\lambda,J,r)
\end{CD}
\]
commutes.  Here $\rho_1''$ is defined as in \eqref{eqn:rho1} but with
$L$ replaced by $L'$;  and the vertical arrows in the diagram are
induced by inclusions of chain complexes.  Now the left square
commutes by the definition of the isomorphism \eqref{eqn:Liso}, while
the right square commutes by a straightforward analogue of
Lemma~\ref{lem:deform}(c).

To prove property (b), let us further drop $r$ from the notation and
consider the diagram
\[
\begin{CD}
ECH_*^L(\lambda,J) @>{\simeq}>> ECH_*^L(\lambda_1,J_1) @>{\simeq}>>
\widehat{HM}^{-*}_L(\lambda_1,J_1) @>{\Phi_{\rho_1}}>>
\widehat{HM}^{-*}_L(\lambda,J)\\
@V{s_J}VV @V{s_{J_1}}VV @V{\Phi_{\rho_c^1}}VV @V{\Phi_{\rho_c}}VV\\
ECH_*^{cL}(c\lambda,J) @>{\simeq}>> ECH_*^{cL}(c\lambda_1,J_1) @>{\simeq}>>
\widehat{HM}^{-*}_{cL}(c\lambda_1,J_1) @>{\Phi_{\rho_1^c}}>>
\widehat{HM}^{-*}_{cL}(c\lambda,J).
\end{CD}
\]
Here $\rho_c$ was defined in \eqref{eqn:rhoc}; $\rho_c^1$ denotes the
analogue of \eqref{eqn:rhoc} for $(\lambda_1,J_1)$; and $\rho_1^c$ is
obtained from \eqref{eqn:rho1} by multiplying the contact forms and
$L$ by $c$.  Also the horizontal isomorphisms on the left are induced
by \eqref{eqn:clf}, and the horizontal isomorphisms in the middle are
induced by \eqref{eqn:Liso}.  By definition, the composition of the
horizontal arrows in the top row of the above diagram is $\Psi^L$, and
the composition of the horizontal arrows in the bottom row is
$\Psi^{cL}$.  So to prove property (b) it is enough to show that the above diagram commutes.  The left square commutes at the
chain level because each map in the left square sends each admissible
orbit set to itself.  The middle square commutes by
Lemma~\ref{lem:deform}(d).  The right square commutes by
Lemma~\ref{lem:deform}(a),(b), because both $\rho_c\circ\rho_1$ and
$\rho_1^c\circ\rho_c^1$ are homotopic to
\[
\left\{(1-t+ct)\lambda_{1-t},(1-t+ct)L,J_{1-t},r)\mid t\in[0,1]\right\}
\]
through admissible deformations.
\end{proof}

\subsection{$J$-independence of filtered ECH (proof)}

We now have enough machinery in place to prove Theorem~\ref{thm:FECH},
asserting that $ECH$ and $ECH^L$ do not depend on the choice of almost
complex structure used to define them.

\begin{proof}[Proof of Theorem~\ref{thm:FECH}.]
  We may assume, by slightly decreasing $L$ if necessary, that there
  is no orbit set of action exactly $L$.  Part (a) then follows from
  the canonical isomorphism \eqref{eqn:FI} given by
  Lemma~\ref{lem:FI}.  Part (b) follows from Lemma~\ref{lem:FI}(a).
  Part (c) follows from part (b) by taking direct limits.  Part (d)
  follows from Lemma~\ref{lem:FI}(b).  Part (e) follows from the
  definition of the isomorphism between ECH and $\widehat{HM}^*$
  reviewed in \S\ref{sec:fulliso} below.
\end{proof}

\begin{remark}
  ECH has various additional structures on it which we are not using
  in this paper, for example a degree $-2$ map $U$.  It is shown in
  \cite{e5} that these agree with analogous structures on
  Seiberg-Witten Floer cohomology under the isomorphism determined by
  \eqref{eqn:FI} (see \S\ref{sec:fulliso} below).  Consequently the
  proof of Theorem~\ref{thm:FECH} shows that these additional
  structures are also independent of $J$.
\end{remark} 

\subsection{The full isomorphism}
\label{sec:fulliso}

We are now in a position to write down the full isomorphism from embedded
contact homology to Seiberg-Witten Floer cohomology.

Let $Y$ be a closed oriented connected 3-manifold with a nondegenerate
contact form $\lambda$, and fix $\Gamma\in H_1(Y)$.  By
Lemma~\ref{lem:FI}, if $\lambda$ has no orbit set of action $L$, then
for each $\Gamma\in H_1(Y)$ there is a well-defined isomorphism
\begin{equation}
\label{eqn:FI2}
ECH_*^L(Y,\lambda,\Gamma) \stackrel{\simeq}{\longrightarrow}
\widehat{HM}_L^{-*}(Y,\lambda,\frak{s}_{\xi,\Gamma}).
\end{equation}
By Corollary~\ref{cor:yls}(a), there is a well-defined map
\begin{equation}
\label{eqn:FI3}
\widehat{HM}_L^{-*}(Y,\lambda,\frak{s}_{\xi,\Gamma})
\longrightarrow 
\widehat{HM}^{-*}(Y,\frak{s}_{\xi,\Gamma}).
\end{equation}
We now define
\begin{equation}
\label{eqn:TL}
T^L:ECH_*^L(Y,\lambda,\Gamma) \longrightarrow
\widehat{HM}^{-*}(Y,\frak{s}_{\xi,\Gamma})
\end{equation}
to be the composition of the maps \eqref{eqn:FI2} and \eqref{eqn:FI3} above.

If $L<L'$, then it follows from Lemma~\ref{lem:FI}(a) that
\[
T^L=T^{L'}\circ \imath^{L,L'},
\]
where $\imath^{L,L'}$ is the inclusion-induced map \eqref{eqn:istar}.
This means that it makes sense to define
\begin{equation}
\label{eqn:T}
T:ECH_*(Y,\lambda,\Gamma) \longrightarrow
\widehat{HM}^{-*}(Y,\frak{s}_{\xi,\Gamma})
\end{equation}
to be the direct limit over $L$ of the maps $T^L$ in \eqref{eqn:TL}.
The main theorem of \cite{e1} (after passing to $\Z/2$ coefficients)
can now be stated as follows:

\begin{theorem}
\label{thm:echswf}
\cite{e1}
The map $T$ in \eqref{eqn:T} is an isomorphism of relatively graded
$\Z/2$-modules.
\end{theorem}

Knowing that \eqref{eqn:FI2} is an isomorphism, the rest of the proof
of Theorem~\ref{thm:echswf} amounts to showing that the maps
\eqref{eqn:FI3} induce an isomorphism
\[
\lim_{\rightarrow}
\widehat{HM}_L^{-*}(Y,\lambda,\frak{s}_{\xi,\Gamma})
\stackrel{\simeq}{\longrightarrow} 
\widehat{HM}^{-*}(Y,\frak{s}_{\xi,\Gamma}),
\]
see \cite[Thm.\ 4.5]{e1}.  (This is not immediately obvious because
one has to increase $r$ as one increases $L$ in order to define the
left hand side, see \S\ref{sec:energyfiltration}.)

\section{Seiberg-Witten Floer cobordism maps and symplectic forms}
\label{sec:swfcob}

We now review from \cite[Ch.\ 24]{km} the maps on Seiberg-Witten Floer
cohomology induced by a (smooth) cobordism.  We then introduce a
perturbation of the relevant Seiberg-Witten equations on an exact
symplectic cobordism using the symplectic form.

\subsection{Smooth cobordisms}
\label{sec:smoothcob}

Let $Y_+$ and $Y_-$ be closed oriented (connected) three-manifolds.
Let $X$ be a cobordism from $Y_+$ to $Y_-$ as in \S\ref{sec:cobech}.

Given some metric on $X$, a {\em spin-c structure\/} on $X$ is a lift
of the frame bundle of $TX$ from $SO(4)$ to
\[
\Spinc(4) = \Spin(4) \times_{\Z/2} U(1).
\]
This is equivalent to a Hermitian vector bundle
${\Sp}=\Sp_+\oplus\Sp_-$, where $\Sp_+$ and $\Sp_-$ each have rank
$2$, together with a Clifford multiplication $cl:TX\to\End(\Sp)$
satisfying \eqref{eqn:clifford}, such that $cl(v)$ exchanges $\Sp_+$
and $\Sp_-$ for each $v\in TX$, and
\[
cl(e_1)cl(e_2)cl(e_3)cl(e_4) = \begin{pmatrix} -1 & 0 \\ 0 &
  1 \end{pmatrix}
\]
on $\Sp_+\oplus\Sp_-$ whenever $(e_1,e_2,e_3,e_4)$ is an oriented
orthonormal basis for $T_xX$.  The set $\Spinc(X)$ of isomorphism
classes of spin-c structures on $X$ is an affine space over
$H^2(X;\Z)$, with the action as in \eqref{eqn:affine}, which does not
depend on the choice of metric.  Given a spin-c structure on $X$, a
spin-c connection is defined as in \eqref{eqn:spincc}.  A spin-c
connection $\A_\Sp$ is equivalent to a Hermitian connection $\A$ on
$\det(\Sp_+)=\det(\Sp_-)$, and adding an imaginary-valued $1$-form $a$
to $\A$ adds $a/2$ to $\A_\Sp$.  As in \eqref{eqn:dirac}, the connection $\A$
defines a Dirac operator
\[
D_\A: C^\infty(X;\Sp_{\pm}) \longrightarrow C^\infty(X;\Sp_{\mp}).
\]
A spin-c structure $\frak{s}$ on $X$ restricts to a spin-c structure
$\frak{s}|_{Y_\pm}$ on $Y_\pm$ as follows.  Let $v$ denote the outward
pointing unit normal vector to $Y_+$, and the inward pointing unit
normal vector to $Y_-$.  If $\Sp=\Sp_+\oplus\Sp_-$ is the spin bundle
for $\frak{s}$ with Clifford multiplication $cl$, then we define the
spin bundle $\Sp_{Y_\pm}$ for $\frak{s}|_{Y_\pm}$ to be
\begin{equation}
\label{eqn:SY}
\Sp_{Y_\pm}\eqdef (\Sp_+)|_{Y_\pm}
\end{equation}
with the Clifford action
$TY\to\End(\Sp_{Y_\pm})$ given by $cl(v)^{-1}cl(\cdot)$.

If $\frak{s}$ is a spin-c structure on $X$ with
$\frak{s}_\pm\eqdef\frak{s}|_{Y_\pm}$, then there is a cobordism map (of ungraded
$\Z/2$-modules)
\begin{equation}
\label{eqn:cm}
\widehat{HM}^*(X,\frak{s}): \widehat{HM}^*(Y_+,\frak{s}_+)
\longrightarrow \widehat{HM}^*(Y_-,\frak{s}_-).
\end{equation}
We now review the basic formalism of the definition of this map; the
details are explained in \cite{km}.  Choose a metric $g_\pm$, exact
$2$-form $\eta_{\pm}$, and abstract perturbation $\frak{p}_\pm$ as needed
to define the chain complex
$\widehat{CM}^*(Y_\pm,\frak{s}_\pm;g_\pm,\eta_{\pm})$.  One defines a
chain map
\begin{equation}
\label{eqn:ccm}
\widehat{CM}^*(Y_+,\frak{s}_+;g_+,\eta_{+}) \longrightarrow
 \widehat{CM}^*(Y_-,\frak{s}_-;g_-,\eta_{-})
\end{equation}
as follows.
Attach cylindrical ends to $X$ to obtain
\[
\overline{X} \eqdef ((-\infty,0]\times Y_-) \cup_{Y_-} X
\cup_{Y_+} ([0,\infty)\times Y_+).
\]
Choose a metric $g$ on $\overline{X}$ which on the ends agrees with
the product of the standard metric on $(-\infty,0]$ or $[0,\infty)$
with the chosen metric $g_\pm$ on $Y_\pm$.  Choose a self-dual 2-form
$\eta$ on $\overline{X}$ which on each end agrees with the self-dual
part of (the pullback of) $\eta_{\pm}$, namely
$\frac{1}{2}(\eta_{\pm}+*\eta_{\pm})$, where $*$ denotes the Hodge star on
$\overline{X}$.  The spin-c structure on $X$ has a canonical extension
over $\overline{X}$, so that on each end, $\Sp_+$ and $\Sp_-$ are both
identified with the boundary spinor bundle, and if $s$ denotes the
$(-\infty,0]$ or $[0,\infty)$ coordinate, then
$cl(\partial_s):\Sp_+\stackrel{\simeq}{\to}\Sp_-$ preserves the
identifications with the boundary spinor bundle.

We now consider solutions to the Seiberg-Witten equations on
$\overline{X}$.  These equations concern a pair $(\A,\Psi)$, where $\A$
is a connection on $\det(\Sp_+)$ and $\Psi$ is a section of
$\Sp_+$. Without abstract perturbation terms (which we will describe
shortly), the equations are
\begin{equation}
\label{eqn:sw4}
\begin{split}
D_\A\Psi &= 0,\\
F_\A^+ &= \frac{1}{2}\rho(\Psi) + i\eta.
\end{split}
\end{equation}
Here $F_\A^+$ denotes the self-dual part of the curvature $F_\A$, and
$\rho:\Sp_+\to\bigwedge^2_+T^*X$ is a quadratic bundle map defined by
\[
\rho(\Psi)(v,w) = -\frac{1}{2}\langle[cl(v),cl(w)]\Psi,\Psi\rangle
\]
for $\Psi\in(\Sp_+)_x$ and $v,w\in T_xX$.  The gauge group
$C^\infty(X;S^1)$ acts on the set of solutions, again by
\eqref{eqn:gauge}. 

A connection $\A$ on $\det(\Sp_+)$ is in {\em temporal gauge\/} on the ends
if on $(-\infty,0]\times Y_-$ and $[0,\infty)\times Y_+$ one has
\begin{equation}
\label{eqn:tg}
\nabla_\A = \frac{\partial}{\partial s} + \nabla_{\A(s)}
\end{equation}
where $\A(s)$ is a connection on the bundle $\det(\Sp_{Y_\pm})$ over
the $3$-manifold $Y_\pm$, depending on $s$.  Any connection can be
placed into temporal gauge by an appropriate gauge transformation.
After this has been done, the equations \eqref{eqn:sw4} on the ends
are equivalent to the instanton equations \eqref{eqn:instanton}.

To define cobordism maps, we also need to consider abstract
perturbations of the equations \eqref{eqn:sw4}.  Suppose that
$\frak{p}_+$ and $\frak{p}_-$ are abstract perturbations for use in
defining the perturbations of the equations \eqref{eqn:sw3} and
\eqref{eqn:instanton} on $Y_+$ and $Y_-$.  It is explained in
\cite[Ch.\ 11]{km} how these are extended as an abstract perturbation
$\frak{p}$ over all of $\overline{X}$.  The resulting perturbation of
\eqref{eqn:sw4} agrees on $(-\infty,0]\times Y_-$ or $[0,\infty)\times
Y_+$ with the corresponding perturbation of \eqref{eqn:instanton} via
$\frak{p}_-$ or $\frak{p}_+$.  Any such extension must be suitably
generic in order to use the solutions of the perturbed version of
\eqref{eqn:sw4} to define the chain map \eqref{eqn:ccm}.  In particular, a nonzero
extension may be necessary even  when $\frak{p}_-$ and
$\frak{p}_+$ are both zero.

Let $(\A_\pm,\Psi_\pm)$ be solutions to the three-dimensional
Seiberg-Witten equations \eqref{eqn:sw3} for
$(Y_\pm,\frak{s}_\pm;g_\pm,\eta_{\pm})$.  We are interested in solutions
to the abstract perturbation of the four-dimensional Seiberg-Witten
equations \eqref{eqn:sw4} which on the ends are in temporal gauge and
satisfy the convergence conditions
\begin{equation}
\label{eqn:coe}
\begin{split}
\lim_{s\to\infty}(\A(s),\Psi(s)) &= (\A_+,\Psi_+) \quad \mbox{as $s\to
  +\infty$ on $[0,\infty)\times Y_+$},\\
\lim_{s\to-\infty}(\A(s),\Psi(s)) &= (\A_-,\Psi_-) \quad \mbox{as $s\to
  -\infty$ on $(-\infty,0]\times Y_-$}.
\end{split}
\end{equation}
A solution to the perturbed equations \eqref{eqn:sw4} satisfying
\eqref{eqn:coe} will be called an ``instanton from $(\A_-,\Psi_-)$ to
$(\A_+,\Psi_+)$''.  We often denote an instanton as above by
$\frak{d}$ and write $\frak{d}|_s\eqdef(\A(s),\Psi(s))$ and
$\frak{c}_{\pm} \eqdef (\A_\pm,\Psi_\pm)$.  Every instanton has an
{\em index\/}, which is the expected dimension of the corresponding
component of the moduli space of instantons (with the same asymptotic
decay rate as $s\to +\infty$ or $s\to-\infty$ if $(\A_+,\Psi_+)$ or
$(\A_-,\Psi_-)$ respectively is reducible) modulo gauge equivalence.
The component of the chain map \eqref{eqn:ccm} from an irreducible
generator $(\A_+,\Psi_+)$ to an irreducible generator $(\A_-,\Psi_-)$
counts index zero instantons from $(\A_-,\Psi_-)$ to $(\A_+,\Psi_+)$
modulo gauge equivalence.  All we need to know about the remaining
components of \eqref{eqn:ccm} is the following: if $(\A_+,\Psi_+)$ is
irreducible, and if there are no index zero instantons to
$(\A_+,\Psi_+)$ from a reducible $(\A_-,\Psi_-)$, then the chain map
\eqref{eqn:ccm} sends $(\A_+,\Psi_+)$ to an element of
$\widehat{CM}^*_{irr}(Y_-,\cdots)$.

Although the chain map \eqref{eqn:ccm} may depend on the abstract
perturbations, the induced map on homology
\begin{equation}
\label{eqn:imh}
\widehat{HM}^*(X,\frak{s};g,\eta):
\widehat{HM}^*(Y_+,\frak{s}_+;g_+,\eta_{+}) \longrightarrow
\widehat{HM}^*(Y_-,\frak{s}_-;g_-,\eta_{-})
\end{equation}
does not. To show that this map does not depend on the extension
$\frak{p}$ of $\frak{p}_+$ and $\frak{p}_-$, given a homotopy of
extensions $\frak{p}$ one defines a chain homotopy between the
corresponding chain maps by counting index $-1$ instantions.  The
proof that the map \eqref{eqn:imh} does not depend on $\frak{p}_+$ or
$\frak{p}_-$ either is a special case of a more general argument which
we will outline shortly.

In the special case when $X$ is a product cobordism $[0,1]\times Y$,
the maps \eqref{eqn:imh} define the canonical isomorphisms that prove
that $\widehat{HM}^*(Y,\frak{s};g,\eta)$ does not
depend\footnote{\cite[\S3h]{e1} says more about this in the
  case when $g$ and $\eta$ are determined by a contact form as in
  \eqref{eqn:tsw3} and \eqref{eqn:Tinstanton}.}  on $g$ or $\eta$ (or
the abstract perturbations that we are supressing from the notation).
The necessary composition property for these isomorphisms follows from
the following more general composition property.  Let $X^+$ be a
cobordism from $Y_+$ to $Y_0$, let $X^-$ be a cobordism from $Y_0$ to
$Y_-$, and let $X=X^-\cup_{Y_0} X^+$ be the composite cobordism from
$Y_+$ to $Y_-$.  If $\frak{s^\pm}\in\Spinc(X^\pm)$, if
$(g^\pm,\eta^\pm)$ are choices to define the cobordism map on $X^\pm$,
and if $(g,\eta)$ are choices to define the cobordism map on $X$, then
\begin{gather}
\nonumber
\widehat{HM}^*(X^-,\frak{s}^-;g^-,\eta^-) \circ
\widehat{HM}^*(X^+,\frak{s}^+;g^+,\eta^+) = \quad\quad\quad\quad\quad \\
\label{eqn:ccs}
\quad\quad\quad\quad\quad =
\sum_{\left\{\frak{s}\in\Spinc(X) \;\big|\; \frak{s}|_{X^\pm}=\frak{s}^\pm\right\}}
\widehat{HM}^*(X,\frak{s};g,\eta).
\end{gather}
Note here that the sum on the right is well-defined, because by \cite[Prop.\ 24.6.6]{km} the
cobordism map \eqref{eqn:imh} is nonzero for only finitely many spin-c
structures on $X$.  Equation \eqref{eqn:ccs} is proved by ``stretching
the neck'' along $Y_0$ and counting index $-1$ instantons to define a
chain homotopy between the corresponding chain maps.

The special case of \eqref{eqn:ccs} when $X^+$ and $X^-$ are both
product cobordisms gives the composition property needed to show that
$\widehat{HM}^*(Y,\frak{s})$ is well-defined.  The special case
of \eqref{eqn:ccs} when just one of $X^+$ or $X^-$ is a product
cobordism then implies that the map \eqref{eqn:imh} induces a
well-defined map \eqref{eqn:cm}.  With these identifications,
\eqref{eqn:ccs} now translates to
\begin{equation}
\label{eqn:coco}
\widehat{HM}^*(X^-,\frak{s}^-) \circ
\widehat{HM}^*(X^+,\frak{s}^+) =
\sum_{\left\{\frak{s}\in\Spinc(X) \;\big|\;
    \frak{s}|_{X^\pm}=\frak{s}^\pm\right\}}
\widehat{HM}^*(X,\frak{s}).
\end{equation}

One can also combine the cobordism maps \eqref{eqn:cm} into a single
cobordism map
\begin{equation}
\label{eqn:scm}
\widehat{HM}^*(X) \eqdef \sum_{\frak{s}\in\Spin^c(X)}
\widehat{HM}^*(X,\frak{s}) : \widehat{HM}^*(Y_+) \longrightarrow
\widehat{HM}^*(Y_-).
\end{equation}
The composition property \eqref{eqn:coco} implies that
\[
\widehat{HM}^*(X) = \widehat{HM}^*(X_-) \circ \widehat{HM}^*(X_+).
\]

Note that when $X$ is not a product, the cobordism map \eqref{eqn:scm} generally
does not preserve the relative gradings, although there is a weaker
relation between the gradings of the inputs and outputs of this map
explained in \cite[\S3.4]{km}.  We will simply regard \eqref{eqn:scm}
as a map of ungraded $\Z/2$-modules.

\subsection{Perturbing the equations on an exact symplectic cobordism}
\label{sec:pesc}

We now introduce a useful perturbation of the four-dimensional
Seiberg-Witten equations on an exact symplectic cobordism.  This is
closely related to the perturbation of the three-dimensional
Seiberg-Witten equations defined in \S\ref{sec:pcf}.

Let $(Y_+,\lambda_+)$ and $(Y_-,\lambda_-)$ be closed oriented
(connected) 3-manifolds with contact forms.  Let $(X,\lambda)$ be an
exact symplectic cobordism from $(Y_+,\lambda_+)$ to
$(Y_-,\lambda_-)$.  Recall the notion of ``cobordism-admissible almost
complex structure'' from Definition~\ref{def:cobadm}.  Below it will
be convenient to work with a slightly stronger notion.  Note that if
$\varepsilon>0$ is as in \eqref{eqn:N-} and \eqref{eqn:N+}, then the
completion $\overline{X}$ in \eqref{eqn:completion} contains subsets
identified with $(-\infty,\varepsilon]\times Y_-$ and
$[-\varepsilon,\infty)\times Y_+$.

\begin{definition}
\label{def:sca}
An almost complex structure $J$ on $\overline{X}$ is {\em strongly
  cobordism-admissible\/} if it is $\omega$-compatible on $X$, and if
it agrees with symplectization-admissible almost complex structures
$J_+$ for $\lambda_+$ on $[-\varepsilon,\infty)\times Y_+$ and
$J_-$ for $\lambda_-$ on $(-\infty,\varepsilon]\times Y_-$, for some
$\varepsilon>0$ as in \eqref{eqn:N-} and \eqref{eqn:N+}.
\end{definition}

Given $\varepsilon$ and $J$ as above, we define a $1$-form
$\widetilde{\lambda}$ on $\overline{X}$ as follows.  Fix a smooth
increasing function
$\phi_-:(-\infty,\varepsilon]\to(-\infty,\varepsilon]$ with
$\phi_-(s)=2s$ for $s\le \varepsilon/10$ and $\phi_-(s)=s$ for
$s>\varepsilon/2$.  Likewise fix a smooth increasing function
$\phi_+:[-\varepsilon,\infty)\to[-\varepsilon,\infty)$ with
$\phi_+(s)=s$ for $s\le -\varepsilon/2$ and $\phi_+(s)=2s$ for $s\ge
-\varepsilon/10$.  Now define
\begin{equation}
\label{eqn:widetildelambda}
\widetilde{\lambda} \eqdef \left\{\begin{array}{cl} e^{\phi_-}\lambda_- & \mbox{on $(-\infty,\varepsilon]\times Y_-$},\\
\lambda & \mbox{on $X\setminus(([0,\varepsilon]\times Y_-) \cup
      ([-\varepsilon,0]\times Y_+))$},\\
    e^{\phi_+}\lambda_+ & \mbox{on $[-\varepsilon,\infty)\times Y_+$}.\\
\end{array}\right.
\end{equation}
Write $\widetilde{\omega}=d\widetilde{\lambda}$; this is a symplectic form on all
of $\overline{X}$.  Also, $J$ is $\widetilde{\omega}$-compatible on all of
$\overline{X}$.

\begin{remark}
  It would be more usual to define $\widetilde{\lambda}$ by extending
  the $1$-form $\lambda$ on all of $X$ to agree with $e^s\lambda_+$ on
  $[0,\infty)\times Y_+$, and with $e^s\lambda_-$ on
  $(-\infty,0]\times Y_-$.  We are using the more nonstandard $1$-form
  \eqref{eqn:widetildelambda} because of the factors of $2$ discussed
  in Remark~\ref{rmk:factor2}.
\end{remark}

We next define a metric $g$ on $\overline{X}$ as follows.  Let $g_\pm$
denote the metric on $Y_\pm$ determined by $\lambda_\pm$ and $J_\pm$
as in \S\ref{sec:pcf}.  Fix a smooth positive function $\sigma_-$ on
$(-\infty,\varepsilon]$ such that $\sigma_-(s)=2e^{2s}$ for $s\le
\varepsilon/10$ and $\sigma_-(s)=2$ for $s\ge \varepsilon/2$.  Likewise
fix a smooth positive function $\sigma_+$ on $[-\varepsilon,\infty)$
such that $\sigma_+(s)=2$ for $s\le -\varepsilon/2$ and
$\sigma_+(s)=2e^{2s}$ for $s\ge -\varepsilon/10$.  Also require that\footnote{The
  condition \eqref{eqn:sigmabound} will be used in
  Lemma~\ref{lem:3.4x}.}
\begin{equation}
\label{eqn:sigmabound}
\sigma_\pm(s)\in [3/2,5/2] \quad \mbox{for $\pm s\in [0,\varepsilon]$}.
\end{equation}
Define a positive function $\sigma$ on $\overline{X}$
to equal $\sigma_-$ on $(-\infty,\varepsilon]\times Y_-$, to equal
$\sigma_+$ on $[-\varepsilon,\infty)\times Y_+$, and to equal $2$ on
the rest of $\overline{X}$.  Define a metric $g$ on $\overline{X}$ by
\begin{equation}
\label{eqn:varphi}
g(\cdot,\cdot) = \sigma^{-1}\widetilde{\omega}(\cdot,J(\cdot))
\end{equation}
Note that $g$ agrees with the product metric with $g_\pm$ on the ends
$[0,\infty)\times Y_+$ and $(-\infty,0]\times Y_-$.  Also,
$\widetilde{\omega}$ is self-dual with respect to $g$ and has norm
$|\widetilde{\omega}|=\sqrt{2}\sigma$.  Define
$\hat{\omega}=\sqrt{2}\widetilde{\omega}/|\widetilde{\omega}|=
\sigma^{-1}\widetilde{\omega}$.

Let $\frak{s}$ be a spin-c structure on $X$ with spinor bundle
$\Sp=\Sp_+\oplus\Sp_-$.  There is a canonical decomposition
\begin{equation}
\label{eqn:seke4}
\Sp_+ = E \oplus K^{-1}E
\end{equation}
into eigenbundles of $cl(\hat{\omega})$, where $E$ is the $-2i$
eigenbundle, and $K$ denotes the canonical bundle of $(X,J)$.  Note
that on $[0,\infty)\times Y_+$ or $(-\infty,0]\times Y_-$, under the
identification \eqref{eqn:SY}, this splitting agrees with the
splitting determined by Clifford multiplication by $\lambda_+$ or
$\lambda_-$ as in \eqref{eqn:SEKE}.  When $E$ is the trivial line
bundle $\underline{\C}$, there is a distinguished connection
$A_{K^{-1}}$ on $K^{-1}$ such that $D_{A_{K^{-1}}}(1,0)=0$.  As in the
three-dimensional case \eqref{eqn:connections}, this allows us to
identify a spin-c connection for a general spin-c structure with a
Hermitian connection $A$ on the corresponding line bundle $E$.

Now choose exact 2-forms $\mu_\pm$ on $Y_\pm$ as in \S\ref{sec:pcf},
and let $\mu$ be an exact 2-form on $\overline{X}$ which agrees with
$\mu_\pm$ on the ends.  For the arguments later in this paper we need
to choose $\mu$ so that its derivatives up to some sufficiently large
(but constant) order have absolute value less than $1/100$.  Let $\mu_*$
denote the self-dual part of $\mu$.   We now consider,
for a connection $A$ on $E$ and a section $\psi$ of $\Sp_+$, the
following version of the four-dimensional Seiberg-Witten equations on
$\overline{X}$:
\begin{equation}
\label{eqn:tsw4}
\begin{split}
  D_A\psi &= 0,\\
  F_A^+ &= \frac{r}{2}\left(\rho(\psi)-i\hat{\omega}\right)
  -\frac{1}{2}F_{A_{K^{-1}}}^+ + i\mu_*.
\end{split}
\end{equation}
Here $r$ is a positive real number which will be taken to be very
large below.  The equations \eqref{eqn:tsw4} are equivalent to the
equations \eqref{eqn:sw4} with perturbation
\begin{equation}
\label{eqn:perturbation4}
\eta = -r\hat{\omega} + 2\mu_*,
\end{equation}
after rescaling the spinor as in \eqref{eqn:rescaling}.  On
$[\varepsilon,\infty)\times Y_+$ and $(-\infty,-\varepsilon]\times
Y_-$, if $A$ is in temporal gauge, then the equations \eqref{eqn:tsw4}
are equivalent to the perturbed instanton equations
\eqref{eqn:Tinstanton} (with a $\pm$ subscript on $\mu$).  Thus we can
use the equations \eqref{eqn:sw4} (with appropriate small abstract
perturbations) to define a chain map
\begin{equation}
\label{eqn:pcm}
\widehat{CM}^*(X,\frak{s};\lambda,J,r):
\widehat{CM}^*(Y_+,\frak{s}_+;\lambda_+,J_+,r)
\longrightarrow
\widehat{CM}^*(Y_-,\frak{s}_-;\lambda_-,J_-,r).
\end{equation}
Here $\frak{s}_\pm$ denotes the restriction of $\frak{s}\in\Spinc(X)$
to $Y_\pm$.  In general we expect the chain map \eqref{eqn:pcm} to
depend on the choice of $\mu$ (and on the choice of abstract
perturbations), although as explained in \S\ref{sec:smoothcob} the
induced map on homology does not.

\section{ECH cobordism maps}
\label{sec:cobexact}

The goal of this section is to define the maps on (filtered) ECH
induced by an exact symplectic cobordism, and to prove that they
satisfy all of the axioms in Theorem~\ref{thm:cob}, except for the
Holomorphic Curves axiom which will be proved in \S\ref{sec:hca}.

\subsection{Cobordism maps and holomorphic curves (statements)}

We now state some key properties of the map on Seiberg-Witten Floer
cohomology induced by an exact symplectic cobordism with the
Seiberg-Witten equations perturbed as in \S\ref{sec:pesc}.  To
simplify notation we henceforth ignore the decomposition via spin-c
structures, as in \eqref{eqn:scm}, although it is straightforward to
insert spin-c structures into the discussion below.

The following proposition asserts that the instantons that are used to
define the chain map \eqref{eqn:pcm} give rise to broken holomorphic
curves, and in particular respect the symplectic action filtration.
It also proves similar statements for certain chain homotopies, for
which we need the following strengthening of the notion of ``homotopy
of exact symplectic cobordisms'' defined in \S\ref{sec:SMT}:

\begin{definition}
\label{def:strongHomotopy}
Two exact symplectic cobordisms $(X,\lambda_0)$ and $(X,\lambda_1)$
from $(Y_+,\lambda_+)$ to $(Y_-,\lambda_-)$ with the same underlying
four-manifold $X$ are {\em strongly homotopic\/} if there is a smooth
one-parameter family of 1-forms $\{\lambda_t\mid t\in[0,1]\}$ on $X$
such that $(X,\lambda_t)$ is an exact symplectic cobordism from
$(Y_+,\lambda_+)$ to $(Y_-,\lambda_-)$ for each $t\in[0,1]$, and there
exists $\varepsilon>0$ such that the identifications \eqref{eqn:N-}
and \eqref{eqn:N+} for $\lambda_t$ do not depend on $t$.  
\end{definition}

Note that the last condition in the above definition ensures that the
completions \eqref{eqn:completion} of $(X,\lambda_t)$ for different
$t$ are diffeomorphic via the obvious identification.

\begin{proposition}
\label{prop:hol}
Fix $L\in\R$, closed connected contact 3-manifolds $(Y_+,\lambda_+)$
and $(Y_-,\lambda_-)$ such that $\lambda_\pm$ is $L$-nondegenerate,
symplectization-admissible almost complex structures $J_\pm$ for
$\lambda_\pm$, $2$-forms $\mu_\pm$ on $Y_\pm$ with $\mc{P}$-norm less
than $1$, and generic perturbations $\frak{p}_\pm$ on $Y_\pm$ as
needed to define the chain complexes
$\widehat{CM}^*(Y_\pm;\lambda_\pm,J_\pm,r)$.
\begin{description}
\item{(a)} Let $(X,\lambda)$ be an exact symplectic cobordism from
  $(Y_+,\lambda_+)$ to $(Y_-,\lambda_-)$.  Suppose $J$ is a strongly
  cobordism-admissible almost complex structure on $\overline{X}$
  which restricts to $J_+$ on $[0,\infty)\times Y_+$ and to $J_-$ on
  $(-\infty,0]\times Y_-$.  Let $\mu$ be a small exact 2-form on
  $\overline{X}$ extending $\mu_\pm$, and let $\frak{p}$ be a
  generic extension of $\frak{p}_\pm$ over $\overline{X}$.  Assume
  that $r$ is sufficiently large, and that $\frak{p}_\pm$ and
  $\frak{p}$ are sufficiently small for the given $r$. Let $\frak{d}$
  be a solution to the corresponding perturbed version of
  \eqref{eqn:tsw4} with index $0$ and with $\energy(\frak{c}_+)<2\pi L$.
  Then:
\begin{description}
\item{(i)}
 $\energy(\frak{c}_-)<2\pi L$.
\item{(ii)} There exists a broken $J$-holomorphic curve from
  $\Theta^+$ to $\Theta^-$, where $\Theta^\pm$ is the orbit set
  determined by $\frak{c}_\pm$ via Proposition~\ref{prop:Liso}(a).
\end{description}
\item{(b)} Let $\{(X,\lambda_t)\mid t\in[0,1]\}$ be a strong homotopy
  of exact symplectic cobordisms from $(Y_+,\lambda_+)$ to
  $(Y_-,\lambda_-)$.  Let $\{(J_t,\mu_t,\frak{p}_t)\mid t\in[0,1]\}$
  be a one-parameter family of choices as in part (a) with
  $\{\frak{p}_t\}$ generic.  Suppose that $r$ is sufficiently large
  and that $\frak{p}_\pm$ and each $\frak{p}_t$ are sufficiently small
  for the given $r$.  Let $t\in[0,1]$ and let $\frak{d}$ be a solution
  to the corresponding perturbed version of \eqref{eqn:tsw4} with
  index $-1$ and with $\energy(\frak{c}_+)<2\pi L$. Then
  $\energy(\frak{c}_-)<2\pi L$.
\end{description}
\end{proposition}

Proposition \ref{prop:hol} is proved in \S\ref{sec:holproof} below.
We can now define cobordism maps on filtered Seiberg-Witten Floer
cohomology:

\begin{corollary}
\label{cor:cobl}
Let $(X,\lambda)$ be an exact symplectic cobordism from
$(Y_+,\lambda_+)$ to $(Y_-,\lambda_-)$, where $\lambda_\pm$ is
$L$-nondegenerate.  Let $J_\pm$ be a symplectization-admissible almost
complex structure for $\lambda_\pm$.  Suppose $r$ is sufficiently large. Fix
$2$-forms $\mu_\pm$ with $\mc{P}$-norm less than $1$ and fix
sufficiently small abstract perturbations $\frak{p}_\pm$ as needed to
define the chain complexes
$\widehat{CM}^*(Y_\pm;\lambda_\pm,J_\pm,r)$.  Then there is a
well-defined map
\begin{equation}
\label{eqn:toolong}
\widehat{HM}^*_L(X,\lambda):
\widehat{HM}^*_L(Y_+;\lambda_+,J_+,r) \longrightarrow
\widehat{HM}^*_L(Y_-;\lambda_-,J_-,r),
\end{equation}
depending only on $X,\lambda,L,r,J_\pm,\mu_\pm,\frak{p}_\pm$, with the
following properties:
\begin{description}
\item{(a)}
If $L'<L$ and if $\lambda_\pm$ is also $L'$-nondegenerate,
then the diagram
\[
\begin{CD}
\widehat{HM}^*_{L'}(Y_+;\lambda_+,J_+,r)
@>{\widehat{HM}^*_{L'}(X,\lambda)}>>
 \widehat{HM}^*_{L'}(Y_-;\lambda_-,J_-,r)\\
@VVV @VVV\\
\widehat{HM}^*_{L}(Y_+;\lambda_+,J_+,r)
@>{\widehat{HM}^*_{L}(X,\lambda)}>>
 \widehat{HM}^*_{L}(Y_-;\lambda_-,J_-,r)
\end{CD}
\]
commutes, where the vertical arrows are induced by inclusions of chain complexes.
\item{(b)}
Likewise the diagram
\begin{equation}
\label{eqn:omfw}
\begin{CD}
\widehat{HM}^*_L(Y_+;\lambda_+,J_+,r)
@>{\widehat{HM}^*_L(X,\lambda)}>>
 \widehat{HM}^*_L(Y_-;\lambda_-,J_-,r)\\
@VVV @VVV\\
\widehat{HM}^*(Y_+;\lambda_+,J_+,r)
@>>> \widehat{HM}^*(Y_-;\lambda_-,J_-,r).
\end{CD}
\end{equation}
commutes, where the bottom arrow is the Seiberg-Witten Floer cobordism
map induced by \eqref{eqn:pcm}.
\item{(c)}
If $\{(X_t,\lambda_t)\mid t\in[0,1]\}$ is a strong homotopy of exact
symplectic cobordisms from $(Y_+,\lambda_+)$ to $(Y_-,\lambda_-)$, and
if $r$ is sufficiently large,
then
\[
\widehat{HM}_L^*(X,\lambda_0) = \widehat{HM}_L^*(X,\lambda_1).
\]
\end{description}
\end{corollary}

Note that for now the map $\widehat{HM}^*_L(X,\lambda)$ may depend on
$r,J_\pm,\mu_\pm,\frak{p}_\pm$, although these choices are not
indicated in the notation.  Proposition~\ref{prop:cobinv} below will
show that in fact $\widehat{HM}^*_L(X,\lambda)$ is independent of these
choices.

\begin{proof}
  Choose a strongly cobordism-admissible almost complex structure $J$
  on $\overline{X}$ extending $J_+$ and $J_-$, and choose small
  perturbations $\mu$ and $\frak{p}$ extending $\mu_\pm$ and
  $\frak{p}_\pm$, as needed to define the chain map \eqref{eqn:pcm}.
  Summing over spin-c structures, we then have a chain map
\begin{equation}
\label{eqn:hba}
\widehat{CM}^*(Y_+;\lambda_+,J_+,r) \longrightarrow
\widehat{CM}^*(Y_-;\lambda_-,J_-,r),
\end{equation}
whose induced map on homology is the bottom arrow in \eqref{eqn:omfw}.
 It follows from
Proposition~\ref{prop:hol}(a) that if $r$ is sufficiently large, and
if the perturbations are sufficiently small, then the chain map
\eqref{eqn:hba} restricts to a chain map
\begin{equation}
\label{eqn:hbal}
\widehat{CM}^*_L(Y_+;\lambda_+,J_+,r) \longrightarrow
\widehat{CM}^*_L(Y_-;\lambda_-,J_-,r).
\end{equation}
We define $\widehat{HM}^*_L(X,\lambda)$ to be the map on homology
induced by \eqref{eqn:hbal}.

We now show that $\widehat{HM}^*_L(X,\lambda)$ does not depend on the
choice of extensions $J,\mu,\frak{p}$ of $J_\pm,\mu_\pm,\frak{p}_\pm$
over $X$.  Given two choices of extensions, we can choose a homotopy
between them.  This homotopy induces a chain homotopy between the
corresponding maps \eqref{eqn:hba}, which counts index $-1$ instantons
that appear during the homotopy.  It follows from
Proposition~\ref{prop:hol}(b) that if $r$ is sufficiently large, then
this chain homotopy maps $\widehat{CM}^*_L$ to $\widehat{CM}^*_L$
(here we are ignoring the gradings as usual), and hence restricts to a chain
homotopy between the corresponding maps \eqref{eqn:hbal}.

Properties (a) and (b) above now hold by construction.  One proves
property (c) by using Proposition~\ref{prop:hol}(b) to define a chain
homotopy.
\end{proof}

The induced maps on $\widehat{HM}_L^*$ constructed above behave nicely
under composition of exact symplectic cobordisms:

\begin{proposition}
\label{prop:comp}
Suppose $(X,\lambda)$ is the composition of an exact symplectic
cobordism $(X^-,\lambda^-)$ from $(Y_0,\lambda_0)$ to
$(Y_-,\lambda_-)$ with an exact symplectic cobordism $(X^+,\lambda^+)$
from $(Y_+,\lambda_+)$ to $(Y_0,\lambda_0)$, where $\lambda_\pm$ and
$\lambda_0$ are $L$-nondegenerate.  Let $J_\pm$ and $J_0$ be
symplectization-admissible almost complex structures for $\lambda_\pm$
and $\lambda_0$. Let $\mu_\pm,\mu_0$ be $2$-forms from the $Y_\pm,Y_0$
versions of $\Omega$ with $\mc{P}$-norm less than $1$, fix $r$
sufficiently large, and let $\frak{p}_\pm,\frak{p}_0$ be sufficiently
small generic abstract perturbations as needed to define the chain
complexes $\widehat{CM}^*$.  Then the maps in Corollary~\ref{cor:cobl}
for these data satisfy
\[
\widehat{HM}^*_L(X,\lambda) =
\widehat{HM}_L^*(X^-,\lambda^-) \circ
\widehat{HM}_L^*(X^+,\lambda^+).
\]
\end{proposition}

Proposition \ref{prop:comp} is proved in \S\ref{sec:holproof} using
a neck stretching argument.

\subsection{Invariance of cobordism maps}

The goal of this subsection is to prove the following proposition,
asserting that the map $\widehat{HM}_L^*(X,\lambda)$ defined in
Corollary~\ref{cor:cobl} depends only on $X,\lambda,L$, and not on the
additional choices made in its definition.

\begin{proposition}
\label{prop:cobinv} 
Let $(X,\lambda)$ be an exact symplectic cobordism from $(Y_+,\lambda_+)$ to
$(Y_-,\lambda_-)$ where $\lambda_\pm$ is $L$-nondegenerate.  Then the
map defined in Corollary~\ref{cor:cobl} induces a well-defined map
\[
\widehat{HM}^*_L(X,\lambda): \widehat{HM}^*_L(Y_+,\lambda_+)
\longrightarrow \widehat{HM}^*_L(Y_-,\lambda_-),
\]
where $\widehat{HM}^*_L(Y_\pm,\lambda_\pm)$ is defined as in
Corollary~\ref{cor:yls}.
\end{proposition}

To prepare for the proof of this proposition, we need the following
lemma, which relates the maps on $\widehat{HM}^*_L$ induced by
exact product symplectic cobordisms to the canonical isomorphisms
between different versions of $\widehat{HM}^*_L$.

\begin{lemma}
\label{lem:technical}
Let $\rho=\{(\lambda_t,L_t,J_t,r)\mid t\in[0,1]\}$ be an admissible
deformation as in Definition~\ref{def:ad}.  Assume further that:
\begin{itemize}
\item
$\lambda_t=f_t\lambda_0$, where $f:[0,1]\times Y\to \R^{>0}$ satisfies
$\partial f/\partial t<0$ everywhere.
\item
$dL_t/dt\le 0$.
\end{itemize}
Let $(X,\lambda)$ be the product exact cobordism $([-1,0]\times
Y,\lambda_{-s})$ from $(Y,\lambda_0)$ to $(Y,\lambda_1)$.  Suppose $r$
is sufficiently large.  Fix small perturbations $\mu_i,\frak{p}_i$ for
$i=0,1$ as needed to define the chain complexes $\widehat{CM}^*$ for
$t=0,1$.  Then the cobordism map $\widehat{HM}^*_{L_0}(X,\lambda)$ in
Corollary~\ref{cor:cobl} is the composition
\[
\widehat{HM}^*_{L_0}(Y;\lambda_0,J_0,r)
\stackrel{\Phi_\rho}{\longrightarrow} \widehat{HM}^*_{L_1}(Y;\lambda_1,J_1,r)
\longrightarrow
\widehat{HM}^*_{L_0}(Y;\lambda_1,J_1,r),
\]
where $\Phi_\rho$ is the isomorphism from 
\eqref{eqn:Phirho}, and the map on the right is induced by the
inclusion of chain complexes.
\end{lemma}

\begin{proof}
The proof has two steps.

{\em Step 1.\/} We start by making choices as in the proof of
Lemma~\ref{lem:deform} to define $\Phi_\rho$.
Choose a path of data $\{(\mu_t,\frak{p}_t)\mid t\in[0,1]\}$ from
$(\mu_0,\frak{p}_0)$ to $(\mu_1,\frak{p}_1)$ where $\mu_t$ is a
$2$-form from $\Omega$ with $\mc{P}$-norm less than $1$, and
$\frak{p}_t$ is a small abstract perturbation such that the data
$D_t=(\lambda_t,J_t,r,\mu_t,\frak{p}_t)$ is suitable for defining the
chain complex $\widehat{CM}^*(Y;\lambda_t,J_t,r)$ for generic
$t\in[0,1]$.  Let $N$ be a large positive integer, and choose numbers
$0=t_0<t_1<\cdots<t_N=1$ such that $t_i-t_{i-1}<2/N$ and such that
$\widehat{CM}^*(Y,\frak{s};\lambda_{t_i},J_{t_i},r)$ is defined for
each $i=1,\ldots,N$.

To shorten the notation below, write $\widehat{HM}^*_L(t)$ to denote
$\widehat{HM}^*_L(Y;\lambda_t,J_t,r)$.  Also, for $t<t'$ let
$\widehat{HM}^*_L([t,t'])$ denote the cobordism map induced by the
portion of the cobordism parametrized by $[-t',-t]\times Y$, and let
$\rho_{[t,t']}$ denote the portion of the admissible deformation
parametrized by the interval $[t,t']$, reparametrized by the interval
$[0,1]$.

Choose $N$ sufficiently large that $\lambda_{t_i}$ has no orbit sets
of action in the interval $[L_{t_i},L_{t_{i-1}}]$ for each
$i=1,\ldots,N$.  Then for $r$ sufficiently large, the lemma holds for
the portion of the cobordism parametrized by $[-t_i,-t_{i-1}]\times
Y$.  That is, the cobordism map
\begin{equation}
\label{eqn:ICM}
\widehat{HM}^*_{L_{t_{i-1}}}([t_{i-1},t_{i}])
\end{equation}
equals the composition
\[
\begin{CD}
\widehat{HM}^*_{L_{t_{i-1}}}(t_{i-1})
@>{\Phi_{\rho_{[t_{i-1},t_i]}}}>>
\widehat{HM}^*_{L_{t_i}}(t_i) @>>>
\widehat{HM}^*_{L_{t_{i-1}}}(t_i),
\end{CD}
\]
where the map on the right is induced by the inclusion of chain
complexes.  The reason is that the map on the right is an isomorphism
on the chain level by Lemma~\ref{lem:diff}(b), so that the cobordism map
\eqref{eqn:ICM} actually maps to $\widehat{HM}^*_{L_{t_i}}(t_i)$.
Then the cobordism map \eqref{eqn:ICM}, regarded as a map to
$\widehat{HM}^*_{L_{t_i}}(t_i)$, agrees with $\Phi_{\rho_i}$ by the
definition of the latter.

{\em Step 2.\/} We now show by induction on $i$ that the lemma holds
for the portion of the cobordism parametrized by $[-t_i,0]\times Y$.
The case $i=1$ follows from Step 1.  Now let $i>1$ and suppose the
claim is true for $i-1$.  We need to show that the cobordism map
\[
\widehat{HM}^*_{L_0}([0,t_i])
\]
agrees with the composition
\[
\begin{CD}
\widehat{HM}^*_{L_0}(0)
@>{\Phi_{\rho_{[0,t_i]}}}>>
\widehat{HM}^*_{L_{t_i}}(t_i) @>>>
\widehat{HM}^*_{L_0}(t_i),
\end{CD}
\]
where the arrow on the right is induced by inclusion.

By Proposition~\ref{prop:comp} we have
\[
\widehat{HM}^*_{L_0}([0,t_i])
=
\widehat{HM}^*_{L_0}([t_{i-1},t_i]) \circ
\widehat{HM}^*_{L_0}([0,t_{i-1}]).
\]
And by Lemma~\ref{lem:deform}(b) we have
\[
\Phi_{\rho_{[0,t_i]}} = \Phi_{\rho_{[t_{i-1},t_i]}} \circ \Phi_{\rho_{[0,t_{i-1}]}}.
\]
So by the inductive hypothesis and Step 1, we just need to show that
the diagram
\[
\begin{CD}
\widehat{HM}^*_{L_0}(t_{i-1}) @>{\widehat{HM}^*_{L_0}([t_{i-1},t_i])}>>
\widehat{HM}^*_{L_0}(t_i) \\
@AAA @AAA \\
\widehat{HM}^*_{L_{t_{i-1}}}(t_{i-1}) @>{\widehat{HM}^*_{L_{t_{i-1}}}([t_{i-1},t_i])}>>
\widehat{HM}^*_{L_{t_{i-1}}}(t_i)
\end{CD}
\]
commutes, where the vertical arrows are induced by inclusion.  But
this holds by Corollary~\ref{cor:cobl}(a).
\end{proof}

We can now prove Proposition~\ref{prop:cobinv}.  The latter is an
immediate consequence of the following lemma:

\begin{lemma}
\label{lem:cobinv}
Let $(X,\lambda)$ be an exact symplectic cobordism from $(Y_+,\lambda_+)$ to
$(Y_-,\lambda_-)$ where $\lambda_\pm$ is $L$-nondegenerate.  Let
$\{J_\pm^t\mid t\in[0,1]\}$ be a one-parameter family of
symplectization-admissible almost complex structures for
$\lambda_\pm$.  Suppose that $\{r_t\mid t\in[0,1]\}$ is a sufficiently
large one-parameter family of real numbers.  Let
$(\mu_\pm^0,\frak{p}_\pm^0)$ and $(\mu_\pm^1,\frak{p}_\pm^1)$ be small
perturbations as needed to define the chain complexes $\widehat{CM}^*$
for $t=0,1$.  Then the versions of $\widehat{HM}^*_L(X,\lambda)$ for
$t=0,1$ fit into a commutative diagram
\begin{equation}
\label{eqn:cobinv}
\begin{CD}
\widehat{HM}^*_L(Y_+,\lambda_+,J_+^0,r_0) @>{\Phi_{\rho_+}}>{\simeq}>
\widehat{HM}^*_L(Y_+,\lambda_+,J_+^1,r_1)\\
@VV{\widehat{HM}^*_L(X,\lambda)_{t=0}}V
  @VV{\widehat{HM}^*_L(X,\lambda)_{t=1}}V \\
\widehat{HM}^*_L(Y_-,\lambda_-,J_-^0,r_0) @>{\Phi_{\rho_-}}>{\simeq}>
\widehat{HM}^*_L(Y_-,\lambda_-,J_-^1,r_1),
\end{CD}
\end{equation}
where $\Phi_{\rho_\pm}$ is the isomorphism from
Lemma~\ref{lem:deform}, and
\[
\rho_\pm=\{(\lambda_\pm,L,J_\pm^t,r_t)\mid t\in[0,1]\}.
\]
\end{lemma}

\begin{proof}
  First note that because $d\lambda_\pm$ is in the space $\Omega$ for
  $Y_\pm$ (see \S\ref{sec:pcf}), a small change in $r$ can be effected
  by a change in $\mu_+$ and $\mu_-$.  Thus, by the homotopy
  properties in Lemma~\ref{lem:deform}(a),(b), it is enough to prove
  the lemma in the case when $r_t$ does not depend on $t$; let us
  write $r_t=r$.

To prove the lemma for constant $r$, it is enough to do so in the
special case when $(J_-^t,\mu_-^t,\frak{p}_-^t)$ do not depend on $t$;
let us denote these by $(J_-,\mu_-,\frak{p}_-)$.  (The case when
$(J_-^t,\mu_-^t,\frak{p}_-^t)$ do depend on $t$, but
$(J_+^t,\mu_+^t,\frak{p}_+^t)$ do not, is proved by a symmetrical
argument; and these two cases together imply the general case.)

Let $\varepsilon>0$ be as in \eqref{eqn:N-} and \eqref{eqn:N+}, so
that a neighborhood of $Y_+$ in $(X,\lambda)$ is identified with
$([-\varepsilon,0]\times Y_+,e^s\lambda_+)$.  Choose $\varepsilon$
sufficiently small so that $\lambda_+$ has no orbit sets with action
in the closed interval $[e^{-\varepsilon}L,L]$.  We can decompose the
exact cobordism $X=X^0\circ X^+$, where $X^+=[-\varepsilon,0]\times
Y_+$, and $X^0$ is the closure of $X\setminus X^+$.  Make choices
$(J_0,\mu_0,\frak{p}_0)$ as needed to define the chain complex
$\widehat{CM}^*(Y_+,e^{-\varepsilon}\lambda_+,J_0,r)$.  We now have a
commutative diagram
\[
\begin{CD}
\widehat{HM}^*_L(Y_+,\lambda_+,J_+^0,r) @>{\Phi_{\rho_+}}>>
\widehat{HM}^*_L(Y_+,\lambda_+,J_+^1,r)\\
@VV{\Phi_{\rho_0}}V @VV{\Phi_{\rho_1}}V \\
\widehat{HM}^*_{e^{-\varepsilon} L}(Y_+,e^{-\varepsilon}\lambda_+,J_0,r) @=
\widehat{HM}^*_{e^{-\varepsilon}
  L}(Y_+,e^{-\varepsilon}\lambda_+,J_0,r)\\
@VVV @VVV\\
\widehat{HM}^*_{L}(Y_+,e^{-\varepsilon}\lambda_+,J_0,r) @=
\widehat{HM}^*_{
  L}(Y_+,e^{-\varepsilon}\lambda_+,J_0,r)\\
@VV{\widehat{HM}^*_L(X^0,\lambda)}V
  @VV{\widehat{HM}^*_L(X^0,\lambda)}V \\
\widehat{HM}^*_L(Y_-,\lambda_-,J_-,r) @=
\widehat{HM}^*_L(Y_-,\lambda_-,J_-,r).
\end{CD}
\]
Here
\[
\rho_0=\{(e^{-\varepsilon t}\lambda_+,e^{-\varepsilon t}L,\hat{J}_t,r)\mid
t\in[0,1]\}
\]
where $\{\hat{J}_t\mid t\in[0,1]\}$ is a path of almost complex
structures from $J_+^0$ to $J_0$.  The admissible deformation $\rho_1$
is defined analogously.  The top square in the diagram commutes by
Lemma~\ref{lem:deform}(a),(b).  The vertical arrows in the middle of
the diagram are induced by the inclusion of chain complexes.  By
Lemma~\ref{lem:technical}, the composition of the two vertical arrows
on the upper left is the cobordism map
$\widehat{HM}^*_L(X^+,e^s\lambda_+)$ defined by
Corollary~\ref{cor:cobl} from the choices $(J_+^0,\mu_+,\frak{p}_+,r)$
and $(J_0,\mu_0,\frak{p}_0,r)$.  Then by Proposition~\ref{prop:comp},
the composition of the three vertical arrows on the left is
$\widehat{HM}^*_L(X,\lambda)_{t=0}$.  Likewise, the composition of the
three vertical arrows on the right is
$\widehat{HM}^*_L(X,\lambda,)_{t=1}$.  Thus the above diagram gives the
desired commutative diagram \eqref{eqn:cobinv}.
\end{proof}

This completes the proof of Proposition~\ref{prop:cobinv}.

We also note the following special case of Lemma~\ref{lem:cobinv},
which is needed in \S\ref{sec:ccm}:

\begin{corollary}
\label{cor:product}
Suppose $X$ is a product cobordism $([-\varepsilon,0]\times
Y,e^s\lambda)$ where $\varepsilon>0$, the variable $s$ denotes the
$[-\varepsilon,0]$ coordinate, and $\lambda$ is an $L$-nondegenerate
contact form on $Y$.  Then $\widehat{HM}^*_L(X,e^s\lambda)$ is the
composition
\[
\widehat{HM}^*_L(Y,\lambda) \stackrel{s}{\longrightarrow}
\widehat{HM}^*_{e^{-\varepsilon}L}(Y,e^{-\varepsilon}\lambda)
\longrightarrow \widehat{HM}^*_L(Y,e^{-\varepsilon}\lambda),
\]
where $s$ is the scaling isomorphism from Corollary~\ref{cor:yls}(c),
and the right arrow is the inclusion-induced map from
Corollary~\ref{cor:yls}(b).
\end{corollary}

\begin{proof}
  Choose a symplectization-admissible almost complex structure $J$ for
  $\lambda$, and let $r$ be large.  The claim then follows by applying
  Lemma~\ref{lem:technical} to the admissible deformation $\rho =
  \{(e^{-\varepsilon t}\lambda,e^{-\varepsilon t}L,J,r)\mid
  t\in[0,1]\}$, because $\Phi_\rho$ agrees with the scaling
  isomorphism $s$ by the definition of the latter in the proof of
  Corollary~\ref{cor:yls}(c).
\end{proof}

\subsection{Construction of ECH cobordism maps}
\label{sec:ccm}

We now begin the proof of Theorem~\ref{thm:cob} by defining the map
\eqref{eqn:PhiL} on filtered ECH induced by an exact symplectic
cobordism.

Let $(X,\lambda)$ be an exact symplectic cobordism from
$(Y_+,\lambda_+)$ to $(Y_-,\lambda_-)$, and assume that $\lambda_+$
and $\lambda_-$ are nondegenerate.  Fix a real number $L$.
Without loss of generality, we can assume (by slightly decreasing $L$
if necessary) that $\lambda_+$ and $\lambda_-$ do not have any orbit
sets of action exactly $L$.  By Proposition~\ref{prop:cobinv}, we have a
well-defined map
\begin{equation}
\label{eqn:wdm}
\widehat{HM}^{*}_L(X,\lambda):
\widehat{HM}^{*}_L(Y_+,\lambda_+) \longrightarrow
\widehat{HM}^{*}_L(Y_-,\lambda_-).
\end{equation}
On the other hand, by Lemma~\ref{lem:FI} we have canonical isomorphisms
\begin{equation}
\label{eqn:FI1}
ECH_*^L(Y_\pm,\lambda_\pm) \stackrel{\simeq}{\longrightarrow}
\widehat{HM}^{-*}_L(Y_\pm,\lambda_\pm).
\end{equation}

\begin{definition}
Define a map of ungraded $\Z/2$-modules
\begin{equation}
\label{eqn:PhiL1}
\Phi^L(X,\lambda): ECH_*^L(Y_+,\lambda_+) \longrightarrow
ECH_*^L(Y_-,\lambda_-)
\end{equation}
to be the composition of the map \eqref{eqn:wdm} with the isomorphisms
\eqref{eqn:FI1}.
\end{definition}

We now prove all of Theorem~\ref{thm:cob} except for the Holomorphic Curves axiom:

\begin{proposition}
\label{prop:cob1}
The map \eqref{eqn:PhiL1} satisfies the Homotopy Invariance,
Inclusion, Direct Limit, Composition, and Scaling axioms in
Theorem~\ref{thm:cob}.
\end{proposition}

\begin{proof}
  The Inclusion axiom follows from Lemma~\ref{lem:FI}(a) and
  Corollary~\ref{cor:cobl}(a).  The Direct Limit axiom follows from
  Corollary~\ref{cor:cobl}(b).  The Composition axiom folows from
  Proposition~\ref{prop:comp}.

  To prove the Homotopy Invariance axiom, let $\{(X,\omega_t)\mid
  t\in[0,1]\}$ be a homotopy of exact symplectic cobordisms, let
  $\lambda_0$ be a Liouville form for $\omega_0$, and let $\lambda_1$
  be a Liouville form for $\omega_1$.  We need to show that
  $\Phi^L(X,\lambda_0) = \Phi^L(X,\lambda_1)$.  Since the space of
  Liouville forms for a given exact symplectic cobordism $(X,\omega)$
  is affine linear, there is no obstruction to connecting $\lambda_0$
  and $\lambda_1$ by a smooth one-parameter family $\{\lambda_t\mid
  t\in[0,1]\}$ of $1$-forms on $X$ such that $\lambda_t$ is a
  Liouville form for $\omega_t$.  Next, fix $\varepsilon>0$ such that
  there are disjoint neighborhoods \eqref{eqn:N-} and \eqref{eqn:N+}
  for each $\lambda_t$.  We can then find a smooth one-parameter
  family $\{\varphi_t\mid t\in[0,1]\}$ of diffeomorphisms of $X$ with
  $\varphi_0=\op{id}_X$ and $\varphi_t |_{\partial
    X}=\op{id}_{\partial X}$ such that $\varphi_t$ pulls back the
  neighborhoods \eqref{eqn:N-} and \eqref{eqn:N+} for $\lambda_t$ to
  those for $\lambda_0$.  Then $\{(X,\varphi_t^*\lambda_t)\mid
  t\in[0,1]\}$ is a strong homotopy from $(X,\lambda_0)$ to
  $(X,\varphi_1^*\lambda_1)$ as in
  Definition~\ref{def:strongHomotopy}.  By
  Corollary~\ref{cor:cobl}(c),
  $\Phi^L(X,\lambda_0)=\Phi^L(X,\varphi_1^*\lambda_1)$.  Now the
  diffeomorphism $\varphi_1$ extends to a symplectomorphism between
  the completions \eqref{eqn:completion} of $(X,\lambda_1)$ and
  $(X,\varphi_1^*\lambda_1)$, and so by construction
  $\Phi^L(X,\varphi_1^*\lambda_1)=\Phi^L(X,\lambda_1)$.

  To prove the Scaling axiom, let
  $(X,\lambda)$ be an exact symplectic cobordism from
  $(Y_+,\lambda_+)$ to $(Y_-,\lambda_-)$, where the contact forms
  $\lambda_\pm$ are nondegenerate and have no orbit sets of action
  $L$.  Write $c=e^{-\varepsilon}$, and assume without loss of
  generality that $\varepsilon>0$.  We need to show that the upper
  square in the diagram
\[
\begin{CD}
ECH_*^L(Y_+,\lambda_+) @>{\Phi^L(X,\lambda)}>> ECH_*^L(Y_-,\lambda_-)\\
@V{s}VV @V{s}VV \\
ECH_*^{e^{-\varepsilon}L}(Y_+,e^{-\varepsilon}
\lambda_+) @>{\Phi^{e^{-\varepsilon}L}(X,e^{-\varepsilon}\lambda)}>>
ECH_*^{e^{-\varepsilon}L}(Y_-,e^{-\varepsilon}\lambda_-)\\
@V{\imath^{e^{-\varepsilon}L,L}}VV
@V{\imath^{e^{-\varepsilon}L,L}}VV\\
ECH_*^{L}(Y_+,e^{-\varepsilon}
\lambda_+) @>{\Phi^{L}(X,e^{-\varepsilon}\lambda)}>>
ECH_*^{L}(Y_-,e^{-\varepsilon}\lambda_-)
\end{CD}
\]
commutes, where $s$ denotes the scaling isomorphism \eqref{eqn:s} for
$c=e^{-\varepsilon}$.

Since the composition of two scaling isomorphisms is a scaling
isomorphism, we may assume without loss of generality that
$\varepsilon$ is sufficiently small so that $\lambda_\pm$ has no orbit
set of action in the interval $[L,e^{\varepsilon}L]$.  Then the lower
vertical arrows in the above diagram are isomorphisms by
Lemma~\ref{lem:diff}(b).  Also, by the Inclusion axiom, the lower
square commutes.  So to prove that the upper square commutes, it is
enough to show that the outer rectangle commutes.

For this purpose consider the product exact cobordisms
$(X^+=[-\varepsilon,0]\times Y_+,e^s\lambda_+)$ and
$(X^-=[-\varepsilon,0]\times Y_-,e^s\lambda_-)$.  By
Corollary~\ref{cor:product} and Lemma~\ref{lem:FI}(a),(b), the
compositions of the vertical arrows in the above diagram are given by
\begin{equation}
\label{eqn:technical}
\Phi^L(X^\pm,e^s\lambda_\pm)=\imath^{e^{-\varepsilon}L,L}\circ s:
ECH_*^L(Y_\pm,\lambda_\pm) \longrightarrow
ECH_*^L(Y_\pm,e^{-\varepsilon}\lambda_\pm).
\end{equation}
So to prove that the outer rectangle in the above diagram commutes, it
is enough to prove that the square
\[
\begin{CD}
ECH_*^L(Y_+,\lambda_+) @>{\Phi^L(X,\lambda)}>> ECH_*^L(Y_-,\lambda_-)\\
@V{\Phi^L(X^+,e^s\lambda_+)}VV @V{\Phi^L(X^-,e^s\lambda_-)}VV \\
ECH_*^{L}(Y_+,e^{-\varepsilon}
\lambda_+) @>{\Phi^{L}(X,e^{-\varepsilon}\lambda)}>>
ECH_*^{L}(Y_-,e^{-\varepsilon}\lambda_-)
\end{CD}
\]
commutes.  By the Composition axiom, this is equivalent to the
assertion that
\begin{equation}
\label{eqn:etat}
\Phi^L((X,e^{-\varepsilon}\lambda)\circ(X^+,e^s\lambda_+)) =
\Phi^L((X^-,e^s\lambda_-)\circ (X,\lambda)).
\end{equation}
But these two compositions of exact symplectic cobordisms are
homotopic through exact symplectic cobordisms from $(Y_+,\lambda_+)$
to $(Y_-,e^{-\varepsilon}\lambda_-)$ if $\varepsilon$ is sufficiently
small as in \eqref{eqn:N-} and \eqref{eqn:N+}.  Thus equation \eqref{eqn:etat}
follows from the Homotopy axiom.
\end{proof}

\section{Proof of the holomorphic curves axiom}
\label{sec:hca}

Let $(X,\lambda)$ be an exact symplectic cobordism as in the statement
of Theorem~\ref{thm:cob}.  To complete the proof of
Theorem~\ref{thm:cob}, we now prove that the maps $\Phi^L(X,\lambda)$
defined in \S\ref{sec:ccm} satisfy the Holomorphic Curves axiom.  For
this purpose fix a cobordism-admissible almost complex structure
$J$ on $\overline{X}$ as in the statement of the Holomorphic Curves
axiom.  Let $J_\pm$ denote the symplectization-admissible almost
complex structure that $J$ determines on $\R\times Y_\pm$, and assume
that this is $ECH^L$-admissible.

In the analysis in this section, we adopt the convention
that $c$ denotes a positive constant whose value may increase from one
appearance to the next.

\subsection{Reduction to the strongly cobordism-admissible case}
\label{sec:rsa}

The first step in the proof of the Holomorphic Curves axiom is to
reduce to the case where $J$ is strongly cobordism-admissible, see
Definition~\ref{def:sca}.  The latter condition ensures that the chain
map \eqref{eqn:hbal} is defined, and will also be convenient in
\S\ref{sec:hcg}.

\begin{lemma}
To prove the Holomorphic Curves axiom, it suffices to prove it in the
special case when $J$ is strongly cobordism-admissible.
\end{lemma}

\begin{proof}
  Assume that the Holomorphic Curves axiom is true in the strongly
  cobordism-admissible case, and let $J$ be any (not necessarily
  strongly) cobordism-admissible almost complex structure.  Fix $L$
  such that $\lambda_+$ and $\lambda_-$ have no ECH generators of
  action exactly $L$.  Choose $\varepsilon>0$ sufficiently small that
  $\lambda_+$ has no ECH generators with action in the interval
  $[L,e^\varepsilon L]$, and $\lambda_-$ has no ECH generators with
  action in the interval $[e^{-\varepsilon} L,L]$.  Define an exact
  cobordism
\[
X'=([-\varepsilon,0]\times Y_-,e^s\lambda_-)\circ X \circ
([0,\varepsilon]\times Y_+, e^s\lambda_+)
\]
from $(Y_+,e^\varepsilon\lambda_+)$ to
$(Y_-,e^{-\varepsilon}\lambda_-)$.

We use the cobordism-admissible almost complex structure $J$ on
$\overline{X}$ to define a strongly cobordism-admissible almost
complex structure $J'$ on $\overline{X'}$ as follows.  Note that there
is a natural identification $\overline{X'}=\overline{X}$, sending
$(-\infty,0]\times Y_-$ and $[0,\infty)\times Y_+$ in $\overline{X'}$
to $(-\infty,-\varepsilon]\times Y_-$ and $[\varepsilon,\infty)\times
Y_+$ in $\overline{X}$ .  Under this identification, the almost
complex structure $J$ on $\overline{X}$ is not quite
cobordism-admissible for $\overline{X'}$, because on the ends
$J(\partial_s)$ is not the Reeb vector field, but rather the Reeb
vector field times $e^{\pm\varepsilon}$.  To repair this defect,
choose a diffeomorphism $\varphi_+:[0,\infty)\to [0,\infty)$ such that
$\varphi_+(s)=s$ for $s$ close to $0$, and
$d\varphi_+(s)/ds=e^{-\varepsilon}$ for $s\ge \varepsilon/2$. Likewise choose a diffeomorphism $\varphi_-:(-\infty,0]\to (-\infty,0]$ such that $\varphi_-(s)=s$ for $s$ close to $0$, and $d\varphi_-(s)/ds = e^\varepsilon$ for $s\le -\varepsilon/2$.
 Define a
diffeomorphism $\phi:\overline{X}\to\overline{X}$ by setting
$\phi|_X=\op{id}_X$ and defining $\phi(s,y)=(\varphi_+(s),y)$ on
$[0,\infty)\times Y_+$ and $\phi(s,y)=(\varphi_-(s),y))$ on
$(-\infty,0]\times Y_-$.  Now $J'\eqdef \phi^*J$ is a strongly
cobordism-admissible almost complex structure on $\overline{X'}$.
Furthermore, product regions for $X'$ with respect to $J'$ correspond
to product regions for $X$ with respect to $J$.

By hypothesis, the Holomorphic Curves axiom holds for $(X',J')$, so
there is a chain map
\[
\hat{\Phi}':ECC_*^L(Y_+,e^\varepsilon\lambda_+,J_+)
\longrightarrow ECC_*^
  L(Y_-,e^{-\varepsilon}\lambda_-,J_-)
\]
which induces $\Phi^L(X')$ and fulfills conditions (i) and (ii) in the
Holomorphic Curves axiom.  To deduce the Holomorphic Curves axiom for
$(X,J)$, define a chain map $\hat{\Phi}$ by composing $\hat{\Phi}'$
with the composition
\begin{equation}
\label{eqn:left}
ECC_*^L(Y_+,\lambda_+,J_+) \longrightarrow
ECC_*^{e^\varepsilon L}(Y_+,e^\varepsilon\lambda_+,J_+) \longrightarrow
ECC_*^L(Y_+,e^\varepsilon\lambda_+,J_+)
\end{equation}
on the left, and the composition
\begin{equation}
\label{eqn:right}
ECC_*^L(Y_-,e^{-\varepsilon}\lambda_-,J_-) \longrightarrow
  ECC_*^{e^\varepsilon L}(Y_-,\lambda_-,J_-) \longrightarrow ECC_*^L(Y_-,\lambda_-,J_-)
\end{equation}
on the right.  In each of \eqref{eqn:left} and \eqref{eqn:right}, the
left arrow is the scaling isomorphism, and the right arrow is the
inverse of the map induced by inclusion of chain complexes (which is
an isomorphism since we chose $\varepsilon$ sufficiently small).  Each
of the compositions \eqref{eqn:left} and \eqref{eqn:right} is the
obvious geometric identification of generators, and so since
$\hat{\Phi}'$ satisfies conditions (i) and (ii) in the Holomorphic
Curves axiom, it follows that $\hat{\Phi}$ satisfies these conditions
as well.  Finally, it follows from \eqref{eqn:technical} and the
Composition axiom that $\hat{\Phi}$ induces the map $\Phi^L(X)$, as
required.
\end{proof}

Assume henceforth that $J$ is strongly cobordism-admissible.

\subsection{The $L$-flat case}
\label{sec:hcl}

We now prove the Holomorphic Curves axiom in the special
case when $(\lambda_+,J_+)$ and $(\lambda_-,J_-)$ are $L$-flat.  In
this case, define a chain map
\begin{equation}
\label{eqn:flf}
\hat{\Phi}:ECC_*^L(Y_+,\lambda_+;J_+) \longrightarrow
ECC_*^L(Y_-,\lambda_-;J_-)
\end{equation}
by composing the chain map \eqref{eqn:hbal} for $r$ large (and some
choice of small $2$-form $\mu$ and small abstract perturbation $\frak{p}$)
with the isomorphisms of chain complexes \eqref{eqn:Liso} on both
sides.

\begin{proposition}
\label{prop:HCLF}
If $(\lambda_+,J_+)$ and $(\lambda_-,J_-)$ are $L$-flat, then there
exists an abstract perturbation $\frak{p}$ such that the chain map
$\hat{\Phi}$ in \eqref{eqn:flf} fulfills the Holomorphic Curves axiom.
\end{proposition}

\begin{proof}
  We need to show that conditions (i) and (ii) in the Holomorphic
  curves axiom hold for this $\hat{\Phi}$.  It follows immediately
  from Proposition~\ref{prop:hol}(a) that condition (i) holds.

  To prove (ii), suppose that $\Theta^+$ is a generator of
  $ECC_*^L(Y_+,\lambda_+;J_+)$ in a product region, and let $\Theta^-$
  denote the corresponding generator of $ECC_*^L(Y_-,\lambda_-;J_-)$.  Let
  $C_\Theta\in\mc{M}^J(\Theta^+,\Theta^-)$ denote the corresponding
  union of product cylinders.  We need the following proposition:

\begin{proposition}
\label{prop:dtheta}
Suppose that $(\lambda_+,J_+)$ and $(\lambda_-,J_-)$ are $L$-flat, and
let $\Theta^\pm$, $C_\Theta$ be as above.  If $r$ is sufficiently
large then:
\begin{description}
\item{(a)} There exists a solution
  $\frak{d}_\Theta=(A_\Theta,\psi_\Theta)$ to the equations
  \eqref{eqn:tsw4} (without abstract perturbation) such that on the
  ends, $\lim_{s\to\pm\infty}\frak{d}_\Theta$ corresponds to
  $\Theta^\pm$ via \eqref{eqn:Liso}.
\item{(b)} The operator $D_{\frak{d}_\Theta}$ obtained from
  linearizing the equations \eqref{eqn:tsw4} at $\frak{d}_\Theta$
  (this is the $\frak{p}=0$ case of the operator in \cite[Eq.\
  (3.9)]{e1}) has index $0$ and trivial cokernel.
\item{(c)}
If $C_\Theta$ is the only broken $J$-holomorphic curve from
$\Theta^+$ to $\Theta^-$, then the instanton $\frak{d}_\Theta$ in (a)
is unique up to gauge equivalence.
\end{description}
\end{proposition}

\begin{proof} 
  (a) If $\Theta^\pm$ is the empty set, then this is proved similarly
  to \cite[Prop.\ 4.3]{tw2}, giving a solution
  $\frak{d}_\emptyset=(A_\emptyset,\psi_\emptyset)$ such that
  $\lim_{s\to\pm\infty}\frak{d}_\emptyset$ corresponds to the empty
  set via \eqref{eqn:Liso}, and $|\psi_0|\ge 1 - \kappa r^{-1}$ and
  $|F_{A_{0}}|\le \kappa$ everywhere for some $r$-independent constant
  $\kappa$.

  In the general case, this is proved by repeating the construction in
  \cite[\S4--7]{e2} with cosmetic changes.  We now briefly summarize
  this construction.

  One starts as in \cite[\S5a]{e2} by building a complex line bundle
  $E$ over $\overline{X}$ and a pair $(A^*,\psi^*)$ consisting of a
  connection on $E$ and a section of ${\mathbb S}_+=E\oplus K^{-1}E$
  (see \eqref{eqn:seke4}) that are close to solving \eqref{eqn:tsw4}.
  The bundle $E$ is such that there is a section of $E$ whose zero set
  with multiplicity is given by $C_\Theta$.  On the complement of a
  small radius neighborhood of $C_\Theta$, the bundle $E$ is
  identified with the trivial line bundle, and $(A^*,\psi^*)$ is close
  to the instanton $(A_\emptyset,\psi_\emptyset)$ constructed above.
  Near a product cylinder $\R\times\gamma$, where $(\gamma,m)$ is an
  element of the orbit set $\Theta^\pm$, the pair $(A^*,\psi^*)$ is
  determined by a map $\frak{v}:\R\times\gamma\to\frak{C}_m$.  Here
  $\frak{C}_m$ denotes the moduli space of degree $m$ vortices on
  $\C$.  The space $\frak{C}_m$ is naturally diffeomorphic to $\C^m$
  with coordinates $(\sigma_1,\ldots,\sigma_m)$; see \cite[\S2]{e2}
  for details.  The map $\frak{v}$ is required to be asymptotic to
  $0\in\C^m$ when the $\R$ coordinate $s$ goes to $\pm\infty$.  It is
  also required to satisfy a certain nonlinear Cauchy-Riemann
  equation.  For each collection of maps $\frak{v}$ satisfying these
  conditions, a gluing construction in \cite[\S5]{e2} then perturbs
  the corresponding pair $(A^*,\psi^*)$ to an instanton\footnote{The
    argument in \cite{e2} is complicated in order to handle
    non-$\R$-invariant holomorphic curves having multiple ends
    converging to (covers of) the same Reeb orbit, or ends converging
    to multiple covers of a Reeb orbit.  For the union of product
    cylinders $C_\Theta$, one can avoid these complications and
    instead use (with appropriate cosmetic changes) the simpler
    construction in \cite{grtosw}, which constructs a Seiberg-Witten
    solution from a holomorphic torus with self-intersection number
    zero in a closed symplectic manifold.}.

  When $m=1$, it turns out that the unique solution $\frak{v}$ for
  $(\gamma,m)$ is given by the constant map $\R\times\gamma\to
  0\in\C$.  If $m>1$, and thus $\gamma$ is elliptic, let $T$ denote the
  symplectic action of $\gamma$. Then the $L$-flatness and ``product
  region'' assumptions imply that a neighborhood of $\R\times\gamma$
  can be identified with $\R\times(\R/T\Z)\times \C$ with coordinates
  $s,t,z$, so that the Reeb vector field is given by $R=\partial_t$,
  and the Liouville form $\lambda$ is given by
\[
\lambda = e^s\left(\left(1-\frac{2\pi {\rm R}}{T}|z|^2\right)dt +
  \frac{i}{2}\left(zd\zbar - \zbar dz\right)\right).
\]
Here ${\rm R}$ is a constant, the ``rotation number'', which is
irrational by the nondegeneracy assumptions.  Meanwhile,
$J\partial_s=f(s)\partial_t$ and $J\partial_z=i\partial_z$, where the
function $f(s)$ is positive and equal to constants when $s>>0$ and
$s<<0$.  (This is only slightly different from the symplectization
context of \cite{e2}, where one would have $f\equiv 1$.) Now the key
point is that in this case, similarly to \cite{e2}, the nonlinear
Cauchy-Riemann equation that $\frak{v}=(\sigma_1,\ldots,\sigma_m)$ has
to satisfy is linear, namely the equation
\[
(f^{-1}\partial_s + i\partial_t)\sigma_q + \frac{2\pi Rq}{T}
\sigma_q=0.
\]
Thus this equation has a (unique) solution $\frak{v}$ with the
required asymptotics $\lim_{s\to\pm\infty}\frak{v}(s,\cdot)=0$, namely
$\frak{v}\equiv 0$.  In conclusion, we obtain a (unique) solution
$\frak{v}$ for each product cylinder $(\R\times\gamma,m)$ in
$C_\Theta$, and this gives rise to the desired instanton.

  (b) This follows similarly to the proof of nondegeneracy in
  \cite[Thm.\ 1.2]{e3}.

  (c) Suppose that $C_\Theta$ is the only broken $J$-holomorphic curve
  from $\Theta^+$ to $\Theta^-$.  We need to show that if $r$ is
  sufficiently large then $\frak{d}_\Theta$ is the unique (up to gauge
  equivalence) solution to \eqref{eqn:tsw4} such that
  $\lim_{s\to\pm\infty}\partial_\Theta$ corresponds to $\Theta^\pm$
  via \eqref{eqn:Liso}.

  Suppose $\frak{d}'=(A',\psi')$ is another such instanton.  First
  observe that for any $\delta>0$, if $r$ is sufficiently large then
  we must have $1-|\psi'|<\delta$ on the complement of the radius
  $\delta$ neighborhood of $C_\Theta$.  Otherwise
  Proposition~\ref{prop:9x} would imply that there is a ``generalized
  broken $J$-holomorphic curve'' (see \S\ref{sec:9xprelim}) from
  $\Theta^+$ to $\Theta^-$ containing a point not on $C_\Theta$,
  contradicting our hypothesis.

  Using the above observation, the arguments in \cite[\S6]{e4} carry
  over\footnote{As in the proof of part (a), the argument needed here
    can be simplified from that in \cite{e4} and differs only
    cosmetically from analogous arguments in \cite{grtosw}.} to show
  that $\frak{d}'$ is gauge equivalent to $\frak{d}_\Theta$.  The idea
  is that $\frak{d}'$ must be obtained from the gluing construction in
  (a), and uniqueness for the instanton then follows because the
  solutions $\frak{v}$ to their respective nonlinear Cauchy-Riemann
  equations are unique and cut out transversely.
\end{proof}

To complete the proof of Proposition~\ref{prop:HCLF},
we need to show that if $C_\Theta$ is the only broken $J$-holomorphic
curve from $\Theta^+$ to $\Theta^-$, then
\begin{equation}
\label{eqn:phiphiphi}
\langle\hat{\Phi}\Theta^+,\Theta^-\rangle=1\in\Z/2.
\end{equation}

Let $\frak{c}_\pm$ denote the Seiberg-Witten Floer generator
corresponding to $\Theta^\pm$ via \eqref{eqn:Liso}.  Recall that to
define the chain map \eqref{eqn:flf}, one fixes small abstract
perturbations $\frak{p}_\pm$ as necessary to define the Seiberg-Witten
chain complexes for $Y_\pm$, and extends these perturbations over
$\overline{X}$ as in \cite[Ch.\ 24]{km} to obtain a small perturbation
$\frak{p}$ as needed to perturb the equations \eqref{eqn:tsw4}.
Recall from \S\ref{sec:swfc} that we choose the perturbations
$\frak{p}_\pm$ so that $\frak{c}_\pm$ are still solutions to the
perturbed version of the Seiberg-Witten equations \eqref{eqn:sw3}.
Likewise the perturbation $\frak{p}$ can be chosen to vanish to second
order on the instantions $\frak{d}_\Theta$ given by
Proposition~\ref{prop:dtheta}, so that these are transverse solutions
to the perturbed version of the instanton equations \eqref{eqn:tsw4}.
A limiting argument similar to Step 2 of the proof of
Proposition~\ref{prop:hol} in \S\ref{sec:finalproofs} now shows that
if $r$ is sufficiently large and if $\frak{p}_\pm$ and $\frak{p}$ are
sufficiently small, then any other solution $\frak{d}'$ to the
corresponding perturbed version of \eqref{eqn:tsw4} with
$\lim_{s\to\pm}\frak{d}'=c_{\pm}$ is gauge equivalent to
$\frak{d}_{\Theta}$.  It follows that \eqref{eqn:phiphiphi} holds as
desired.
\end{proof}

\subsection{The non-$L$-flat case}
\label{sec:hcg}

To prove the Holomorphic Curves axiom in the non-$L$-flat case, we
reduce to the $L$-flat case by defining a sequence of modified exact
symplectic cobordisms $\{(X,\lambda_n)\}_{n=1,2,\ldots}$ between
$L$-flat pairs, equipped with strongly cobordism-admissible almost
complex structures $J_n$, such that $(\lambda_n,J_n)$ converges to
$(\lambda,J)$ in an appropriate sense.  Fix $\varepsilon>0$ as in
Definition~\ref{def:sca}.  We can then write $X=X^-\circ X^0 \circ
X^+$ where $X^-=[0,\varepsilon]\times Y_-$ and
$X^+=[-\varepsilon,0]\times Y_+$, and $\lambda|_{X^\pm} =
e^s\lambda_\pm$.  Here $s$ denotes the $[-\varepsilon,0]$ or
$[0,\varepsilon]$ coordinate as usual.  The idea of the construction
is to define $(\lambda_n,J_n)$ by suitably modifying $(\lambda,J)$ on
$X^\pm$, and in neighborhoods of product regions.  The construction
has four steps.

{\em Step 1.\/} To begin the construction, fix a positive integer $n$.
Let $U_{\pm}$ denote the set of points in $Y_\pm$ that are
within distance $1/n$ of a Reeb orbit with action less than $L$, as
measured using some arbitrary metrics on $Y_\pm$.  By
Lemma~\ref{lem:plfa}(c), there exists a preferred homotopy
$\{(\lambda_{\pm}^t,J_{\pm}^t)\mid t\in[0,1]\}$ on $Y_\pm$ where
$(\lambda_{\pm}^0,J_{\pm}^0)=(\lambda_\pm,J_\pm)$, the pair
$(\lambda_{\pm}^1,J_{\pm}^1)$ is $L$-flat, and
$(\lambda_{\pm}^t,J_{\pm}^t)$ agrees with $(\lambda_\pm,J_\pm)$ on
$Y_\pm\setminus U_{\pm}$.  To ensure smooth gluing below, let us
reparametrize the homotopy so that
$(\lambda_{\pm}^t,J_{\pm}^t)=(\lambda_\pm,J_\pm)$ for $t$ in a
neighborhood of $0$, and
$(\lambda_{\pm}^t,J_{\pm}^t)=(\lambda_\pm^1,J_\pm^1)$ for $t$ in a
neighborhood of $1$.  Also, we can assume that if a component of
$U_\pm$ is contained in $Z$ where $[s_-,s_+]\times Z$ is a product
region, then
\[
e^{-s_+}\lambda_{+}^t = e^{-s_-}\lambda_{-}^t,
\quad\quad
J_{+}^t|_{\Ker(\lambda_0)} =J_{-}^t|_{\Ker(\lambda_0)}
\]
on this component.

Keep in mind that $(\lambda_\pm^t,J_\pm^t)$ depends on $n$, although
we do not indicate this in the notation.  We will need the following
estimates on this $n$-dependence:

\begin{lemma}
\label{lem:gray}
There exists an $n$-independent constant $c>0$ such that the homotopy
$\{(\lambda_\pm^t,J_\pm^t)\}$ above can be chosen so that
\begin{align}
\label{eqn:gray1}
\left\|\frac{\partial\lambda_\pm^t}{\partial t}\right\|_{C^1},
\left\|\frac{\partial J_\pm^t}{\partial t}\right\|_{C^0} & \le
c n^{-1},\\
\label{eqn:gray2}
\left\|\frac{\partial\lambda_\pm^t}{\partial
    t}\right\|_{C^2}, \left\|\frac{\partial J_\pm^t}{\partial
    t}\right\|_{C^1} &\le c.
\end{align}
\end{lemma}

\begin{proof}
  Let $\gamma:\R/T\Z\to Y_\pm$ be a Reeb orbit of action less than
  $L$.  Recall from \cite[Eq.\ (2-1)]{e1} that there exists a disk
  $D\subset\C$ about the origin and an extension of $\gamma$ to an
  embedding $\varphi:(\R/T\Z)\times D\to Y_\pm$ such that:
\begin{itemize}
\item
If $t$
  denotes the $\R/T\Z$ coordinate and $z$ denotes the $\C$ coordinate,
  then
\begin{equation}
\label{eqn:2-1}
\varphi^*\lambda_\pm^0 = (1-2\nu|z|^2-\mu\zbar^2-\overline{\mu}z^2)dt +
\frac{i}{2}(zd\zbar - \zbar dz) + \cdots
\end{equation}
where $\nu$ and $\mu$ are respectively real and complex valued
functions on $\R/T\Z$, and the unwritten terms are $O(|z|^3)$.
\item At $z=0$, the restriction of $J_\pm$ to $\xi$ is the
  standard almost complex structure on $\C$.
\end{itemize}

By \cite[Eq.\ (2-11)]{e1}, $\lambda^1_\pm$ differs from
$\lambda_\pm^0$ only in the $\mu$ terms and higher order terms in
\eqref{eqn:2-1}, and these differences occur only where $|z|\le c/n$.
It follows that $\lambda_\pm^1-\lambda_\pm^0$ satisfies the $C^1$ and
$C^2$ bounds in \eqref{eqn:gray1} and \eqref{eqn:gray2}, and because
of the way a preferred homotopy is constructed in \cite[App.\ A]{e1},
$\partial\lambda_\pm^t/\partial t$ also satisfies these bounds.

It also follows from \cite[Eq.\ (2-11)]{e1} and the second bullet
point above that $J_\pm^1$ and $J_\pm^0$ agree along $\gamma$, and
therefore their difference is $O(|z|)$.  Since their difference is
supported where $|z|\le c/n$, it follows from this and the cutoff
construction of $J_\pm^1$ in \cite[App.\ A]{e1} that $J_\pm^1-J_\pm^0$
satisfies the $C^0$ and $C^1$ bounds in \eqref{eqn:gray1} and
\eqref{eqn:gray2}.  It then follows from the construction of a
preferred homotopy that $\partial J_\pm^t/\partial t$ also satisfies
these bounds.
\end{proof}

As a first step to defining $\lambda_n$, define a
$1$-form $\lambda_n'$ on $X$ by
\[
\lambda_n' \eqdef \left\{\begin{array}{cl}
    e^s\lambda_+^{1+\varepsilon^{-1}s} & \mbox{on
      $X^+=[-\varepsilon,0]\times Y_+$},\\
\lambda & \mbox{on $X^0$},\\
e^{s}\lambda_-^{1-\varepsilon^{-1}s} & \mbox{on
  $X^-=[0,\varepsilon]\times Y_-$}.
\end{array}
\right.
\]
It follows from \eqref{eqn:gray1} that if $n$ is sufficiently large
(which we assume that it is), then $(X,\lambda_n')$ is an exact
symplectic cobordism from $(Y_+,\lambda_{+}^1)$ to
$(Y_-,\lambda_-^1)$.

{\em Step 2.\/} We now relate the maps on ECH induced by $(X,\lambda)$
to those induced by $(X,\lambda_n')$.

\begin{lemma}
\label{lem:PhiXn}
The following diagram commutes:
\begin{equation}
\label{eqn:PhiXn}
\begin{CD}
ECH_*^L(Y_+,\lambda_+) @>{\simeq}>> ECH_*^L(Y_+,\lambda_+^1) \\
@VV{\Phi^L(X,\lambda)}V @VV{\Phi^L(X,\lambda_n')}V \\
ECH_*^L(Y_-,\lambda_-) @>{\simeq}>> ECH_*^L(Y_-,\lambda_-^1).
\end{CD}
\end{equation}
Here the horizontal arrows are induced by the canonical isomorphism of
chain complexes \eqref{eqn:Liso2}.
\end{lemma}

\begin{proof}
  Let $\Psi_+$ and $\Psi_-$ denote the top and bottom arrows in
  \eqref{eqn:PhiXn}.  By the Composition axiom we have
\[
\begin{split}
  \Phi^L(X,\lambda_n') &
  =\Phi^L(X^-,\lambda_n')\circ\Phi^L(X^0,\lambda_n')\circ\Phi^L(X^+,\lambda_n'),\\
  \Phi^L(X,\lambda) &
  =\Phi^L(X^-,\lambda)\circ\Phi^L(X^0,\lambda)\circ\Phi^L(X^+,\lambda).
\end{split}
\]
Since $\lambda_n'$ agrees with $\lambda$ on $X^0$, it then suffices to
show that
\begin{gather}
\label{eqn:xn+}
\Phi^L(X^+,\lambda) = \Phi^L(X^+,\lambda_n')\circ\Psi_+  ,\\
\label{eqn:xn-}
  \Phi^L(X^-,\lambda_n') =\Psi_-\circ\Phi^L(X^-,\lambda).
\end{gather}

To prove \eqref{eqn:xn+}, observe that by Lemmas~\ref{lem:technical},
\ref{lem:deform}(a),(b) and \ref{lem:FI}, we have a commutative
diagram
\[
\begin{CD}
ECH_*^L(Y_+,\lambda_+) @>{\Psi_+}>> ECH^L(Y_+,\lambda_+^1)\\
@VV{s}V @VV{\Phi^L(X^+,\lambda_n')}V  \\
ECH_*^{e^{-\varepsilon}L}(Y_+,e^{-\varepsilon}\lambda_+)
 @>{\imath}>>
ECH_*^L(Y_+,e^{-\varepsilon}\lambda_+).\end{CD}
\]
By \eqref{eqn:technical}, the composition $\imath\circ s$ in the above
square is equal to $\Phi^L(X^+,\lambda)$.

To prove \eqref{eqn:xn-}, by Lemmas~\ref{lem:technical},
\ref{lem:deform}(a),(b) and \ref{lem:FI} again, we have a commutative
diagram
\[
\begin{CD}
ECH_*^{e^{-\varepsilon}L}(Y_-,\lambda_-) @>{s}>>
ECH_*^L(Y_-,e^{\varepsilon}\lambda_-)\\
@VV{\simeq}V @VV{\Phi^L(X^-,\lambda_n')}V \\
ECH_*^{e^{-\varepsilon}L}(Y_-,\lambda_-^1) @>{\imath}>> 
ECH_*^L(Y_-,\lambda_-^1).\end{CD}
\]
Here the left vertical arrow is induced by \eqref{eqn:Liso2}.
Similarly to Lemma~\ref{lem:deform}(c), the latter map fits into a
commutative diagram
\[
\begin{CD}
ECH_*^{e^{-\varepsilon}L}(Y_-,\lambda_-) @>{\imath}>> ECH_*^L(Y_-,\lambda_-)\\
 @VV{\simeq}V @VV{\Psi_-}V\\
 ECH_*^{e^{-\varepsilon}L}(Y_-,\lambda_-^1) @>{\imath}>> ECH_*^L(Y_-,\lambda_-^1).
\end{CD}
\]
Combining the above two diagrams gives a commutative diagram
\[
\begin{CD}
ECH_*^{e^{-\varepsilon}L}(Y_-,\lambda_-) @>{s}>>
ECH_*^L(Y_-,e^{\varepsilon}\lambda_-)\\
@VV{\imath}V @VV{\Phi^L(X^-,\lambda_n')}V \\
ECH_*^{L}(Y_-,\lambda_-) @>{\Psi_-}>> ECH_*^L(Y_-,\lambda_-^1).
\end{CD}
\]
By \eqref{eqn:technical} again, the composition $\imath\circ s^{-1}$
in the above square is equal to $\Phi^L(X^-,\lambda)$.
\end{proof}

{\em Step 3.\/} We now construct a strongly cobordism-admissible
almost complex structure $J_n'$ for $(X,\lambda_n')$. On $X^0$ we take
$J_n'=J$.  To define $J_n'$ on $X^\pm$, write
$t=1\pm\varepsilon^{-1}s$, and let $R_\pm^t$ denote the Reeb vector
field associated to $\lambda_\pm^t$.  As a step towards defining
$J_n'$, define an almost complex structure $J_n''$ on $X^\pm$ by
\begin{equation}
\label{eqn:Jnprovisional}
J_n''\frac{\partial}{\partial s} =
R_\pm^t,\quad\quad J_n''|_{\Ker(\lambda_\pm^t)} =
{J_\pm^t}|_{\Ker(\lambda_\pm^t)}.
\end{equation}
It follows from \eqref{eqn:gray1} and \eqref{eqn:gray2} that
\begin{equation}
\label{eqn:jn0j}
\|J_n''-J\|_{C^0} \le cn^{-1}, \quad\quad \|J_n''\|_{C^1} \le c.
\end{equation}
It also follows from \eqref{eqn:gray1} that if $n$ is sufficiently
large, then $J_n''$ is $d\lambda_n'$-tame.  However $J_n''$ is not
necessarily $d\lambda_n'$-compatible, except near $s=0,\mp\varepsilon$.
We can measure the failure of compatibility by a $2$-form $\Omega$ on
$X^\pm$ defined by
\[
\Omega(v_1,v_2) \eqdef d\lambda_n'(v_1,J_n''v_2) - d\lambda_n'(v_2,J_n''v_1).
\]
By \eqref{eqn:gray1} and \eqref{eqn:gray2}, we have
\[
\|\Omega\|_{C^1} \le c n^{-1}.
\]
Now $\Omega$, regarded as a bundle map from the space of almost
complex structures on $X^\pm$ to the space of real $(1,1)$-forms, is
transverse to $0$ at each fiber.  It then follows from the inverse
function theorem and \eqref{eqn:jn0j} that if $n$ is sufficiently
large, then we can find a $d\lambda_n'$-compatible almost complex
structure $J_n'$, which agrees with $J_n''$ near
$s=0,\mp\varepsilon$, and which satisfies
\begin{equation}
\label{eqn:Jnest}
\|J_n'-J\|_{C^0}\le c n^{-1}, \quad\quad \|J_n'\|_{C^1}\le c.
\end{equation}

{\em Step 4.\/}
The last step in the construction is to replace $(\lambda_n',J_n')$ by
a pair $(\lambda_n,J_n)$ which is better behaved with respect to
product regions.  Let us call an embedded Reeb orbit $\gamma$ in $Y_+$
a ``product Reeb orbit (with respect to $(X,\lambda,J)$)'' if
$\gamma\subset \{s_+\}\times Z$ where $[s_-,s_+]\times Z$ is a product
region in $X$ (with respect to $\lambda$ and $J$).  Fix $\delta>0$
such that if $\gamma$ is a product Reeb orbit with action less than
$L$, then:
\begin{description}
\item{(i)} If $Z$ denotes the radius $\delta$
  neighborhood of $\gamma$, then $[s_-,s_+]\times Z$ is a product
  region in $X$ for some $s_-,s_+$.
\item{(ii)} $\gamma$ has distance at least $2\delta$ from all other
  %product
 Reeb orbits in $Y_+$ with action less than $L$.
\end{description}

\begin{lemma}
\label{lem:ljn}
If $n>\delta^{-1}$, then there is a $1$-form $\lambda_n$ on $X$ such
that $(X,\lambda_n)$ is an exact symplectic cobordism from
$(Y_+,\lambda_+^1)$ to $(Y_-,\lambda_-^1)$, and a strongly
cobordism-admissible almost complex structure $J_n$ on $\overline{X}$
for $\lambda_n$, with the following properties:
\begin{description}
\item{(a)} The exact symplectic cobordisms $(X,\lambda_n)$ and
  $(X,\lambda_n')$ from $(Y_+,\lambda_+^1)$ to $(Y_-,\lambda_-^1)$ are
  homotopic in the sense of \S\ref{sec:SMT}.
\item{(b)} If $\gamma$ is a product Reeb orbit in $Y_+$ of action less
  than $L$ with respect to $(X,\lambda,J)$, then the radius $\delta$
  neighborhood of $\gamma$ is contained in a product region for
  $(X,\lambda_n,J_n)$.
\item{(c)} $(\lambda_n,J_n)$ agrees with $(\lambda_n',J_n')$ on
  $\overline{X}\setminus X$, and on the complement in $X$ of the
  product regions $[s_-,s_+]\times Z$ where $Z$ is the radius
  $1/n$ neighborhood of a product Reeb orbit in $Y_+$ of action less
  than $L$.
\item{(d)}
$\|J_n-J\|_{C^0}\le cn^{-1}$ and $\|J_n\|_{C^1}\le c$.
\end{description}
\end{lemma}

\begin{proof}
  Let $\gamma$ be a product Reeb orbit in $Y_+$ with action less than
  $L$, and let $[s_-,s_+]\times Z$ be the corresponding product region
  as in (i) above.
By the construction of $\lambda_\pm^t$, the $1$-forms $e^{-s_+}\lambda_+^t$ and
$e^{-s_-}\lambda_-^t$ agree on $Z$, so let us denote this $1$-form
simply by $\lambda_0^t$.  Now on $[s_-,s_+]\times Z$, replace
$\lambda_n'$ by
\[
\lambda_n \eqdef e^s\lambda_0^1.
\]

To construct $J_n$ on $[s_-,s_+]\times Z$, recall from the
construction of $J_\pm^1$ that the restrictions of $J_+^1$ and $J_-^1$
to $\Ker(\lambda_0^1)$ agree.  Let $R_0^t$ denote the
Reeb vector field associated to $\lambda_0^t$, and recall from the
definition of ``product region'' that on this region,
$J(\partial/\partial s)=fR_0^0$ where $f$ is some function of $s$
which, by the definition of ``strongly cobordism-admissible'', equals
$e^s$ near $s=s_\pm$.
Now define $J_n$ on this region by
\[
J_n\frac{\partial}{\partial s} =
fR_0^{1},\quad\quad {J_n}|_{\Ker(\lambda_0^1)} =
{J_\pm^1}|_{\Ker(\lambda_0^{1})}.
\]

Let $(\lambda_n,J_n)$ be obtained by modifying $(\lambda_n',J_n')$ as
above for each product Reeb orbit of action less than $L$.  These
satisfy properties (a), (b), and (c) by construction, and property (d)
follows from \eqref{eqn:gray1}, \eqref{eqn:gray2}, and
\eqref{eqn:Jnest}.
\end{proof}

We now state a lemma implying that if the hypothesis of (i) or (ii) in
the Holomorphic Curves axiom holds for $(X,\lambda,J)$, then it also
holds for $(X,\lambda_n,J_n)$ when $n$ is sufficiently large.
Consider pairs $(\Theta_+,\Theta_-)$ where $\Theta_\pm$ is an ECH
generator for $\lambda_\pm$ of action less than $L$.  Recall from
Definition~\ref{def:lfa} that $\Theta_\pm$ corresponds to an ECH
generator for $\lambda_\pm^1$ of action less than $L$, and we denote
this also by $\Theta_\pm$.  Let $A$ denote the set of pairs
$(\Theta_+,\Theta_-)$ for which there exists no broken $J$-holomorphic
curve from $\Theta_+$ to $\Theta_-$.  Let $A_n$ denote the set of
pairs $(\Theta_+,\Theta_-)$ for which there exists no broken
$J_n$-holomorphic curve from $\Theta_+$ to $\Theta_-$.
Let $B$ denote the set of pairs $(\Theta_+,\Theta_-)$ for which the
only broken $J$-holomorphic curve from $\Theta_+$ to $\Theta_-$ is a
union of covers of product cylinders.  Let $B_n$ denote the set of
pairs $(\Theta_+,\Theta_-)$ for which the only broken
$J_n$-holomorphic curve from $\Theta_+$ to $\Theta_-$ is a union of
covers of product cylinders.

\begin{lemma}
\label{lem:gc1}
If $n$ is sufficiently large, then $A\subset A_n$ and $B\subset B_n$.
\end{lemma}

Lemma~\ref{lem:gc1} is proved by a Gromov compactness argument in
\S\ref{sec:gc} below.  Assuming this, we can now give:

\begin{proof}[Proof of the Holomorphic Curves axiom (strongly
  cobordism-admissible case)]
  Choose $n$ sufficiently large as in Lemmas~\ref{lem:ljn} and
  \ref{lem:gc1}.  Define a chain map
\[
\hat{\Phi}: ECC_*^L(Y_+,\lambda_+;J_+) \longrightarrow
ECC_*^L(Y_-,\lambda_-;J_-)
\]
as the composition
\[
ECC_*^L(Y_+,\lambda_+;J_+) \to ECC_*^L(Y_+,\lambda_+^1;J_+^1) \to
ECC_*^L(Y_-,\lambda_-^1;J_-^1) \to ECC_*^L(Y_-,\lambda_-;J_-).
\]
Here the first map is the canonical isomorphism of chain complexes
\eqref{eqn:Liso2} for $Y_+$, the second map is the chain map
\eqref{eqn:flf} for the cobordism $(X,\lambda_n,J_n)$, and the third
map is the inverse of the canonical isomorphism of chain complexes
\eqref{eqn:Liso2} for $Y_-$.  By Lemmas~\ref{lem:PhiXn} and
\ref{lem:ljn}(a) and the Homotopy Invariance axiom, the chain map
$\hat{\Phi}$ induces the map $\Phi^L(X,\lambda)$ on homology.

To prove that $\hat{\Phi}$ fulfills conditions (i) and (ii) in the Holomorphic
Curves axiom, we must show that if $(\Theta_+,\Theta_-)\in A$ then
$\langle \hat{\Phi}\Theta_+,\Theta_-\rangle=0$, and if $(\Theta_+,\Theta_-)\in
B$ then $\langle \hat{\Phi}\Theta_+,\Theta_-\rangle = 1$.  If
$(\Theta_+,\Theta_-)\in A$ (resp.\ $B$), then by Lemma~\ref{lem:gc1} we
have $(\Theta_+,\Theta_-)\in A_n$ (resp.\ $B_n$), and by
Proposition~\ref{prop:HCLF} applied to $(X,\lambda_n,J_n)$ we have
$\langle \hat{\Phi}\Theta_+,\Theta_-\rangle =0$ (resp.\ $1$).
\end{proof}

\subsection{Gromov compactness}
\label{sec:gc}

We now prove Lemma~\ref{lem:gc1}.  Continuing with the setting of
\S\ref{sec:hcg}, it is enough to show the following:

\begin{lemma}
\label{lem:gc2}
Let $\Theta_\pm$ be ECH generators for $\lambda_\pm$ of action less
than $L$.  Suppose that $(n_1,n_2,\ldots)$ is an increasing infinite
sequence of positive integers such that for each
$n\in\{n_1,n_2,\ldots\}$ there exists a broken $J_n$-holomorphic curve
$u_n\in\overline{\mc{M}^{J_n}(\Theta_+,\Theta_-)}$. Then:
\begin{description}
\item{(a)} After passing to a subsequence, the broken
  $J_n$-holomorphic curves $u_n$ converge (in the sense of
  \cite[\S9]{pfh2}, using currents instead of maps) to a broken
  $J$-holomorphic curve $u\in\overline{\mc{M}^J(\Theta_+,\Theta_-)}$.
\item{(b)} If $u$ is a union of covers of product cylinders, then so is
  $u_n$ for all sufficiently large $n$.
\end{description}
\end{lemma}

To clarify assertion (a), note that by construction, the Liouville
forms $\lambda_n$ and $\lambda$ on $X$ have the same Liouville vector
field near $\partial X$, and so there is a canonical diffeomorphism
between the completions \eqref{eqn:completion} of $(X,\lambda_n)$ and
$(X,\lambda)$, which is the identity on each of the three subsets in
\eqref{eqn:completion}.

To prove Lemma~\ref{lem:gc2}, note first that part (b) follows quickly
from part (a).  The reason is that if $u$ is a union of covers of
product cylinders, then by Lemma~\ref{lem:ljn}(b), if $n$ is
sufficiently large then each level of $u_n$ is either (i) a
$J_n$-holomorphic curve in $\overline{X}$ from $\Theta_+$ to
$\Theta_-$ contained in a product region for $(X,\lambda_n,J_n)$, or
(ii) a $J_\pm^1$-holomorphic curve in $\R\times Y_\pm$ from
$\Theta_\pm$ to itself.  In case (ii), since $d\lambda_\pm^1$ is
pointwise nonnegative on any $J_\pm^1$-holomorphic curve, and zero
only where the holomorphic curve is tangent to $\R$ cross the Reeb
flow, it follows by Stokes' theorem that any level of type (ii) maps
to a union of $\R$-invariant cylinders, and in particular does not
exist by the nontriviality condition in our definition of ``broken
holomorphic curve''.  So there is only a level of type (i), and the
same argument shows that this maps to a union of product cylinders.

To prove Lemma~\ref{lem:gc2}(a), first note that the arguments for
\cite[Lem.\ 9.8]{pfh2} can be used with only minor notational changes
to see that it is enough to prove the following assertion about
unbroken holomorphic curves:

\begin{lemma}
\label{lem:gc3}
Let $\Theta_\pm$ be ECH generators for $\lambda_\pm$ of action less
than $L$.  Suppose that $(n_1,n_2\ldots)$ is an increasing sequence of
positive integers such that for each $n\in\{n_1,n_2,\ldots\}$ there is
a $J_n$-holomorphic curve
$C_n\in\mc{M}^{J_n}(\Theta_+,\Theta_-)$. Then:
\begin{description}
\item{(a)} After passing to a subsequence, the $J_n$-holomorphic
  curves $C_n$ converge as currents on $\overline{X}$ to a
  $J$-holomorphic curve $C\in\mc{M}^J(\Theta_+',\Theta_-')$ for some
  orbit sets $\Theta_\pm'$ for $\lambda_\pm$.
\item{(b)} Let $s_n$ be a sequence of positive real numbers with
  $\lim_{n\to\infty}s_n=\infty$.  Let
  $C_n'\subset[-s_n,s_n]\times Y_+$ denote the translate by $-s_n$ of
  the intersection of $C_n$ with $[0,2s_n]\times
  Y_+\subset\overline{X}$.  Then after passing to a subsequence, the
  curves $C_n'$ converge as a current to a $J_+$-holomorphic curve in $\R\times
  Y_+$ between some orbit sets for $\lambda_+$.
\item{(c)} Likewise, let $s_n$ be a sequence of negative real numbers
  with $\lim_{n\to\infty}s_n=-\infty$.  Let
  $C_n'\subset[s_n,-s_n]\times Y_-$ denote the translate by $-s_n$ of
  the intersection of $C_n$ with $[2s_n,0]\times
  Y_-\subset\overline{X}$.  Then after passing to a subsequence, the
  curves $C_n'$ converge as a current to a $J_-$-holomorphic curve in
  $\R\times Y_-$ between some orbit sets for $\lambda_-$.
\end{description}
\end{lemma}

Note that this lemma does not directly follow from standard Gromov
compactness results, because the sequence $\{J_n\}$ does not converge
to $J$ in $C^1$; we just have $C^0$ convergence and a $C^1$ bound from
Lemma~\ref{lem:ljn}(d).

\begin{proof}[Proof of Lemma~\ref{lem:gc3}.]  We will just prove part
  (a), as the proofs of parts (b) and (c) are essentially the same.
  The arugment has three steps.
 
  {\em Step 1.\/}  We first obtain convergence to some current (which
  we will later show is $J$-holomorphic).

  Let $\Sigma\subset\overline{X}$ denote the union of the product
  cylinders $\R\times\gamma$ where $\gamma$ is a product Reeb orbit of
  length less than $L$, the half-cylinders
  $[-\varepsilon,\infty)\times\gamma_+$ where $\gamma_+$ is a Reeb
  orbit of $\lambda_+$ of length less than $L$, and the half-cylinders
  $(-\infty,\varepsilon]\times\gamma_-$ where $\gamma_-$ is a Reeb
  orbit of $\lambda_-$ of action less than $L$.  Let
  $\Sigma_{1/n}\subset\overline{X}$ denote the radius $1/n$
  neighborhood of $\Sigma$.  By construction, $(\lambda_n,J_n)$ agrees
  with $(\lambda,J)$ on $\overline{X}\setminus\Sigma_{1/n}$.

Observe that by Stokes' theorem,
\begin{gather*}
\int_{C_n\cap((-\infty,0]\times Y_-)}d\lambda_-^1 + \int_{C_n\cap
  X}d\lambda_n + \int_{C_n\cap([0,\infty)\times Y_+)}d\lambda_+^1
=\quad\quad\\
\quad\quad\quad\quad\quad\quad\quad\quad\quad\quad=
\int_{\Theta_+}\lambda_+^1 - \int_{\Theta_-}\lambda_-^1 \le L.
\end{gather*}
It follows from this that for any compact set $K\subset \overline{X}$,
the area of $C_n\cap K$ has an $n$-independent upper bound.
It now follows from the compactness theorem for currents, see
\cite[4.2.17]{federer} or \cite[Thm.\ 5.5]{morgan}, that we can pass
to a subsequence so that $\{C_n\}$ converges weakly as a current to an
integral rectifiable current $C$ with
locally finite $2$-dimensional Hausdorff measure. 

\begin{lemma}
The convergence to $C$ is pointwise in the sense that
\begin{equation}
\label{eqn:pointwise}
\lim_{n\to\infty}\left(\sup_{x\in C\cap K}\op{dist}(x,C_n) + \sup_{x\in
  C_n\cap K}\op{dist}(x,C)\right) = 0
\end{equation}
for every compact set $K\subset\overline{X}$.
\end{lemma}

\begin{proof}
This is proved by
copying the arguments in \cite[\S5c]{swtogr} and using
Lemma~\ref{lem:monotonicity} below.
\end{proof}

Given $\rho>0$ and $x\in C_n$, let
$a_n(x,\rho)$ denote the integral of $d\lambda$ over the subset of
$C_n$ with distance less than or equal to $\rho$ from $x$.

\begin{lemma}
\label{lem:monotonicity}
  There exists a constant $\kappa>1$ such that for all $n>\kappa$ and
  $x\in C_n$, if $\kappa^{-1}>\rho>\rho'>0$, then
\[
a_n(x,\rho) > \kappa^{-1}(\rho/\rho')^2 a_n(x,\rho').
\]
\end{lemma}

\begin{proof}
It follows from Lemma~\ref{lem:ljn}(d) that $J_n$ is tamed by
$d\lambda$ for all sufficiently large $n$.  Moreover, if $|\cdot|$
denotes the metric determined by $d\lambda$ and $J$, then there
exists a constant $\delta>0$ such that if $n$ is sufficiently large
then $d\lambda(v,J_nv)\ge \delta|v|^2$.  One can then apply
\cite[Thm.\ 2.1]{ye}.
\end{proof}

{\em Step 2.\/} We now recall a criterion for $C$ to be
$J$-holomorphic.

Let $D$ denote the closed unit disk.  Call a smooth map
$\sigma:D\to\overline{X}$ {\em admissible\/} if
$\sigma(\partial D)\subset\overline{X}\setminus C$.

\begin{definition}
\label{def:pca}
  (cf.\ \cite[\S6a]{swtogr}) A {\em positive cohomology assigment\/}
  is an assigment, to each admissible map $\sigma$, of an integer
  $I(\sigma)$, satisfying the following conditions:
\begin{description}
\item {(a)} $I(\sigma)=0$ if the image of $\sigma$ is disjoint from
  $C$.
\item
{(b)} If $\sigma_0$ and $\sigma_1$ are admissible maps that are homotopic
through admissible maps, then $I(\sigma_0)=I(\sigma_1)$.
\item
{(c)}
If $\sigma$ is admissible and if $\phi:D\to D$ is a smooth map so that
$\phi:\partial D\to \partial D$ is a degree $k$ covering, then
$I(\sigma\circ\phi)=k I(\sigma)$.
\item{(d)} Suppose that $\sigma$ is admissible and that
  $\sigma^{-1}(C)$ is contained in the interior of a finite disjoint
  union $\coprod_i D_i$ where each $D_i$ is the image of an
  orientation-preserving embedding $\theta_i:D\to D$.  Then
  $I(\sigma)=\sum_i I(\sigma\circ\theta_i)$.
\item{(e)}
If $\sigma$ is a $J$-holomorphic embedding whose image intersects $C$,
then $I(\sigma)>0$.
\end{description}
\end{definition}

If there exists a positive cohomology assignment, then it follows as
in \cite[Lem.\ 4.4]{e4} that $C$ is a $J$-holomorphic subvariety of
$\overline{X}$.  The arguments in \cite[Lem.\ 9.8]{pfh2} then show
that $C$ is an element of $\mc{M}^J(\Theta_+',\Theta_-')$ for some
$\Theta_\pm'$.

{\em Step 3.\/} To complete the proof of Lemma~\ref{lem:gc3}, we
define a positive cohomology assignment $I$ as follows.  If $\sigma:D\to
C$ is an admissible map, then it follows from the pointwise
convergence \eqref{eqn:pointwise} that $\sigma(\partial D)$ is
disjoint from $C_n$ whenever $n$ is sufficiently large.  It then
follows from the convergence of currents that the intersection number
of $D$ with $C_n$ is independent of $n$ when $n$ is sufficiently
large.  Define $I(\sigma)$ to be this intersection number.

Conditions (a)--(d) in Definition~\ref{def:pca} follow directly from
the definition of $I$, together with the fact that $C_n$ converges to
$C$ both as a current and pointwise in the sense of
\eqref{eqn:pointwise}.  Condition (e) is immediate in the special case
when $\sigma$ maps to $\overline{X}\setminus\Sigma$, because then
$C_n$ is $J$-holomorphic in a neighborhood of $\sigma(D)$ for all
sufficiently large $n$.  In particular, it follows from \cite[Lem.\
4.4]{e4} that $C\cap(\overline{X}\setminus\Sigma)$ is a
$J$-holomorphic submanifold on the complement of a discrete set.  This
last fact can also be deduced from standard Gromov compactness
theorems, see e.g.\ \cite{hummel,wolfson,ye}, since the intersection
of $C$ with any compact subset of $\overline{X}\setminus\Sigma$ is a
pointwise limit of $J$-holomorphic subvarieties.

It remains to prove condition (e) when $\sigma(D)$ is allowed to
intersect $\Sigma$.  By \cite[Lem.\ 5.5]{swtogr}, any holomorphic disk
(without boundary constraint) can be perturbed to a holomorphic disk
that is transverse to $\Sigma$.  So by conditions (a) and (d), we can
reduce to the case where $\sigma(D)$ has small radius and intersects
$\Sigma$ only at its center point, transversely, which is also in $C$.
To prove property (e) in this case, we use the following lemma, which
allows us to perturb a family of $J$-holomorphic disks to a family of
$J_n$-holomorphic disks.

\begin{lemma}
\label{lem:g}
Let $D_1,D_2$ be disks centered at the origin in $\C$, and let
$\phi:D_1\times D_2\to\overline{X}$ be a map such that
$\phi|_{D_1\times\{z_2\}}$ is a $J$-holomorphic embedding for each
$z_2\in D_2$, and $\phi^{-1}(\Sigma) = \{0\}\times D_2$.  After
replacing $D_1$ by a sufficiently small radius subdisk, given
$\epsilon>0$, if $n$ is sufficiently large, then there exists a smooth
map $\varphi_n:[0,1]\times D_1\times D_2\to\overline{X}$ with the
following properties:
\begin{itemize}
\item
$\varphi_n(0,\cdot,\cdot)=\phi$.
\item For each $z_2\in D_2$, the map $\varphi_n(1,\cdot,z_2)$ is an
  embedding with $J_n$-holomorphic image.
\item
$\sup_{t\in[0,1],z_1\in D_1,z_2\in
  D_2}\op{dist}(\phi(z_1,z_2),\varphi_n(t,z_1,z_2))<\epsilon$.
\end{itemize}
\end{lemma}

Granted Lemma~\ref{lem:g}, the proof of property (e) is completed as
follows.  Let $\sigma:D_1\to\overline{X}$ be an admissible map which
intersects $\Sigma$ only at its center point, transversely, which is
also in $C$.  By \cite[Lem.\ 5.5]{swtogr}, we can then find
$\phi:D_1\times D_2\to\overline{X}$ as in Lemma~\ref{lem:g} such that
$\phi$ restricts to a diffeomorphism from a neighborhood of $(0,0)$ to
an open set $U$ in $\overline{X}$.  We can shrink $D_1$ as in
Lemma~\ref{lem:g}, and also shrink $D_2$, so that
$\phi|_{D_1\times\{z_2\}}$ is admissible for all $z_2\in D_2$.  By the
pointwise convergence \eqref{eqn:pointwise}, if $n$ is sufficiently
large, then $C_n$ intersects $U$.  It follows that if $\varepsilon$ in
Lemma~\ref{lem:g} is chosen sufficiently small, and if $n$ is
sufficiently large, then $\varphi_n(1,\cdot,z_2)$ intersects $C_n$ for
some $z_2\in D_2$.  Moreover, it follows from the pointwise
convergence \eqref{eqn:pointwise} that if $\varepsilon$ is
sufficiently small and $n$ is sufficiently large then
$\varphi_n(1,\cdot,z_2)$ is homotopic to $\sigma$ through disks whose
boundaries do not intersect $C_n$.  Therefore $I(\sigma_1)$ equals the
intersection number of $C_n$ with $\varphi_n(1,\cdot,z_2)$ when
$\varepsilon$ is sufficiently small and $n$ is sufficiently large.
Since the latter disk is $J_n$-holomorphic and intersects $C_n$, we
conclude that $I(\sigma_1)>0$ as desired.
\end{proof}

\begin{proof}[Proof of Lemma~\ref{lem:g}.]
  To simplify notation we will just prove the lemma in the case when
  $D_2$ is a point, and we will drop $z_2$ from the notation and write
  $D=D_1$.  The lemma in the general case then follows by noting that
  the estimates used to prove the lemma when $D_2$ is a point vary
  continuously with a smooth family of holomorphic disks.  So let
  $\phi:D\to\overline{X}$ be a holomorphic map such that
  $\phi^{-1}(\Sigma)=\{0\}$; we need to show that after replacing $D$
  by a smaller radius disk, given $\varepsilon>0$, if $n$ is
  sufficiently large then there exists $\varphi_n:[0,1]\times
  D\to\overline{X}$ such that $\varphi_n(0,\cdot)=\phi$, the map
  $\varphi_n(1,\cdot)$ is an embedding with $J_n$-holomorphic image,
  and $\sup_{t\in[0,1]}\sup_{z\in
    D}\op{dist}(\phi(z),\varphi_n(t,z))<\epsilon$.  We do so in five
  steps.

  {\em Step 1.\/} We first write down the equations we need to solve in
  a convenient coordinate system.

  We can choose complex coordinates $(z,w)$ for a neighborhood of
  $\phi(0)$ in $\overline{X}$ with the following properties: First,
  the intersection of $D$ with this neighborhood is given by $w=0$.
  Second, each constant $z$ slice is $J$-holomorphic.  Third, the $J$
  version of $T^{1,0}\overline{X}$ is spanned by
\begin{equation}
\label{eqn:t01}
dz + \sigma d\overline{z}, \quad\quad dw +
  \gamma d\overline{z},
\end{equation}
where $\sigma$ and $\gamma$ are smooth functions that obey
$|\sigma(\cdot,w)| + |\gamma(\cdot,w)| \le c|w|$.  Such coordinates
can be found in a neighborhood of any point on a $J$-holomorphic curve
in an almost complex $4$-manifold, as explained in
\cite[\S5d]{swtogr}.
Similarly to \eqref{eqn:t01}, the $J_n$ version of $T^{1,0}\overline{X}$ is spanned by
\begin{equation}
\label{eqn:t01n}
dz + \sigma_nd\zbar + \mu_n d\wbar, \quad\quad dw + \gamma_n d\zbar +
\nu_n d\wbar
\end{equation}
where $\sigma_n,\mu_n,\gamma_n,\nu_n$ are smooth functions.  By
Lemma~\ref{lem:ljn}(d), these satisfy $|\sigma_n - \sigma| + |\mu_n| +
|\gamma_n-\gamma| + |\vu_n| \le c n^{-1}$, and the first derivatives
of $\sigma_n-\sigma$, $\mu_n$, $\gamma_n-\gamma$ and $\nu_n$ are
bounded in absolute value by $c$.

Now fix $r>0$ such that the coordinates $z$ and $w$ are defined where
both have norm less than $2r$, and replace $D$ with the disk $(w=0,
|z|\le r)$.  Let $\eta:D\to\C$ be a smooth function with $|\eta|<r$.
It follows from \eqref{eqn:t01n} that the graph $w=\eta(z)$ is
$J_n$-holomorphic if and only if
\[
\frac{\partial\eta}{\partial\zbar} + \gamma_n -
\sigma_n\frac{\partial\eta}{\partial z} +
\nu_n\frac{\partial\overline{\eta}}{\partial\zbar} +
(\mu_n\gamma_n-\sigma_n\nu_n)\frac{\partial\overline{\eta}}{\partial
  z} + \mu_n\left(\frac{\partial\overline{\eta}}{\partial
    z}\frac{\partial\eta}{\partial\zbar} -
  \frac{\partial\overline{\eta}}{\partial\zbar}\frac{\partial\eta}{\partial
    z}\right) = 0.
\]
It proves useful to rewrite the above equation in the schematic form
\begin{equation}
\label{eqn:schematic}
\frac{\partial\eta}{\partial\zbar} + \gamma-
\sigma\frac{\partial\eta}{\partial z} + \frak{r}_0 +
\frak{r}_1(\neta,\nabla\eta) + \frak{r}_2(\eta,\nabla\eta).
\end{equation}
Here $\frak{r}_0=(\gamma_n-\gamma)|_{w=0}$ is a function of $z$ with
$|\frak{r}_0|\le c n^{-1}$ and with first derivatives that are bounded
in absolute value by $c$.  Meanwhile $\frak{r}_1(a,\cdot)$ for fixed
$a$ is a $z$-dependent affine linear function that obeys
$|\frak{r}_1(a,b)|\le cn^{-1}(|a|+|b|)$.  The first derivatives of
$\frak{r}_1(a,\cdot)$ are bounded in absolute value by $c$.  Finally,
$\frak{r}_2(a,\cdot)$ for fixed $a$ is a quadratic function of its
second entry with $|\frak{r}_2(a,b)|\le c n^{-1}|b|^2$.  The first
derivatives of $\frak{r}_2$ with respect to both $z$ and $a$ are
bounded in absolute value by $c$.  Also observe that since
$\phi^{-1}(\Sigma)=\{0\}$, it follows that for any $\delta>0$, if $n$
is sufficiently large then $\frak{r}_0=0$ where $|z|>\delta$.

To prove Lemma~\ref{lem:g}, it now suffices to show that for every
$\varepsilon>0$, if $n$ is sufficiently large then there exists a
solution $\eta_n$ to the equation \eqref{eqn:schematic} with
$|\eta_n|<\varepsilon$.  One can then define $\varphi_n(t,z)=(z,w=t\eta_n(z))$.

{\em Step 2.\/} We will solve \eqref{eqn:schematic} using a fixed
point construction in a certain Banach space $\mc{H}$ of $C^1$
functions.

To define the Banach space $\mc{H}$, fix once and for all a number
$\nu\in(0,1/16)$.  If ${\mathbb V}$ is any finite dimensional normed
vector space over $\C$, define a norm $\|\cdot\|_\diamond$ on the
space of bounded smooth functions $f:\C\to{\mathbb V}$ by
\[
\|f\|_\diamond^2 \eqdef
\sup_{z\in\C}\sup_{\rho\in[0,1]}\rho^{-\nu}\int_{|z'-z|<\rho}|f(z')|^2.
\]
Now let $\mc{C}$ denote the space of smooth functions $\eta:\C\to\C$
that are holomorphic on the complement of the unit disk and that
satisfy $\lim_{|z|\to\infty}\eta(z)=0$.  Define a norm $\|\cdot\|_{*}$
on $\mc{C}$ by
\[
\|\eta\|_{*} \eqdef \|\nabla\eta\|_2 + \|\nabla\eta\|_\diamond +
\|\nabla\nabla\eta\|_\diamond.
\]
Finally, define $\mc{H}$ to be the completion of $\mc{C}$ with respect
to the norm $\|\cdot\|_{*}$.  The following lemma about $\mc{H}$
will be needed below:

\begin{lemma}
\label{lem:morrey}
$\mc{H}$ is a subset of the H\"older space $C^{1,\nu/2}$, and the
inclusion $\mc{H}\to C^{1,\nu/2}$ is a bounded linear map of Banach spaces.
\end{lemma}

\begin{proof}
By \cite[Thm.\ 3.5.2]{morrey}, there exists a constant $c$ (depending
on $\nu$) such that
\begin{equation}
\label{eqn:morrey}
|\eta|\le c\|\nabla\eta\|_\diamond, \quad\quad |\nabla\eta|\le
c\|\nabla\nabla\eta\|_\diamond,
\end{equation}
and the exponent $\nu/2$ H\"older norm of $|\nabla\eta|$ is also
bounded by $c\|\nabla\nabla\|_\diamond$.
\end{proof}

{\em Step 3\/} (of the proof of Lemma~\ref{lem:g}). Fix a smooth
function $\chi:\C\to[0,1]$ that is equal to $1$ on the disk of radius
$r/4$ and equal to $0$ outside of the disk of radius $r/2$.  Given
$\eta\in\mc{C}$, a standard use of the Green's function for $\dbar$ on
$\C$ finds a unique solution $T=T(\eta)\in\mc{C}$ of the equation
\begin{equation}
\label{eqn:green}
\frac{\partial T}{\partial\zbar} = -\chi\left(\gamma - \sigma
  \frac{\partial\eta}{\partial z} + \frak{r}_0 +
  \frak{r}_1(\eta,\nabla\eta) + \frak{r}_2(\eta,\nabla\eta)\right).
\end{equation}
Here $\frak{r}_1$ and $\frak{r}_2$ should be extended arbitrarily for
$|w|>r$ so that they still satisfy the estimates from Step 1.  It
follows from \eqref{eqn:green}, using \eqref{eqn:morrey} and
\cite[Thms.\ 3.5.2 and 5.4.1]{morrey}, that
\begin{equation}
\label{eqn:moremorrey}
\|T\|_* \le c\left(\|\frak{r}_0\|_\infty +
  \|\nabla\frak{r}_0\|_\diamond + n^{-1}\|\eta\|_{*} +
  \|\eta\|_{*}^2\right).
\end{equation}

{\em Step 4.\/} Fix $\varepsilon>0$ and let
$\mc{H}_\varepsilon\subset\mc{H}$ denote the ball of radius
$\varepsilon$ centered at the origin.  We claim that if $n$ is
sufficiently large, then the map $\eta\mapsto T(\eta)$ maps
$\mc{H}_\varepsilon\cap\mc{C}$ to itself.  By \eqref{eqn:moremorrey},
it is enough to show that
\begin{equation}
\label{eqn:ad}
\|\nabla\frak{r}_0\|_\diamond < \frac{1}{2}c^{-1}\varepsilon
\end{equation}
if $n$ is sufficiently large, where $c$ here denotes the same constant
as in \eqref{eqn:moremorrey}.  To do so, recall that for any $\delta>0$,
if $n$ is large enough then $\frak{r}_0$ is supported in the disk of
radius $\delta$.  Then the bound $|\nabla\frak{r}_0|\le c$ implies
that for each $z$ we have
\[
\int_{|z'-z|<\rho}|\nabla\frak{r}_0(z')|^2 \le c\min(\rho^2,\delta^2).
\]
It follows that $\|\nabla\frak{r}_0\|_\diamond\le c\delta^{1-\nu/2}$.
By taking $\delta$ sufficiently small, we conclude that the desired
inequality \eqref{eqn:ad} holds if $n$ is sufficiently large.

{\em Step 5.\/} By Step 4, for any $\varepsilon>0$, if $n$ is
sufficiently large then $T^k(0)\in\mc{H}_\varepsilon$ for all $k\ge
0$.  By Lemma~\ref{lem:morrey} and the Arzela-Ascoli theorem, the
sequence $\{T^k(0)\}_{k=0,1,\ldots}$ then converges uniformly in the
$C^1$ topology to a $C^1$ function $\eta$.  Since the convergence is
in $C^1$, the limit function $\eta$ obeys \eqref{eqn:schematic}.
Also, elliptic bootstrapping shows that $\eta$ is in fact $C^\infty$.
Finally, by \eqref{eqn:morrey} we have $|\eta|< c\varepsilon$, where
$c$ does not depend on $\varepsilon$.  As explained at the end of Step
1, this completes the proof of Lemma~\ref{lem:g}.
\end{proof}

\section{Cobordism maps and holomorphic curves (proofs)}
\label{sec:holproof}

To complete the unfinished business, this section proves
Propositions~\ref{prop:hol} and \ref{prop:comp}, which were used in
\S\ref{sec:cobexact} to define the map on $\widehat{HM}^*_L$ induced
by an exact symplectic cobordism.

\subsection{Statement of Proposition~\ref{prop:9x}}
\label{sec:9xprelim}

Propositions~\ref{prop:hol} and \ref{prop:comp} will be deduced from
Proposition~\ref{prop:9x} below, which describes how Seiberg-Witten
solutions in a cobordism give rise to holomorphic curves.  The
statement of Proposition~\ref{prop:9x} requires the following
preliminaries.

\paragraph{The Seiberg-Witten action functional.}
Let $Y$ be a closed oriented 3-manifold with a contact form $\lambda$,
and let $J$ be a symplectization-admissible almost complex structure
on $\R\times Y$.  These determine a metric on $Y$ according to the
conventions in \S\ref{sec:pcf}.  Fix a spin-c structure and recall the
splitting \eqref{eqn:SEKE}.  

As noted in \S\ref{sec:pcf}, solutions to our perturbed Seiberg-Witten
equations \eqref{eqn:Tinstanton} on $\R\times Y$ correspond to
gradient flow lines of the functional \eqref{eqn:aeta}, under the
identifications \eqref{eqn:connections}, \eqref{eqn:perturbation}
and \eqref{eqn:rescaling}.  However it will be convenient below to
regard these solutions as gradient flow lines of a different
functional $\frak{a}$ on connections on $E$ and sections of $\Sp$
defined by
\begin{equation}
\label{eqn:a}
\frak{a}(A,\psi) \eqdef \frac{1}{2}\left(cs(A) - r\energy(A)\right) +
  \frak{e}_\mu(A) + r\int_Y\langle D_A\psi,\psi\rangle,
\end{equation}
where the terms in \eqref{eqn:a} are defined as follows.

Choose a reference (Hermitian) connection
$A_E$ on the line bundle $E$.  An arbitrary connection $A$ on $E$
differs from $A_E$ by an imaginary-valued $1$-form.  We define the
{\em Chern-Simons functional\/}
\[
cs(A) \eqdef -\int_Y (A-A_E)\wedge d(A-A_E) - 2\int_Y (A-A_E)\wedge
\left(F_{A_{E}}+\frac{1}{2}F_{A_{K^{-1}}}\right).
\]
Here $A_{K^{-1}}$ is the distinguished connection on $K^{-1}$ defined
in \S\ref{sec:pcf}.  Also, $\energy(A)$ in \eqref{eqn:a} is the energy
defined in \eqref{eqn:energy}, and
\[
\frak{e}_\mu(A) \eqdef i\int_Y(A-A_E)\wedge\mu.
\]

The functionals \eqref{eqn:aeta} and \eqref{eqn:a} differ by a
constant as follows:  If we make the identifications
\eqref{eqn:connections}, \eqref{eqn:perturbation} and
\eqref{eqn:rescaling}, and choose $\mathbb{A}_0=A_{K^{-1}}+2A_E$, then
\begin{equation}
\label{eqn:compareFunctionals}
\frak{a}_\eta(\mathbb{A},\Psi) = \frak{a}(A,\psi) + \frac{ir}{2}\int_Y
F_{A_E}\wedge\lambda.
\end{equation}

\paragraph{Geometric setup.}
Proposition~\ref{prop:9x} is applicable to two geometric setups:

{\em Case 1:\/} The first geometric setup, which is needed for
Proposition~\ref{prop:hol}, is where $(X,\lambda)$ is an exact
symplectic cobordism from $(Y_+,\lambda_+)$ to $(Y_-,\lambda_-)$.  In
this case let $\overline{X}$ denote the completion of $X$ as in
\eqref{eqn:completion}.  Let us denote the ends of $\overline{X}$ by
$\mc{E}_-\eqdef (-\infty,0]\times Y_-$ and $\mc{E}_+\eqdef
[0,\infty)\times Y_+$.  Also let $s_*:\overline{X}\to\R$ denote the
piecewise smooth function which agrees with the $(-\infty,0]$
coordinate on $\mc{E}_-$, which agrees with the $[0,\infty)$
coordinate on $\mc{E}_+$, and which equals $0$ on $X$.

Recall from \S\ref{sec:pesc} that to write down the Seiberg-Witten
equations \eqref{eqn:tsw4} on $\overline{X}$, we need to choose a
strongly cobordism-admissible almost complex structure $J$ on
$\overline{X}$, see Definition~\ref{def:sca}, which restricts to
symplectization-admissible almost complex structures $J_\pm$ for
$\lambda_\pm$ on $\mc{E}_\pm$.  Then $\lambda$ and $J$ determine a
metric $g$ on $\overline{X}$, as well as the $2$-form $\hat{\omega}$
that appears in \eqref{eqn:tsw4}.  We also need to choose small exact
$2$-forms $\mu_\pm$ on $Y_\pm$, and a small exact 2-form $\mu$ on
$\overline{X}$ which restricts to $\mu_\pm$ on
$\mc{E}_\pm$.

{\em Case 2:\/} The second geometric setup, which is needed for
Proposition~\ref{prop:comp}, considers the composition $(X,\lambda)$
of exact symplectic cobordisms $(X^+,\lambda^+)$ from
$(Y_+,\lambda_+)$ to $(Y_0,\lambda_0)$ and $(X^-,\lambda^-)$ from
$(Y_0,\lambda_0)$ to $(Y_-,\lambda_-)$.  For the purposes of ``neck
stretching'', given $R\ge 0$ consider the diffeomorphic manifold
\begin{equation}
\label{eqn:XR}
X_R = X^- \union_{\{-R\}\times Y_0} ([-R,R]\times Y_0)
\union_{\{R\}\times Y_0} X^+.
\end{equation}
Define the completion $\overline{X_R}$ as usual by attaching ends
$\mc{E}_-=(-\infty,0]\times Y_-$ and $\mc{E}_+=[0,\infty)\times Y_+$
to $X_R$.  We now specify how to write down a version of the
Seiberg-Witten equations \eqref{eqn:tsw4} on $\overline{X_R}$.

To start, define $s_*:\overline{X_R}\to\R$ as follows.  Let
$s_*^\pm:\overline{X^\pm}\to\R$ denote the function defined in Case 1
above.  Then define $s_*$ to agree with $s_*^--R$ on $\mc{E}_-\cup
X^-$, to agree with the $[-R,R]$ coordinate on $[-R,R]\times Y_0$, and
to agree with $s_*^++R$ on $X^+ \cup \mc{E}_+$.

Let $\widetilde{\lambda}^\pm$ denote the 1-form on the completion
$\overline{X^\pm}$ defined in \eqref{eqn:widetildelambda}.  Define a
1-form $\widetilde{\lambda}_R$ on $\overline{X_R}$ by
\begin{equation}
\label{eqn:lambdatilde}
\widetilde{\lambda}_R =
\left\{\begin{array}{cl}
e^{-2R}\widetilde{\lambda}^- & \mbox{on $\mc{E}_-\cup X^-$},\\
e^{2s_*}\lambda_0 & \mbox{on $[-R,R]\times Y_0$},\\
e^{2R}\widetilde{\lambda}^+ & \mbox{on $X^+\cup\mc{E}_+$}.
\end{array}\right.
\end{equation}
When $R$ is fixed, we usually denote $\widetilde{\lambda}_R$ simply by
$\widetilde{\lambda}$.  Define
$\widetilde{\omega}=d\widetilde{\lambda}$ as before.  Note that
$\left(X_R,\widetilde{\lambda}|_{X_R}\right)$ is an exact symplectic
cobordism from $(Y_+,e^{2R}\lambda_+)$ to $(Y_-,e^{-2R}\lambda_-)$.
However below, references to the ``length'' of Reeb orbits on $Y_\pm$
refer to the length as defined by $\lambda_\pm$, which does not depend
on $R$.  We denote this length as usual by $\mc{A}$.

Let $J_\pm$ and $J_0$ be symplectization-admissible almost complex
structures for $\lambda_\pm$ and $\lambda_0$ respectively.  Let
$J^\pm$ be strongly cobordism-admissible almost complex structures on
$X^\pm$ restricting to $J_\pm$ and $J_0$ on the ends.  These determine
a strongly cobordism-admissible almost complex structure $J$ on
$\overline{X_R}$ which agrees with $J^\pm$ on $\mc{E}_\pm\cup X^\pm$,
and which agrees with $J_0$ on $[-R,R]\times Y_0$.

Let $g^\pm$ be the metric on $\overline{X^\pm}$ determined by
$\lambda^\pm$ and $J^\pm$ as in \S\ref{sec:pesc}.  These extend to a
metric $g$ on $\overline{X_R}$ which agrees with $g^\pm$ on
$\mc{E}_\pm\cup X^\pm$, and which on $[-R,R]\times Y_0$ agrees with
the $\R$-invariant metric on $\R\times Y_0$ determined by $\lambda_0$
and $J_0$ according to the conventions in \S\ref{sec:pcf}.  Using the
metric $g$, define $\hat{\omega} \eqdef
\sqrt{2}\widetilde{\omega}/\left|\widetilde{\omega}\right|$ as before.

Finally, let $\mu_\pm$ and $\mu_0$ be small exact $2$-forms on $Y_\pm$
and $Y_0$.  Let $\mu^\pm$ be small exact $2$-forms on $X^\pm$ as in
Case 1 which restrict to $\mu_\pm$ and $\mu_0$ on the ends.  These
determine an exact $2$-form $\mu$ on on $\overline{X_R}$ which
restricts to $\mu^\pm$ on $\mc{E}_\pm\cup X^\pm$, and which restricts
to $\mu_0$ on $[-R,R]\times Y_0$.

Below, when we wish to consider both geometric setups simultaneously,
we let $X_*$ denote $X$ in Case 1 and $X_R$ in Case 2.  Likewise,
$\overline{X_*}$ denotes $\overline{X}$ or $\overline{X_R}$ as appropriate.

\paragraph{Variations in the data.} Proposition~\ref{prop:9x}
considers variations in the given data $(\lambda,J,\mu)$.  To clarify,
fix $\varepsilon>0$ for use in defining neighborhoods as in
\eqref{eqn:N-} and \eqref{eqn:N+} of the positive and negative
boundaries of $X$ in Case 1 or $X^\pm$ in Case 2, and for defining the
data on their completions as in \S\ref{sec:pesc}.  A ``variation''
then consists of data $(\lambda',J',\mu')$ which are constrained to be
usable above for given data $(\lambda_\pm,J_\pm,\mu_\pm)$ (and
$(\lambda_0,J_0,\mu_0)$ in Case 2), with the further requirement that
$\lambda'$ agree with $\lambda$ on the above boundary neighborhoods.
The proposition refers to a ``neighborhood'' of $(\lambda,J,\mu)$;
this consists of data $(\lambda',J',\mu')$ as above in a
$C^\infty$-Frechet neighborhood of $(\lambda,J,\mu)$.

\paragraph{Index and action difference.}
Let
$\frak{d}$ be a instanton solution to \eqref{eqn:tsw4} on
$\overline{X_*}$.  We now introduce two numbers associated to
$\frak{d}$ which will be needed below.

First, let $\frak{i}_{\frak{d}}$ denote the index of the instanton
$\frak{d}$.  This is the Fredholm index of the operator $D_{\frak{d}}$
obtained from linearizing the equations \eqref{eqn:tsw4} at
$\frak{d}$.

Second, recall that the solutions to the perturbed Seiberg-Witten
equations in \eqref{eqn:tsw3} are the critical points of the
``Seiberg-Witten action'' functional \eqref{eqn:a} on the space of
pairs $(A,\psi)$.  As in \S\ref{sec:smoothcob}, let $\frak{c}_\pm$
denote the $s_*\to\pm\infty$ limit of $\frak{d}$.  Let
$\frak{a}_\pm$ denote the $Y_\pm$ version of the action
functional.  We then define
\[
A_\frak{d} \eqdef \frak{a}_-(\frak{c}_-) - \frak{a}_+(\frak{c}_+).
\]
Note that while the functionals $\frak{a}_\pm$ are generally not gauge
invariant, the quantity $A_\frak{d}$ is still gauge invariant.

\paragraph{Spinor decomposition.} If $\psi$ is a section of $\Sp_+$, we write
$\psi=(\alpha,\beta)$, where $\alpha$ and $\beta$ respectively denote the $E$ and
$K^{-1}E$ components of $\psi$ in the decomposition \eqref{eqn:seke4}.

\paragraph{Generalized broken $J$-holomorphic curves.}  If
$\Theta_\pm$ are orbit sets in $Y_\pm$, we define a {\em generalized
  broken $J$-holomorphic curve\/} from $\Theta_+$ to $\Theta_-$ to be
a collection of holomorphic curves $\{C_k\}_{1\le k\le N}$ as in
Definition~\ref{def:broken}, but with one difference: Recall that in
Definition~\ref{def:broken} the curves $C_k$ for $k>k_0$ are in
$\R\times Y_+$, the curve $C_{k_0}$ is in $\overline{X_*}$, and the
curves $C_k$ for $k<k_0$ are in $\R\times Y_-$.  The difference is
that now we do {\em not\/} mod out by $\R$-translation of the curves
$C_k$ in $\R\times Y_\pm$ for $k\neq k_0$.  Note that if $k>k_0$ we
can then identify $C_k\cap([0,\infty)\times Y_+)$ with a subset of
$\overline{X_*}$, and if $k<k_0$ we can likewise identify
$C_k\cap((-\infty,0]\times Y_-)$ with a subset of $\overline{X_*}$.

\begin{proposition}
\label{prop:9x}
Fix a data set consisting of $(\lambda,J,\mu)$.
Let $\mc{K}\ge 1$ be given, and assume that all Reeb
orbits of $\lambda_\pm$ (and $\lambda_0$ in Case 2) of length less
than or equal to $(2\pi)^{-1}\mc{K}$ are nondegenerate.  Then there exist:
\begin{description}
\item{(i)}
$\kappa\ge 1$
\item{(ii)}
A neighborhood of the given data set,
\item{(iii)}
Given $\delta>0$, a number $\kappa_\delta\ge 1$,
\end{description}
such that the following holds: Take $r\ge \kappa_\delta$ and a data
set from the given neighborhood (and take any $R$ in Case 2) so as to
define \eqref{eqn:tsw4} on $\overline{X_*}$.  Let
$\frak{d}=(A,\psi=(\alpha,\beta))$ denote an instanton solution to
this version of \eqref{eqn:tsw4} with $A_{\frak{d}}\le \mc{K}r$ or
$\frak{i}_{\frak{d}}> -\mc{K}r$.  Assume also that $\energy(\frak{c}_+)\le
\mc{K}$.  Then:
\begin{itemize}
\item
$\energy(\frak{c}_-)\le \energy(\frak{c}_+)+\delta$.
\item
Each point in $\overline{X_*}$ where $|\alpha|\le 1-\delta$ has
distance less than $\kappa r^{-1/2}$ from $\alpha^{-1}(0)$.
\item
There exist
\begin{description}
\item{(a)}
a positive integer $N\le\kappa$ and a partition of $\R$ into intervals
$I_1<\cdots<I_N$, each of length at least $2\delta^{-1}$,  with
$[-1,1]\subset I_{k_0}$, and
\item{(b)} a generalized broken $J$-holomorphic curve $\{C_k\}_{1\le k
    \le N}$ in $\overline{X_*}$ from an orbit set $\Theta^+$ in $Y_+$
  to an orbit set $\Theta^-$ in $Y_-$
\end{description}
such that for each $k=1,\ldots,N$, with the above identifications of
subsets of $C_k$ with subsets of $\overline{X_*}$, we have
\[
\sup_{z\in C_k\cap s_*^{-1}(I_k)}\op{dist}(z,\alpha^{-1}(0)) +
\sup_{z\in\alpha^{-1}(0)\cap s_*^{-1}(I_k)}\op{dist}(C_k,z) < \delta.
\]
In particular, $\Theta_\pm$ is the orbit set determined by
$\frak{c}_\pm$ under the map in Proposition~\ref{prop:Liso}(a).
\end{itemize}
\end{proposition}

\begin{remark}
\label{rem:stable}
The constants $\kappa$ and $\kappa_\delta$ in
Proposition~\ref{prop:9x} can be chosen to be {\em stable\/}, by which
we mean that they satisfy the conclusions of the proposition for data
in some neighborhood of the given data $(\lambda,J,\mu)$.  Various
lemmas in the proof of Proposition~\ref{prop:9x} below also refer to
constants which are stable in this sense.  In general, we omit proofs
of stability, as these follow from the proofs below with only cosmetic
changes.
\end{remark}

\subsection{Preliminaries to the proof of Proposition~\ref{prop:9x}}

An analogue of Proposition~\ref{prop:9x} for a symplectization
$\R\times Y$ with $\R$-invariant $(J,\mu)$ was proved in \cite[Prop.\
5.5]{e4}.  A slight difference is that \cite[Prop.\ 5.5]{e4} applies
only to a single data set $(\lambda,J,\mu)$, while
Proposition~\ref{prop:9x} applies to every suitable data set
$(\lambda',J',\mu')$ in some neighborhood of a given $(\lambda,J,\mu)$
and to every $R$ where applicable.  The proof of
Proposition~\ref{prop:9x} below mostly follows the proof of
\cite[Prop.\ 5.5]{e4}, indicating the necessary modifications for our
situation.  Before starting the proof, we need to make a few more
definitions.

\paragraph{The
  spectral flow function.}

Returning to the setting of the beginning of \S\ref{sec:9xprelim},
given $r\ge 1$, a pair $\frak{c}=(A,\psi)$ of connection on $E$ and
section of $\Sp$ determines a self-adjoint operator
$\mc{L}_{r,\frak{c}}$ defined in \cite[Eq.\ (3.8)]{e2}.  Roughly
speaking this operator is the Hessian of $\frak{a}$ at $\frak{c}$
(after modding out by gauge transformations).  Let us call a pair
$(r,\frak{c})$ {\em nondegenerate\/} if the corresponding operator
$\mc{L}_{r,\frak{c}}$ has trivial kernel.

Now fix a reference pair $\frak{c}_*$ such that the pair
$(1,\frak{c}_*)$ is nondegenerate, and fix $r\ge 1$. If $\frak{c}$ is
such that the pair $(r,\frak{c})$ is nondegenerate, then we define the
{\em spectral flow function\/} $f(\frak{c})$ to be the spectral flow
from $\mc{L}_{1,\frak{c}_*}$ to $\mc{L}_{r,\frak{c}}$.

If the spin-c structure has non-torsion first Chern class, i.e.\ if
$c_1(\det(\Sp))$ is not torsion in $H^2(Y;\Z)$, then the functional
$\frak{a}$ is not invariant under the action of the gauge group
$C^\infty(Y;S^1)$, and neither is the spectral flow function $f$.
However the combination
\begin{equation}
\label{eqn:af}
\frak{a}^f(\cdot)\eqdef \frak{a}(\cdot) - 2\pi^2 f(\cdot)
\end{equation}
is always gauge invariant.

\paragraph{Index and spectral flow.}

Returning to the setting of Proposition~\ref{prop:9x}, we now relate
the index of an instanton to the spectral flow functions on $Y_\pm$.
Fix a spin-c structure $\Sp$ on $\overline{X_*}$ and let $E$ be
defined by the splitting \eqref{eqn:seke4}.  Fix a reference pair
$\frak{d}_*=(A_*,\psi_*)$ of connection on $E$ and section of $\Sp_+$
with the following properties: First, the restriction to the $\pm s>1$
portion of $\overline{X_*}$ is pulled back from a configuration
$\frak{c}_{\pm *}$ on $Y_\pm$.  Second, require that the pair
$(r=1,\frak{c}_{\pm *})$ is nondegenerate in the sense described
above.  This guarantees that the operator $D_{\frak{d}_*}$ is
Fredholm.  (Note that this operator is defined regardless of whether
$\frak{d}_*$ solves the Seiberg-Witten equations \eqref{eqn:tsw4}.)
Let $\imath_*$ denote the index of $D_{\frak{d}_*}$.  Let $f_\pm$
denote the spectral flow function on $Y_\pm$ defined using
$\frak{c}_{\pm *}$ as the reference pair.

If $\frak{d}$ is an instanton solution to \eqref{eqn:tsw4} with
nondegenerate $s_*\to\pm\infty$ limits $\frak{c}_\pm$, then it follows
from \cite{aps1} that its index
is given by\footnote{In a symplectization with $\R$-invariant
  $(J,\mu)$, one can take $\frak{d}_*$ to be independent of the $\R$
  factor, so that $\frak{i}_*=0$.  In this case $\frak{i}_{\frak{d}}$
  agrees with the quantity $f_{\frak{d}}$ in \cite{e4}.
}
\begin{equation}
\label{eqn:if}
\frak{i}_{\frak{d}} = \frak{i}_* + f_+(\frak{c}_+) - f_-(\frak{c}_-).
\end{equation}

\subsection{Estimates on instantons}
\label{sec:estimates}

To begin the proof of Proposition~\ref{prop:9x}, we now establish
various estimates for instanton solutions to \eqref{eqn:tsw4} on
$\overline{X_*}$, parallel to \cite[\S3]{e4}, where analogous
estimates are derived for instantons on a symplectization.  Assume in
what follows that $(\lambda_\pm,J_\pm,\mu_\pm)$ (and
$(\lambda_0,J_0,\mu_0,R)$ in Case 2) are given.  Fix data
$(\lambda,J,\mu)$ as in \S\ref{sec:9xprelim}.  Below, $c_0$ denotes a
number that is greater than $1$, that is stable in the sense of
Remark~\ref{rem:stable}, and that does not depend on any given
solution to \eqref{eqn:tsw3} or \eqref{eqn:tsw4} or on the value of
$r$ used to define these equations.  The value of $c_0$ can increase
from one appearance to the next.

\begin{lemma}
\label{lem:3.1x}
  (cf.\ \cite[Lem.\ 3.1]{e4}) There exists a stable $\kappa\ge 1$ such
  that if $r\ge \kappa$ and if $(A,\psi=(\alpha,\beta))$ is an
  instanton solution to \eqref{eqn:tsw4} on $\overline{X_*}$, then
\[
\begin{split}
|\alpha| &\le 1+\kappa r^{-1},\\
|\beta|^2 &\le \kappa r^{-1}(1-|\alpha|^2) + \kappa^2r^{-2}.
\end{split}
\]
\end{lemma}

\begin{proof}
  This follows from the maximum principle as in \cite[Lem.\ 3.1]{e4},
  using the corresponding inequalities in the 3-dimensional case
  \cite[Lemm.\ 2.2]{tw1} to obtain the necessary bounds as
  $s\to\pm\infty$.
\end{proof}

\begin{lemma}
\label{lem:3.2x}
(cf.\ \cite[Lem.\ 3.2]{e4})
There exists a stable $\kappa\ge 1$ with the following property:
Suppose that $r\ge \kappa$ and that $\frak{d}=(A,\psi)$ is an
instanton solution to \eqref{eqn:tsw4} on $\overline{X_*}$ with
$A_{\frak{d}}\le r^2$ or $\frak{i}_\frak{d}\ge -r^2$.  Then
$|F_A|\le \kappa r$.
\end{lemma}

\begin{proof}
  Copy the proof of \cite[Lem.\ 3.2]{e4}, replacing \cite[Lem.\
  3.3]{e4} in that argument with Lemma~\ref{lem:3.3x} below.
\end{proof}

To state the next lemma, let $\nabla_A$ denote the covariant
derivative on $\Sp_+=E\oplus K^{-1}E$ determined by the connection $A$
on $E$ together with the distinguished connection $A_{K^{-1}}$ on
$K^{-1}$ from \S\ref{sec:pesc}.  Note that under the identification
\eqref{eqn:connections}, the difference
$\nabla_A-\nabla_\A\in\Omega^1(\overline{X_*};\op{End}(\Sp_+))$ is
bounded in $C^0$ and does not depend on $A$.

\begin{lemma}
\label{lem:3.3x}
(cf.\ \cite[Lem.\ 3.3]{e4}) There exists a stable $\kappa\ge 1$ with
the following property: Suppose that $r\ge\kappa$ and that
$\frak{d}=(A,\psi)$ is an instanton solution to \eqref{eqn:tsw4} on
$\overline{X_*}$ with $A_{\frak{d}}\le r^2$ or $\frak{i}_\frak{d}\ge
-r^2$.  Let $I\subset\R$ denote an interval of length $2$.  Then
\[
\int_{s_*^{-1}(I)}(|F_A|^2 + r|\nabla_A\psi|^2) \le \kappa r^2.
\]
\end{lemma}

The proof of Lemma~\ref{lem:3.3x} requires two additional lemmas.  To
state these, let $I_0=\{0\}$ in Case 1 and let $I_0=\{-R,R\}$ in Case 2.
On $s_*^{-1}(\R\setminus I_0)$, define two $1$-forms by
\begin{align}
\label{eqn:mcb}
\mc{B}_{(A,\psi)} &\eqdef {*}F_A - r(\tau(\psi) - ia) - i{*}\mu +
\frac{1}{2}{*}F_{A_K^{-1}},\\
\nonumber
E_A &\eqdef F_A\left(\frac{\partial}{\partial s},\cdot\right).
\end{align}
Here $*$ denotes the three-dimensional Hodge star, $a$ denotes the
relevant contact form ($\lambda_\pm$ or $\lambda_0$), and $s$ denotes
the $\R$ coordinate on $\R\times Y_\pm$ or $\R\times Y_0$.  Also let
$\nabla_{A,s}$ denote the covariant derivative with respect to the
connection $\nabla_A$ on $\Sp_+$ in the direction $\partial/\partial
s$. We then have:

\begin{lemma}
\label{lem:3.4x}
(cf.\ \cite[Lem.\ 3.4]{e4}) There exists a stable $\kappa\ge 1$ such
that if $r\ge\kappa$ and if $\frak{d}=(A,\psi)$ is an instanton
solution to \eqref{eqn:tsw4}, then the following hold:
\begin{description}
\item{(a)}
Suppose that $s_+\ge s_-$ are in the same component of $\R\setminus
I_0$.  Then
\begin{equation}
\label{eqn:3.4}
\begin{split}
&\frak{a}(\frak{d}|_{s_*=s_-}) - \frak{a}(\frak{d}|_{s_*=s_+})
\\
&\quad
=\frac{1}{2}\int_{s_*\in[s_-,s_+]} (|E_A|^2 + |\mc{B}_{(A,\psi)}|^2 +
2r(|\nabla_{A,s}\psi|^2 + |D_{A(s_*)}\psi|^2)).
\end{split}
\end{equation}
Here $\frak{a}$ denotes the functional \eqref{eqn:a} on $Y_+$, $Y_-$,
or $Y_0$ as appropriate.
\item{(b)}
In Case 1,
\begin{equation}
\label{eqn:desiredb}
\kappa^{-1}\int_X\left(|F_A|^2 + r|\nabla_A\psi|^2\right) \le
-\frak{a}\left(\frak{d}|_{\partial X}\right) + \kappa r.
\end{equation}
where $\frak{a}\left(\frak{d}|_{\partial X}\right) \eqdef
\frak{a}\left(\frak{d}|_{\{0\}\times Y_+}\right) -
\frak{a}\left(\frak{d}|_{\{0\}\times Y_-}\right)$.  In Case 2,
analogous inequalities hold with $X$ replaced by $X^-$ or $X^+$.
\item{(c)} If $s_+,s_-\in\R\setminus I_0$ and $s_+>s_-$ then
\[
\kappa^{-1}\int_{s_*^{-1}[s_-,s_+]}(|F_A|^2 + 2r|\nabla_A\psi|^2) \le
\frak{a}(\frak{d}|_{s_*=s_-}) - \frak{a}(\frak{d}|_{s_*=s_+}) + \kappa
(s_+-s_-) r^2 + \kappa r.
\]
\item{(d)}
\[
\frac{1}{2}\int_{s_*^{-1}(\R\setminus I_0)} (|E_A|^2 +
|\mc{B}_{(A,\psi)}|^2 + 2r(|\nabla_{A,s}\psi|^2 + |D_A\psi|^2)) \le
\frak{a}(\frak{c}_-) - \frak{a}(\frak{c}_+) + \kappa r.
\]
\end{description}
\end{lemma}

\begin{proof}
  (a) We can apply a gauge transformation to put the connection $A$
  into temporal gauge \eqref{eqn:tg} on $s_*^{-1}[s_-,s_+]$.  Equation
  \eqref{eqn:3.4} then becomes
\begin{equation}
\label{eqn:gradient}
\begin{split}
  &\frak{a}(\frak{d}|_{s_*=s_-}) - \frak{a}(\frak{d}|_{s_*=s_+})
  \\
  &\quad =\frac{1}{2}\int_{s_*^{-1}[s_-,s_+]} \left(\left|\frac{\partial
        A}{\partial s_*}\right|^2 + |\mc{B}_{(A,\psi)}|^2 +
    2r\left(\left|\frac{\partial \psi}{\partial s_*}\right|^2 +
      |D_{A(s_*)}\psi|^2\right)\right) .
\end{split}
\end{equation}
This is equivalent to the first equation in \cite[Lem.\ 3.4]{e4}.  An
alternate way to understand this equation is to recall that
$(A(s_*),\psi(s_*))$ is a downward gradient flow line of the
functional $\frak{a}$ in \eqref{eqn:a}.  In particular, the $L^2$
gradient of $\frak{a}$ at $(A,\psi)$ is
$(\mc{B}_{(A,\psi)},\sqrt{2r}D_A\psi)$.  Equation \eqref{eqn:gradient}
then follows from the fact that if $\gamma(s)$ is a downward gradient flow
line of a function $f$ then
\begin{equation}
\label{eqn:gradient2}
f(s_-)-f(s_+) = \frac{1}{2}\int_{s\in[s_-,s_+]}\left(\|\nabla f\|^2 +
  \|\partial \gamma/\partial s\|^2\right).
\end{equation}

(b) We just consider Case 1 since the proof in Case 2 is the same.
Recall that our solution $(A,\psi)$ to \eqref{eqn:tsw4} corresponds to
a solution $(\A,\Psi)$ of \eqref{eqn:sw4} via \eqref{eqn:connections},
\eqref{eqn:rescaling}, and \eqref{eqn:perturbation4}. 
Identify $Y_\pm$ with $\{0\}\times Y_\pm$ in $\overline{X}$.  By
\eqref{eqn:compareFunctionals}, $\frak{a}(\frak{d}_{\partial X})$
differs by an $O(r)$ constant from $\frak{a}_\eta(\frak{d}_{\partial
  X}) \eqdef \frak{a}_{\eta_+}(\frak{d}_{Y_+}) -
\frak{a}_{\eta_-}(\frak{d}_{Y_-})$, so it is enough to prove the claim
with $\frak{a}(\frak{d}_{\partial X})$ replaced by
$\frak{a}_\eta(\frak{d}_{\partial X})$.

Recall from \S\ref{sec:pesc} that
$\hat{\omega}=\sigma^{-1}d\widetilde{\lambda}$ where $\sigma:X\to
[3/2,5/2]$ is a smooth function with $\sigma|_{\partial X}=2$.  Now
start with the Bochner-Weitzenb\"ock formula
\[
D_\A^*D_\A\Psi = \nabla_\A^*\nabla_\A\Psi +
\frac{1}{2}\op{cl}(F_\A^+)\Psi + \frac{s}{4}\Psi
\]
where $s$ denotes the scalar curvature of $X$.  Putting in the Dirac
equation $D_\A\Psi=0$ from \eqref{eqn:sw4}, multiplying the resulting
equation by $\sigma$, taking the inner product with $\Psi$ and
integrating by parts gives
\begin{align}
\nonumber
0 =& \int_X\sigma|\nabla_\A\Psi|^2 +
\frac{1}{2}\int_X\sigma\langle\op{cl}(F_\A^+)\Psi,\Psi\rangle +
\frac{1}{4}\int_X \sigma s |\Psi|^2 + \int_X\langle
d\sigma\tensor\Psi,\nabla_\A\Psi\rangle \\
\label{eqn:ibp1}
& + 2\int_{Y_+}\langle D_{\A|_{Y_+}}\Psi,\Psi\rangle -
2\int_{Y_-}\langle D_{\A|_{Y_-}}\Psi,\Psi\rangle.
\end{align}
Second, taking the norm square of the curvature equation in
\eqref{eqn:sw4}, multiplying by $\sigma$, and integrating over $X$
gives
\begin{align}
\label{eqn:ibp2}
0 = & \int_X\sigma|F_\A^+|^2 + \frac{1}{2}\int_X\sigma|\Psi|^4 +
\int_X\sigma|\eta|^2 \\
\nonumber
& - \int_X\sigma\langle\op{cl}(F_\A^+)\Psi,\Psi\rangle -
2\int_X\sigma\langle F_\A^+,i\eta\rangle +
\int_X\sigma\langle\op{cl}(i\eta)\Psi,\Psi\rangle.
\end{align}
Third, by \eqref{eqn:aeta} and Stokes' theorem we have
\begin{align}
\label{eqn:ibp3}
\frak{a}_\eta(\frak{d}|_{\partial X}) =&\;
\frac{1}{8}\int_{X} \left(|F_\A^+|^2
  - |F_\A^-|^2 + F_{\A_0}\wedge F_{\A_0}\right) \\
\nonumber &+\frac{1}{2}\int_{Y_+}\langle D_{\A|_{Y_+}}\Psi,\Psi\rangle -
\frac{1}{2}\int_{Y_-}\langle D_{\A|_{Y_-}}\Psi,\Psi\rangle\\
\nonumber
&+ \frac{1}{4}\int_{Y_+}(\A-\A_0)\wedge i\eta_+ -
\frac{1}{4}\int_{Y_-}(\A-\A_0)\wedge i\eta_-.
\end{align}
Here $\A_0$ is any reference connection on $\det(\Sp)$ over $X$ extending the
chosen reference connections over $Y_+$ and $Y_-$; and $\eta_\pm$
denotes the perturbation \eqref{eqn:perturbation} for $Y_\pm$.

Adding two times equation \eqref{eqn:ibp1} to equation
\eqref{eqn:ibp2} and subtracting eight times equation \eqref{eqn:ibp3}
gives
\begin{align}
\nonumber
-8 \frak{a}_\eta(\frak{d}|_{\partial X}) =&\;
2\int_X\sigma|\nabla_\A\Psi|^2 +
\int_X\left((\sigma-1)|F_\A^+|^2 + |F_\A^-|^2\right) +
\int_X\langle d\sigma\tensor
\Psi,\nabla_\A\Psi\rangle\\
\label{eqn:ibp4} &
+ \frac{1}{2}\int_X\sigma s |\Psi|^2 -\int_X
F_{\A_0}\wedge F_{\A_0} 
\\
\nonumber &+ \frac{1}{2}\int_X\sigma|\Psi|^4 +
\int_X\sigma\langle\op{cl}(i\eta)\Psi,\Psi\rangle
+ \int_X\sigma|\eta|^2  \\
\nonumber
&- 2\int_X\sigma\langle F_\A^+,i\eta\rangle -
2\int_{Y_+}(\A-\A_0)\wedge i\eta_+ + 2\int_{Y_-}(\A-\A_0)\wedge i\eta_-.
\end{align}
On the right side of \eqref{eqn:ibp4}, in the first term we have
$|\nabla_\A\Psi|^2=2r|\nabla_A\psi|^2 + O(r)$, since
$(2r)^{-1/2}|\Psi|=|\psi|=O(1)$ by Lemma~\ref{lem:3.1x}; in the
second term we have $F_\A=2F_A+O(1)$; and in the third term we have
\[
\langle d\sigma\tensor
\Psi,\nabla_\A\Psi\rangle \ge -\frac{1}{100}|\nabla_\A\Psi|^2 - c_0|\Psi|^2,
\]
where $|\Psi|^2$ is $O(r)$ by Lemma~\ref{lem:3.1x}.
The second line on the right side of \eqref{eqn:ibp4} is $O(r)$ by
Lemma~\ref{lem:3.1x} again.  Using $\Psi=\sqrt{2r}(\alpha,\beta)$ and
\eqref{eqn:perturbation4}, we can expand the sum of the integrands in
the third line of the right side of \eqref{eqn:ibp4} as $\sigma$ times
\[
\begin{split}
\frac{1}{2}|\Psi|^4 + \langle\op{cl}(i\eta)\Psi,\Psi\rangle + |\eta|^2
=& 2r^2((|\alpha|^2-1)^2 + 2|\alpha|^2|\beta|^2 + 2|\beta|^2 +
|\beta|^4) \\
& + \langle \op{cl}(2i\mu_*)\Psi,\Psi\rangle - 4r\langle
\hat{\omega},\mu_*\rangle + 4|\mu_*|^2,
\end{split}
\]
which is $O(r)$ by Lemma~\ref{lem:3.1x}.  Since the 2-forms $\eta_\pm$
on $Y_\pm$ extend over $X$ to the exact $2$-form
$-rd\widetilde{\lambda}+4\mu$, the fourth line on the right side of
\eqref{eqn:ibp4} can be rewritten using Stokes's theorem as
\begin{equation}
\label{eqn:ibp5}
2\int_X F_{\A_0}\wedge i(rd\widetilde{\lambda}-4\mu)
+ \int_X\left(\langle F_\A^+,4i\sigma\mu\rangle + \langle
  F_\A^-,8i\mu\rangle\right).
\end{equation}
The first term in \eqref{eqn:ibp5} is $O(r)$.  Since we assumed in
\S\ref{sec:pesc} that $|\mu|\le 1/100$, the second term in
\eqref{eqn:ibp5} is bounded from below by
$\frac{-1}{10}\int_X(|F_\A^+|^2+|F_\A^-|^2)$, so we can combine this
with the second term in the first line on the right hand side of
\eqref{eqn:ibp4} to obtain the desired inequality \eqref{eqn:desiredb}.

(c) By part (b), it is enough to show that the stated inequality holds
when $s_+$ and $s_-$ are in the same component of $\R\setminus I_0$.
We can further replace the functional $\frak{a}$ by
$\frak{a}_{\eta_0}$, where $\eta_0$ denotes the perturbation
\eqref{eqn:perturbation} for $Y_+$, $Y_-$, or $Y_0$ as appropriate.

As in \eqref{eqn:gradient2}, we have
\[
\begin{split}
  \frak{a}_{\eta_0}(\frak{d}_{s_*=s_-}) -
  \frak{a}_{\eta_0}(\frak{d}_{s_*=s_+}) = &
  \frac{1}{2}\int_{s_*^{-1}[s_-,s_+]}\left(
    \frac{1}{4}\left|-{*}F_{\A(s_*)} + \tau(\Psi(s_*)) +
      i{*}\eta_0\right|^2\right.\\
  & \left.\quad\quad + |D_{\A(s_*)}\Psi|^2 + \frac{1}{4}|\partial_s\A|^2 +
    |\partial_s\Psi|^2\right).
\end{split}
\]
Expanding the first term in the integrand, and using the
Bochner-Weitzenb\"ock formula for the three-dimensional Dirac operator
$D_{\A(s_*)}$ on constant $s_*$ slices to expand the second term in
the integrand, the right hand side becomes
\begin{gather*}
\frac{1}{2}\int_{s_*^{-1}[s_-,s_+]}\left(
  \frac{1}{4}|F_\A|^2
+ |\nabla_{\A}\Psi|^2
+ \frac{s}{4}|\Psi|^2
+ \frac{1}{4}|\tau(\Psi(s_*))|^2
\right.\\ \quad\quad\quad\left.
+ \frac{1}{4}|\eta_0|^2
+ \frac{1}{2}\langle \tau(\Psi(s_*)),i{*}\eta_0\rangle
- \frac{1}{2}\langle *F_{\A(s_*)},i{*}\eta_0\rangle
\right).
\end{gather*}
The sum of the first two terms in the integrand is $|F_A|^2 +
2r|\nabla_A\psi|^2 + O(r)$.  The third and fourth terms are $O(r)$ by
Lemma~\ref{lem:3.1x}, the fifth term is $O(r^2)$ by
\eqref{eqn:perturbation}, and likewise the sixth term is $O(r^{3/2})$.
The last term is $O(r^2)$ because $|{*}F_{\A(s_*)}| \le 2|F_\A^+|$,
which is $O(r)$ as noted in the proof of (b).

(d) This follows immediately from (a) and (b).
\end{proof}

Continuing with the proof of Lemma~\ref{lem:3.3x}, note that the case
$A_\frak{d}\le r^2$ follows immediately from Lemma~\ref{lem:3.4x}(a)--(c).
To deal with the remaining cases we need:

\begin{lemma}
\label{lem:3.5x}
  (cf.\ \cite[Lem. 3.5]{e2}) There exists a stable constant $\kappa\ge
  1$ such that if $\frak{d}=(A,\psi)$ is an instanton solution to
  \eqref{eqn:tsw4}, then
\[
\begin{split}
\frak{a}(\frak{c}_-) - \frak{a}(\frak{c}_+) \le & -2\pi^2
\frak{i}_\frak{d} + \frac{r}{2}(E(A_+)-E(A_-)) \\
&+ \kappa
r^{2/3}(\ln r)^\kappa\left(1 + |E(A_+)|^{4/3} + |E(A_-)|^{4/3}\right).
\end{split}
\]
\end{lemma}

\begin{proof}
By equations \eqref{eqn:af} and \eqref{eqn:if} we have
\[
\frak{a}(\frak{c}_-)-\frak{a}(\frak{c}_+) = \frak{a}^f(\frak{c}_-) -
\frak{a}^f(\frak{c}_+) - 2\pi^2(\frak{i}_{\frak{d}} - \frak{i}_*).
\]
The lemma then follows from \cite[Prop.\ 4.10]{e1}.
\end{proof}

\begin{proof}[Proof of Lemma~\ref{lem:3.3x}.]
  Using Lemmas~\ref{lem:3.4x}(a),(d) and \ref{lem:3.5x}, the arguments
  in the proof of \cite[Lem.\ 3.3]{e4} establish the asertions of
  Lemma~\ref{lem:3.3x} if $\op{dist}(I,I_0)\ge T=c_0(\ln r)^{c_0}$.
  To deal with the remaining cases, we will restrict to Case 1, as the
  proof in Case 2 is very similar.  By what was just said, there
  exist points $s_-\in [-T-2,-T]$ and $s_+\in[T,T+2]$ such that
\begin{equation}
\label{eqn:sliceInequality}
\int_{s_*=s_\pm}\left(|F_A|^2 + r|\nabla_A\psi|^2\right) \le c_0 r^2.
\end{equation}

Now let $A_{E_\pm}$ denote the reference connection on $E|_{Y_\pm}$
used to define the functional $\frak{a}$ in \eqref{eqn:a} for $Y_\pm$.
It is convenient below to choose the reference connection $A_{E_\pm}$
so that $F_{A_{E_{\pm}}}+\frac{1}{2}F_{A_{K^{-1}}}$ is harmonic on
$Y_\pm$.  Choose identifications of ${E_-}|_{(-\infty,0]\times Y_-}$
and ${E_+}|_{[0,\infty)\times Y_+}$ with the pullbacks of $E_-$ and
$E_+$ respectively.  Extend $A_{E_\pm}$ to a reference connection
$A_E$ on $E$ over $\overline{X}$ such that on $(-\infty,0]\times Y_-$
and $[0,\infty)\times Y_+$, with respect to the above identifications,
we have
\begin{equation}
\label{eqn:AE}\nabla_{A_E}=\partial_s + \nabla_{A_{E_\pm}}.
\end{equation}
Let
\begin{equation}
\label{eqn:nuE}
\nu_E\eqdef \frac{i}{\pi}\left(F_{A_E} +
  \frac{1}{2}F_{A_{K^{-1}}}\right).
\end{equation}
This is a closed form which represents the cohomology class
$\frac{1}{2}c_1(\frak{s})$.  Write $A=A_E+\hat{a}$.  We claim that
\begin{equation}
\label{eqn:actionDifference}
\left| \frak{a}(\frak{d}|_{s_*=s_-}) -
  \frak{a}(\frak{d}|_{s_*=s_+})\right| \le c_0 \left| \int_{s_*=s_-}
  \hat{a}\wedge \nu_E - \int_{s_*=s_+}\hat{a}\wedge \nu_E\right| + c_0r^2.
\end{equation}

To prove \eqref{eqn:actionDifference}, note that the functional
$\frak{a}$ on $Y_\pm$ is invariant under nullhomotopic gauge
transformations.  Thus to evaluate $\frak{a}(\frak{d}_{s_*=s_+})$ or
$\frak{a}(\frak{d}_{s_*=s_-})$, we may assume that
$\hat{a}|_{s_*=s_\pm}=\sigma + \hat{a}'$ where $\sigma$ is harmonic,
$\hat{a}'$ is co-closed, and $\hat{a}'$ is orthogonal to the space of
harmonic forms on $Y_\pm$, so that $\int_{s_*=s_\pm}\nu_E\wedge
\hat{a} = \int_{s_*=s_\pm}\nu_E\wedge\sigma$, and
$\|\hat{a}'\|_{L^2_1}\le c_0\|d\hat{a}'\|_{L^2}$.  By
\eqref{eqn:sliceInequality}, the last inequality implies that
$\|\hat{a}'\|_{L^2_1}\le c_0r$, and \eqref{eqn:sliceInequality} also
implies that $\|\nabla_A\psi|_{s_*=s_\pm}\|_{L^2}\le c_0r^{1/2}$.
Putting these last two inequalities and Lemma~\ref{lem:3.1x} into
\eqref{eqn:a}, we obtain
\[
\frak{a}(\frak{d}|_{s_*=s_\pm}) =
i\pi\int_{s_*=s_\pm}\hat{a}\wedge\nu_E + O(r^2),
\]
from which \eqref{eqn:actionDifference} follows.

Applying Stokes's theorem to the right hand side of
\eqref{eqn:actionDifference}, and using the fact that $|\nu_E|$ and
$|F_{A_E}|$ enjoy $r$-independent upper bounds, we obtain
\[
\frak{a}(\frak{d}|_{s_*=s_-}) -
  \frak{a}(\frak{d}|_{s_*=s_+}) \le
  c_0T^{1/2}\left(\int_{s_*^{-1}[s_-,s_+]}|F_A|^2\right)^{1/2} + c_0T +
  c_0r^2.
\]
It follows from this and Lemma~\ref{lem:3.4x}(c) that if $r$ is larger than some
stable constant, then
\[
\int_{s_*^{-1}[s_-,s_+]}|F_A|^2 \le
c_0 T r^2.
\]
Putting this inequality back into the previous one, we obtain
\[
\frak{a}(\frak{d}|_{s_*=s_-}) -
\frak{a}(\frak{d}|_{s_*=s_+})
\le
c_0 r^2.
\]
It follows from this and Lemma~\ref{lem:3.4x}(a),(b) that
\begin{equation}
\label{eqn:actionBound}
\frak{a}(\frak{d}|_{s_*=s_-'}) -
\frak{a}(\frak{d}|_{s_*=s_+'})
\le
c_0 r^2
\end{equation}
whenever $s_-\le s_-' \le s_+'\le s_+$.  When $s_+'-s_-'=2$, using
Lemma~\ref{lem:3.4x}(c) with \eqref{eqn:actionBound} proves the
remaining cases of Lemma~\ref{lem:3.3x}.
\end{proof}

We will also need counterparts of the estimates in \cite[Lemmas
3.6--3.10, 4.3]{e4}.  Since these lemmas and their proofs carry over
almost verbatim to our setting, we will not repeat them here, but just
note the following modifications: First, the constants $\kappa$
provided by these lemmas are stable.  The constant $\kappa_q$ provided
by \cite[Lemma 3.6]{e4} is also stable, although the neighborhood of
stability may depend on $q$.  Second, $[x,y]\times M$ is to be
replaced by $s_*^{-1}[x,y]$, and $\R\times M$ is to be replaced by
$\overline{X_*}$.  Third, $f_\frak{d}$ is to be replaced by
$\frak{i}_\frak{d}$.  Finally, $\frac{\partial}{\partial s}A \pm B_A$
is to be replaced by $F_A^\pm$, the (anti-)self-dual part of the
curvature $F_A$.

\subsection{Instantons and holomorphic curves}
\label{sec:ih1}

We now establish counterparts of results from \cite[\S4]{e4}.  The
latter explains how instantons can be used to define parts of
holomorphic curves.

We need to introduce the following notation.  Continue with the
geometric setup from \S\ref{sec:9xprelim}.
If $(A,\psi=(\alpha,\beta))$ is an instanton solution to the perturbed
Seiberg-Witten equations \eqref{eqn:tsw4} on $\overline{X_*}$, define
a function $\underline{M}:\R\to \R$ by
\[
\underline{M}(s) \eqdef
r\int_{s_*^{-1}[s-1,s+1]}\left(1-|\alpha|^2\right).
\]
The idea of this function is that an $r$-independent upper bound on
$\underline{M}$ will allow us to find, for large $r$, a holomorphic
curve near the zero locus of $\alpha$ whose area in
$s_*^{-1}[s-1,s+1]$ is approximately $\frac{1}{2\pi}\underline{M}(s)$.
In particular, the propositions to follow assume certain upper bounds
on $\underline{M}$; we will establish upper bounds on $\underline{M}$
later in \S\ref{sec:ih2}.

Meanwhile, continuing with the notation, define a connection
$\widehat{A}$ on $E$ by
\[
\widehat{A} \eqdef A - \frac{1}{2}\left(\overline{\alpha}\nabla_A\alpha -
  \alpha\nabla_A\overline{\alpha}\right).
\]
Note that this has curvature
\[
F_{\widehat{A}} = (1-|\alpha|^2)F_A - \nabla_A\overline{\alpha} \wedge
\nabla_A\alpha.
\]
Below, on the subsets of $\overline{X_*}$ that are identified with
$(-\infty,0]\times Y_-$ or $[0,\infty)\times Y_+$, or $[-R,R]\times
Y_0$ in Case 2, let $\lambda_Y$ denote the pullback of the relevant
contact form $\lambda_\pm$ or $\lambda_0$ on $Y_\pm$ or $Y_0$.

\begin{proposition}
\label{prop:4.1x}
(cf.\ \cite[Prop.\ 4.1]{e4}) Given $\delta>0$ and $\mc{K}\ge 1$, there
exist a stable $\kappa\ge 1$ and a neighborhood of the given data set
$(\lambda,J,\mu)$ such that the following holds: Let $r\ge\kappa$, and
let $\frak{d}=(A,\psi=(\alpha,\beta))$ be an instanton solution to
\eqref{eqn:tsw4} as defined with a data set from this neighborhood of
$(\lambda,J,\mu)$ (and any $R$ in Case 2).  Assume that $A_\frak{d}\le
r^2$ or $\frak{i}_\frak{d}\ge -r^2$.  Let $\mathbb{I}$ be a connected
subset of $\R$ of length at least $2\delta^{-1}+16$ such that
$\underline{M}(\cdot)\le \mc{K}$ on $\mathbb{I}$.  Let
$I\subset\mathbb{I}$ be a connected set of points with distance at
least $7$ from the boundary of $\mathbb{I}$ and length $2\delta^{-1}$.
Then:
\begin{itemize}
\item Each point in $s_*^{-1}(I)$ where $|\alpha|\le 1-\delta$ has
  distance at most $\kappa r^{-1/2}$ from $\alpha^{-1}(0)$.
\item There exists a finite set $\mc{C}$ of pairs of the form $(C,m)$
  where $C$ is a closed, irreducible $J$-holomorphic subvariety in a
  neighborhood of the closure of $s^{-1}(I)$, $m$ is a positive
  integer, and the subvarieties $C$ for diferent elements of $\mc{C}$
  are distinct, such that:
\begin{description}
\item{(a)}
\end{description}
\[
\sup_{z\in \cup_{(C,m)\in\mc{C}}C \cap s_*^{-1}(I)}
\op{dist}(z,\alpha^{-1}(0)) + \sup_{z\in \alpha^{-1}(0)\cap
  s_*^{-1}(I)}\op{dist}\left(z,\cup_{(C,m)\in\mc{C}}C\right) < \delta.
\]
\begin{description}
\item{(b)} Let $\nu$ be a $2$-form on $\overline{X_*}$ with support in
  $s_*^{-1}(I)$ such that $|\nu|\le 1$ and
  $|\nabla\nu|\le \delta^{-1}$.  Then
\[
\left|
\frac{i}{2\pi}\int_{s_*\in I}\nu\wedge F_{\widehat{A}} -
\sum_{(C,m)\in\mc{C}} m \int_C\nu
\right|
\le \delta.
\]
\item{(c)} Let $I_Y\subset I$ denote a component of the subset of $I$
  where the distance to $I_0$ is at least $2$.  Then
\[
\sum_{(C,m)\in\mc{C}} m\int_{C\cap s_*^{-1}(I_Y)}d\lambda_Y\le \kappa.
\]
\end{description}
\end{itemize}
\end{proposition}

\begin{proof}
  The proof of the first bullet differs only in notation from the
  proof of the first bullet of \cite[Prop.\ 4.1]{e4}, except for the
  following change: Violation of the first bullet requires sequences
  $\{(\lambda_n,J_n,\mu_n)\}_{n=1,2,\ldots}$ and
  $\{r_n,A_n,\psi_n)\}_{n=1,2\ldots}$, as well a sequence
  $\{R_n\}_{n=1,2,\ldots}$ in Case 2, where
  $\{(\lambda_n,J_n,\mu_n)\}_{n=1,2,\ldots}$ converges to
  $(\lambda,J,\mu)$, the pair $(A_n,\psi_n)$ solves the version of
  \eqref{eqn:tsw4} defined using $(\lambda_n,J_n,\mu_n,r_n,R_n)$, and
  the rest of the assumptions on the sequence are the same as in
  \cite{e4}.  Note that the {\em stability\/} of the constants
  $\kappa$ provided by the lemmas in \S\ref{sec:estimates} must be
  used to obtain the contradiction that proves the first bullet.

  The proof of the second bullet is obtained by similarly modifying
  the proof of the second bullet of \cite[Prop.\ 4.1]{e4}, using the
  stability of the constants $\kappa$, and making the usual notational
  changes to replace $\R\times M$ in \cite{e4} by $\overline{X_*}$
  here.  In particular, the form $ds\wedge a + \frac{1}{2} {*} a$ in
  \cite[Eq.\ (4.5)]{e4} is to be replaced by $\hat{\omega}$ here.
\end{proof}

The following proposition is similar to Proposition~\ref{prop:4.1x}, but with the interval $I$ expanded.

\begin{proposition}
\label{prop:4.5x}
(cf. \cite[Prop.\ 4.5]{e4}) Given $\mc{K}\ge 1$, suppose that each
Reeb orbit with length $\le \mc{K}/2\pi$ of $\lambda_\pm$, and of
$\lambda_0$ in Case 2, is nondegenerate.  Then there exists $\kappa\ge
1$, and given $\delta>0$ there exist $\kappa_\delta\ge 1$ and a
neighborhood of the given data set $(\lambda,J,\mu)$ such that the
following holds: Suppose that $r\ge \kappa_\delta$ and that
$\frak{d}=(A,\psi=(\alpha,\beta))$ is an instanton solution to
\eqref{eqn:tsw4} as defined with a data set from this neighborhood of
$(\lambda,J,\mu)$ (and any $R$ in Case 2).  Assume that $A_\frak{d}\le
r^2$ or $\frak{i}_\frak{d}\ge -r^2$.  Let $\mathbb{I}\subset\R$ be a
connected set of length at least $4\delta^{-1}+16$ such that
$\underline{M}(\cdot)\le \mc{K}$ on $\mathbb{I}$.  Assume in addition
that $I_0\cap\mathbb{I}$ has distance at least
$\frac{4}{3}\delta^{-1}$ from $\partial\mathbb{I}$.  Let $I\subset
\mathbb{I}$ denote the set of points with distance at least $7$ from
the boundary of $\mathbb{I}$.  Then:
\begin{itemize}
\item Each point in $s_*^{-1}(I)$ where $|\alpha|\le 1-\delta$ has
  distance less than $\kappa r^{-1/2}$ from $\alpha^{-1}(0)$.
\item
There exist:
\begin{description}
\item{(1)} A positive integer $N\le \kappa$ and a cover
  $I=\cup_{k=1}^NI_k$ where each $I_k$ is a connected open set of
  length at least $2\delta^{-1}$, such that $I_k\cap I_{k'}=\emptyset$
  when $|k-k'|>1$.  If $|k-k'|=1$, then $I_k\cap I_{k'}$ has length
  between $\frac{1}{128}\delta^{-1}$ and $\frac{1}{64}\delta^{-1}$.
  Finally, each boundary point of each $I_k$ has distance at least
  $\delta^{-1}$ from $I_0\cap\mathbb{I}$.  
\item{(2)} For each $k\in\{1,\ldots,N\}$, a finite set
  $\mc{C}_k$ of pairs $(C,m)$ where $m$ is a positive integer and $C$
  is a closed irreducible $J$-holomorphic subvariety in a neighborhood
  of $s_*^{-1}(I_k)$.  The subvarieties $C$ for
  different elements of $\mc{C}_k$ are distinct.
\end{description}
These are such that:
\begin{description}
\item{(a)}
\end{description}
\[
\sup_{z\in \cup_{(C,m)\in\mc{C}_k}C \cap s_*^{-1}(I_k)}
\op{dist}(z,\alpha^{-1}(0)) + \sup_{z\in \alpha^{-1}(0)\cap
  s_*^{-1}(I_k)}\op{dist}\left(z,\cup_{(C,m)\in\mc{C}_k}C\right) < \delta.
\]
\begin{description}
\item{(b)} Let $I'\subset I_k$ be an interval of length $1$ and let
  $\nu$ be a $2$-form on $s_*^{-1}(I')$ with $|\nu|\le 1$ and
  $|\nabla\nu|\le \delta^{-1}$.  Then
\[
\left|
\frac{i}{2\pi}\int_{s_*\in I'}\nu\wedge F_{\widehat{A}} -
\sum_{(C,m)\in\mc{C}_k} m \int_{C\cap s_*^{-1}(I')}\nu
\right|
\le \delta.
\]
\item{(c)}
\[
\sum_{(C,m)\in\mc{C}_k} m\int_{C\cap (\overline{X_*}-s_*^{-1}(I_0))} d\lambda_Y
< \kappa.
\]
\end{description}
\item Suppose that $\mathbb{I}$ is unbounded from above.  Fix
  $\energy_+\le \mc{K}$, and require nondegenerate Reeb orbits only
  for length at most $\frac{1}{2\pi}\energy_+$.  Assume also that
  $\energy(\frak{c}_+)\le \energy_+$.  Then the preceding conclusions
  hold with $\kappa$ depending on $\mc{K}$ and $\energy_+$, and with
  $\kappa_\delta$ depending only on $\mc{K}$, $\energy_+$, and
  $\delta$.  Moreover, if $\mathbb{I}=\R$ then $\energy(\frak{c}_-)\le
  \energy_++\delta$.
\end{itemize}
\end{proposition}

\begin{proof}
  The first bullet follows from the first bullet in
  Proposition~\ref{prop:4.1x}.  The proof of the rest of
  Proposition~\ref{prop:4.5x} is a slight modification of the proof of
  \cite[Prop.\ 4.5]{e4}.  The latter proof has five parts.  The first
  two parts establish \cite[Lem.\ 4.6, Cor.\ 4.7, Lem.\ 4.8]{e4},
  which are applicable here with the contact manifold $M$ in \cite{e4}
  replaced by $Y_\pm$ or $Y_0$ here.  The third part of the proof of
  \cite[Prop.\ 4.5]{e4} has the following analogue here:

\begin{lemma}
\label{lem:4.9x}
(cf.\ \cite[Lem.\ 4.9]{e4}) Given $\mc{K}\ge 1$, suppose that each
Reeb orbit with length $\le \mc{K}/2\pi$ of $\lambda_\pm$, and
$\lambda_0$ in Case 2, is nondegenerate.  Given also $\varepsilon>0$,
there exists $\kappa\ge 1$ and a neighborhood of the given data set
$(\lambda,J,\mu)$ such that the following holds: Suppose that $r\ge
\kappa$ and that $\frak{d}=(A,\psi=(\alpha,\beta))$ is an instanton
solution to \eqref{eqn:tsw4} as defined with a data set in this
neighborhood of $(\lambda,J,\mu)$ (and any $R$ in Case 2), with
$A_\frak{d}\le r^2$ or $\frak{i}_\frak{d}\ge -r^2$.  Let
$\mathbb{I}\subset\R\setminus I_0$ denote a connected subset of length
at least $16$ such that $\underline{M}(\cdot)\le \mc{K}$ on
$\mathbb{I}$.  Let $\mc{I}$ denote the set of integers $k$ such that
$[k,k+1]\in\mathbb{I}$ and
\[
\frac{i}{2\pi}\int_{s_*^{-1}[k,k+1]}d\lambda_Y\wedge F_{\widehat{A}}
\ge \varepsilon.
\]
Let $I'$ be a component of $\mathbb{I}\setminus
\cup_{k\in\mc{I}}[k,k+1]$.  Then
\[
\frac{i}{2\pi}\int_{s_*^{-1}(I')}d\lambda_Y\wedge F_{\widehat{A}} \ge
-\varepsilon^2.
\]
\end{lemma}

\begin{proof}
  Copy the proof of \cite[Lem.\ 4.9]{e4}.  Wherever the latter proof
  invokes lemmas from \cite[\S3]{e4}, replace these as indicated in
  \S\ref{sec:estimates} above.
\end{proof}

The fourth part of the proof of \cite[Prop.\ 4.5]{e4} carries over
with only notational changes to deduce the second bullet in
Proposition~\ref{prop:4.5x} from Proposition~\ref{prop:4.1x}.

The fifth part of the proof of \cite[Prop.\ 4.5]{e4} carries over to
prove the third bullet in Proposition~\ref{prop:4.5x}, with the
following modification: The key step is to show that given
$\varepsilon_0>0$, if $r$ is sufficiently large, then if $k_-<k_+$ are
integers in $\mathbb{I}$ with $k_+-k_-<\varepsilon_0^{-1}$, then
\begin{equation}
\label{eqn:compareIntervals}
\frac{i}{2\pi}\int_{s_*^{-1}(k_+,k_++1)}ds\wedge \lambda_Y \wedge F_{\widehat{A}} -
\frac{i}{2\pi}\int_{s_*^{-1}(k_-,k_-+1)}ds\wedge \lambda_Y\wedge F_{\widehat{A}} >
-\varepsilon_0.
\end{equation}
If the intervals $(k_-,k_-+1)$ and $(k_+,k_++1)$ are in the same
component of $\R\setminus I_0$, then the inequality
\eqref{eqn:compareIntervals} follows from an integration by parts in
\cite[\S4d, Part 5]{e4}.  So to complete the proof, we just need to
prove \eqref{eqn:compareIntervals} when $k_-+1=k_+\in I_0$.  To
simplify notation, restrict to Case 1, so that $k_-+1=k_+=0$.  The
aforementioned integration by parts can be used to show that the
integrals on the left hand side of \eqref{eqn:compareIntervals} satisfy
\[
\left|\frac{i}{2\pi}\int_{s_*^{-1}(k_\pm,k_\pm+1)}ds\wedge \lambda_Y \wedge
F_{\widehat{A}}
-
\frac{i}{2\pi}\int_{\{0\}\times Y_\pm}\lambda_+\wedge F_{\widehat{A}} \right|
< \frac{\varepsilon_0}{3}
\]
if $r$ is sufficiently large.  So to prove
\eqref{eqn:compareIntervals}, it is enough to show that
\[
\frac{i}{2\pi}\int_X d\lambda\wedge F_{\widehat{A}} >
-\frac{\varepsilon_0}{3}
\]
if $r$ is sufficiently large.  This last inequality follows from the a
priori estimates in Lemma~\ref{lem:3.1x} and \cite[Lem.\ 3.8]{e4}.
\end{proof}

\subsection{Proof of Proposition~\ref{prop:9x}}
\label{sec:ih2}

We now carry over material from \cite[\S5]{e4} to our setting and
prove Proposition~\ref{prop:9x}.  The proof of
Proposition~\ref{prop:9x} uses the following proposition, which is
similar to the $\mathbb{I}=\R$ case of Proposition~\ref{prop:4.5x},
but with the assumption on $\underline{M}$ replaced by an assumption
on $\energy(\frak{c}_+)$.

\begin{proposition}
\label{prop:5.1x}
(cf.\ \cite[Prop.\ 5.1]{e4}) Fix $\mc{K}\ge 1$ and $\energy_+\le
\mc{K}$.  Assume all Reeb orbits of $\lambda_\pm$, and $\lambda_0$ in
Case 2, of length $\le\frac{1}{2\pi}\energy_+$ are nondegenerate.
Then there exists $\kappa\ge 1$, and given $\delta>0$ there exist
$\kappa_\delta\ge 1$ and a neighborhood of the given data set
$(\lambda,J,\mu)$ such that the following holds: Suppose that $r\ge
\kappa_\delta$ and that $\frak{d}=(A,\psi=(\alpha,\beta))$ is an
instanton solution to \eqref{eqn:tsw4} as defined with a data set in
this neighborhood of $(\lambda,J,\mu)$ (and any $R$ in Case 2).
Assume that $A_\frak{d}\le \mc{K}r$ or $\frak{i}_\frak{d}\ge
-\mc{K}r$.  Assume also that $\energy(\frak{c}_+)\le E_+$.  Then:
\begin{itemize}
\item $\energy(\frak{c}_-) \le \energy_+ + \delta$.
\item The first two bullets of Proposition~\ref{prop:4.5x} hold
  with $I=\R$.
\end{itemize}
\end{proposition}

\begin{proof}
  This follows from Proposition~\ref{prop:4.5x} if we can show that
  given an instanton solution $\frak{d}$ to \eqref{eqn:tsw4} with
  $A_\frak{d}\le \mc{K}r$ or $\frak{i}_\frak{d}\ge -\mc{K}r$ and
  $\energy(\frak{c}_+)<\energy_+$, there exists an $r$-independent
  upper bound on $\underline{M}(\cdot)$ when $r$ is sufficiently
  large.  We now explain how to obtain such a bound on $\underline{M}$
  by modifying the arguments in \cite[\S5]{e4}, which obtain an
  analogous bound on $\underline{M}$ when $\overline{X_*}=\R\times M$.

First note that our assumptions imply that
\begin{equation}
\label{eqn:Abound}
A_\frak{d} < c_0(\mc{K}+1)r.
\end{equation}
To prove \eqref{eqn:Abound}, we can assume that
$\frak{i}_\frak{d}>-\mc{K}r$, and this
implies that
\[
\frak{a}(\frak{c}_-) - \frak{a}(\frak{c}_+) <
\frak{a}^f(\frak{c}_-) - \frak{a}^f(\frak{c}_+) + c_0\mc{K}r.
\]
By \cite[Pop.\ 4.11]{e1}, the assumption $\energy(\frak{c}_+)<\mc{K}$
implies that $\frak{a}^f(\frak{c}_+) > - c_0\mc{K} r$, see
\eqref{eqn:invoke1} below.  Meanwhile, an almost verbatim version of
an argument from \cite[\S5d]{e4} proves that
$\frak{a}^f(\frak{c}_-)\le c_0$ when $\frak{i}_\frak{d} > -\mc{K}r$.
The inequality \eqref{eqn:Abound} follows.

We now explain how to obtain a bound on $\underline{M}(s)$ when
 $s\ge R+2$.  If the interval
 $[s,s+1]$ does not intersect $I_0$, define
\[
\underline{\energy}(s) \eqdef i\int_{s_*^{-1}[s,s+1]}ds\wedge\lambda_Y\wedge F_A.
\]
When $s\ge R+2$, we will first obtain a bound $\underline{\energy}(s)$, and
then use this to bound $\underline{M}(s)$.

To obtain bounds on $\underline{\energy}$, we need three inequalities.
To state these, recall the reference connection $A_E$ from
\eqref{eqn:AE} and the $2$-form $\nu_E$ defined from its curvature in
\eqref{eqn:nuE}.  Let $u_\pm:Y_\pm\to S^1$, and also $u_0:Y_0\to S^1$
in Case 2, be gauge transformations. If $s>R$, write the connection
component of $u_+\cdot\frak{d}$ as $A_{E}+\hat{a}_+$, and define
\[
\frak{p}_+(s) \eqdef -i\int_{s_*^{-1}(s)} \hat{a}_+\wedge \nu_E.
\]
Here and below, our convention is that $R=0$ in Case 1.  Define
$\frak{p}_-(s)$ analogously if $s<-R$, and define $\frak{p}_0$
analogously in Case 2 if $-R<s<R$.

The first inequality asserts that if $s>R+3$, then
\begin{equation}
\label{eqn:Ebound1}
\begin{split}
r \underline{\energy}(s-1) \le & -c_0\frak{a}(u_+\frak{c}_+) +
c_0(r+\frak{a}(\frak{d}|_{s_*=R+3}) - \frak{a}(\frak{c}_+))\\
& + c_0
r^{2/3} \sup_{x\ge s}|\underline{\energy}(x)|^{4/3} +
c_0\sup_{[s,s+1]}|\frak{p}_+|.
\end{split}
\end{equation}
The second inequality asserts that if $s<s'<-R-3$, then for suitable
$s_-\in[s',s'+1]$ we have
\begin{equation}
\label{eqn:Ebound2}
\begin{split}
r\underline{\energy}(s-1) \le & -c_0\frak{a}(u_-\frak{d}|_{s_*=s_-}) +
c_0(r+\frak{a}(\frak{c}_-) - \frak{a}(\frak{d}_{s_*=s_-}))\\
&+ c_0r^{2/3}\sup_{x\in[s,s_-]}|\underline{\energy}(x)|^{4/3} +
c_0\sup_{[s,s+1]}|\frak{p}_-|.
\end{split}
\end{equation}
Here $s_-$ is ``suitable'' if $O(s_-)\le \int_{s'}^{s'+1}O(s_*)ds_*$,
where $O(s)$ is defined in \eqref{eqn:O(s)} below.  The third
inequality asserts that in Case 2, if $-R+3<s<s'<R-3$, and if
$s_0\in[s',s'+1]$ satisfies $O(s_0)\le \int_{s'}^{s'+1}O(s_*)ds_*$,
then
\begin{equation}
\label{eqn:Ebound3}
\begin{split}
  r\underline{\energy}(s-1) \le & -c_0\frak{a}(u_0\frak{d}|_{s_*=s_0}) +
  c_0(r+\frak{a}(\frak{d}|_{s_*=-R+3}) -
  \frak{a}(\frak{d}|_{s_*=s_0}))\\
  & + c_0r^{2/3}\sup_{x\in[s,s_0]}\underline{\energy}(x)|^{4/3} +
  c_0\sup_{[s,s+1]}|\frak{p}_0|.
\end{split}
\end{equation}
The inequalities \eqref{eqn:Ebound1}--\eqref{eqn:Ebound3} are all
proved analogously to \cite[(5-18)]{e4}.

To exploit the inequalities \eqref{eqn:Ebound1}--\eqref{eqn:Ebound3},
we need appropriate bounds on the terms that do not involve
$\underline{\energy}(\cdot)$.  We first observe that the action differences
in \eqref{eqn:Ebound1}--\eqref{eqn:Ebound3} are bounded by
\begin{equation}
\label{eqn:9.18}
\begin{split}
\frak{a}(\frak{d}|_{s_*=R+3}) - \frak{a}(\frak{c}_+) & \le
c_0(\mc{K}+1)r,\\
\frak{a}(\frak{c}_-) -
\frak{a}(\frak{d}|_{s_*=s_-}) & \le c_0(\mc{K}+1)r,\\
\frak{a}(\frak{d}|_{s_*=-R+3}) -
\frak{a}(\frak{d}|_{s_*=s_0}) &\le c_0(\mc{K}+1)r.
\end{split}
\end{equation}
To prove \eqref{eqn:9.18}, first use Lemma~\ref{lem:3.4x}(a),(b) to see
that each action difference is bounded from above by $A_\frak{d}+c_0r$, and
then use \eqref{eqn:Abound}.

To bound the remaining terms in \eqref{eqn:Ebound1}, the discussion in
\cite[\S5d]{e4} finds a gauge transformation $u_+:Y_+\to S^1$ such
that $\frak{a}(u_+\frak{c}_+) > -c_0\energy_+r$ and $\lim_{s\to \infty}
|\frak{p}_+|\le c_0$.  The first of these conditions allows us to
replace \eqref{eqn:Ebound1} by
\begin{equation}
\label{eqn:Ebound4}
\underline{\energy}(s-1) \le c_0(1+\mc{K}) + c_0r^{-1/3}\sup_{x\ge
  s}|\underline{\energy}(x)|^{4/3} + c_0 r^{-1} \sup_{[s,s+1]}|\frak{p}_+|
\end{equation}
for $s\ge R+3$.  The arguments in \cite[\S5d]{e4} can be applied verbatim
using \eqref{eqn:Ebound4} to give the bound $\underline{\energy}(s) \le
c_0\mc{K}$ for $s\ge R+2$.  The arguments in \cite[\S5d]{e4} also
explain why this last bound implies that $\underline{M}(s)\le
c_0\mc{K}$ for $s\ge R+2$.

It remains to bound $\underline{M}(s)$ for $s\le R+2$.  Let
$t_*\in\{-R,R\}$ and suppose that $\underline{\energy}$ (where defined) and
$\underline{M}$ have been bounded above by $c_0$ on $[t_*+2,\infty)$.
Let $t_{**}=-\infty$ if $R<10$ or $t_*=-R$, and let $t_{**}=-R+2$ otherwise.
We now explain how to extend a bound of this sort on $\underline{M}$
and $\underline{\energy}$ to the interval $(t_{**},\infty)$, in two steps.
Applying this procedure once if $R<10$, and twice if $R\ge 10$, will
give the desired bound on $\underline{M}(s)$ for all $s\in\R$.

{\em Step 1.\/} This step bounds $\underline{\energy}$ (where defined) and
$\underline{M}$ on $[t_*-100,t_*+2]$.

For $s\in \R\setminus I_0$, define
\[
\energy(s) \eqdef i\int_{s_*^{-1}(s)}\lambda_Y\wedge F_A.
\]
Recall that $d\widetilde{\lambda}=\sigma\hat{\omega}$ where
$\sigma:\overline{X_R}\to\R$ agrees with $2e^{2s_*}$ on
$s_*^{-1}(\R\setminus I_0)$.  Now use Stokes' theorem and
\eqref{eqn:tsw4} to see that
\begin{equation}
\label{eqn:EStokes}
\begin{split}
\energy(s) &=
ie^{-2s}\int_{s_*^{-1}(s)}\widetilde{\lambda}\wedge F_A\\
&= ie^{-2s}\int_{s_*^{-1}(-\infty,s]}\sigma\hat{\omega}\wedge F_A\\
&= re^{-2s}\int_{s_*^{-1}(-\infty,s]}\sigma(1-|\alpha|^2+|\beta|^2) +
O(1).
\end{split}
\end{equation}
Integrating this equation over $s\in[t_*+2,t_*+3]$ (or a slight upward
translation of this interval as needed to avoid $I_0$) and using the a
priori bounds in Lemma~\ref{lem:3.1x} shows that the bound on
$\underline{\energy}(s)$ for $s\ge t_*+2$ gives rise to a bound on
$\underline{M}$ on $[t_*-100,t_*+2]$.

Similarly to \eqref{eqn:EStokes}, if $s_-<s_+$ are in $\R\setminus
I_0$ then
\[
\energy(s_+) - e^{-2(s_+-s_-)}\energy(s_-) =
re^{-2s_+}\int_{s_*^{-1}[s_-,s_+]}\sigma(1-|\alpha|^2+|\beta|^2) +
O(1).
\]
Using this equation and Lemma~\ref{lem:3.1x} shows that the bound on
$\underline{\energy}(s)$ for $s\ge t_*+2$ gives rise to a bound on
$\underline{\energy}$ (where defined) on $[t_*-100,t_*+2]$.

{\em Step 2.\/} We now extend the bounds on $\underline{\energy}$ and
$\underline{M}$ over $[t_{**},t_*-100]$.  We assume below that $t_{**}<
t_*-100$.

For $s\in\R\setminus I_0$ define
\begin{equation}
\label{eqn:O(s)}
O(s) \eqdef \int_{s_*^{-1}(s)}\left(|\mc{B}_{(A,\psi)}|^2 +
  r|D_{A(s)}\psi|^2\right),
\end{equation}
where $\mc{B}_{(A,\psi)}$ was defined in \eqref{eqn:mcb}.  Also,
define $\underline{O}(s)=\int_s^{s+1}O(s_*)ds_*$ when $[s,s+1]$ does
not intersect $I_0$.  Write $Y=Y_-$ when $t_{**}=-\infty$ and $Y=Y_0$
when $t_{**}$ is finite.  There exists $s_Y\in[t_*-100,t_*-99]$ such
that $O(s_Y)\le \underline{O}(t_*-100)$.  Then
\begin{equation}
\label{eqn:Obound}
O(s_Y)\le \underline{O}(t_*-100)\le 2A_\frak{d}+c_0r
\le c_0(\mc{K}+1)r
\end{equation}
by Lemma~\ref{lem:3.4x}(d) and the inequality \eqref{eqn:Abound}.

Next, note that there is a map $u_Y:Y\to S^1$ such that the connection
component of $u_Y\cdot \frak{d}|_{s_Y}$ can be written as
$A_E+\hat{a}_Y$ where $\hat{a}_Y$ is a co-closed, $i$-valued $1$-form
on $Y$ whose $L^2$ orthogonal projection to the space of harmonic
$1$-forms is bounded by $c_0$.  Combining this with
\eqref{eqn:Ebound2} or \eqref{eqn:Ebound3} as appropriate with
$s'=t_*-100$, and using the bound \eqref{eqn:Obound} on $O(s_Y)$ and
the bound on $\underline{\energy}(s)$ for $s\in[t_*-100,t_*-2]$, the
arguments leading to \cite[Eq.\ (5.14)]{e4} can be used to obtain a lower
bound
\begin{equation}
\label{eqn:9.27}
\frak{a}(A_E+\hat{a}_Y) \ge -c_0r.
\end{equation}

To continue, extend the map $u_Y$ to all of $[t_{**},t_*-99]$ to be
independent of $s_*$, and replace $\frak{d}$ on this portion of
$\overline{X_*}$ by $u_Y\cdot\frak{d}$.  It follows from
\eqref{eqn:9.18}, \eqref{eqn:9.27}, and the relevant inequality
\eqref{eqn:Ebound2} or \eqref{eqn:Ebound3} with $s_-$ or $s_0$ set
equal to $s_Y$, that for $s\in[t_{**},t_*-100]$ we have
\begin{equation}
\label{eqn:Ebound5}
\underline{\energy}(s-1) \le c_0(1+\mc{K}) +
c_0r^{-1/3}\sup_{x\in[s,s_Y]}\left|\underline{\energy}(x)\right|^{1/3} +
c_0r^{-1}\sup_{[s,s+1]}|\frak{p}_Y|.
\end{equation}
Here $\frak{p}_Y$ denotes $\frak{p}_-$ or $\frak{p}_0$ as appropriate.
Moreover, we have
\begin{equation}
\label{eqn:9.28}
|\frak{p}_Y|\le c_0\mc{K} \quad\mbox{on $[t_*-100,t_*-99]$}.
\end{equation}
To see why \eqref{eqn:9.28} is true, note that by our choice of $u_Y$
we have $|\frak{p}_Y(s_Y)|\le c_0$.  Meanwhile, \cite[Lem.\ 3.9]{e4}
bounds the derivative of the function $s\mapsto |\frak{p}_Y(s)|$ by
$c_0(1+|M(s)|)$.  Integrating this derivative bound and applying the
conclusions from Step 1 gives \eqref{eqn:9.28}.

Granted \eqref{eqn:Ebound5} and \eqref{eqn:9.28}, arguments from
\cite[\S5d]{e4} can be used in an almost verbatim fashion to bound
$\underline{\energy}$ on $[t_{**},t_*-100]$ by $c_0\mc{K}$.  One just needs
to replace all references to the $s\to\infty$ limit of $\frak{d}$ by
$\frak{d}_{s_*=t_*-100}$.  As noted previously, arguments from
\cite[\S5b]{e4} can be used with this bound on $\underline{\energy}$ to
bound $\underline{M}$ by $c_0\mc{K}$ on $[t_{**},t_*-100]$.
\end{proof}

\begin{proof}[Proof of Proposition~\ref{prop:9x}.]
  The first two bullets of Proposition~\ref{prop:9x} follow
  immediately from Proposition~\ref{prop:5.1x}.  The third bullet of
  Proposition~\ref{prop:9x} is deduced from
  Proposition~\ref{prop:5.1x} in the same way that the third bullet of
  \cite[Prop.\ 5.5]{e4} is deduced from \cite[Prop.\ 5.1]{e4} in
  \cite[\S5e]{e4}.
\end{proof}

\subsection{Proof of Propositions~\ref{prop:hol} and \ref{prop:comp}}
\label{sec:finalproofs}

\begin{proof}[Proof of Proposition~\ref{prop:hol}.]
  (a) We consider Case 1 of the geometric setup in
  \S\ref{sec:9xprelim}. If the perturbations $\frak{p}_\pm$ and
  $\frak{p}$ are zero, then assertion (a) follows immediately from
  Case 1 of Proposition~\ref{prop:9x}.  For the case when the
  perturbations $\frak{p}_\pm$ and $\frak{p}$ are not zero, the proof
  has two steps.

{\em Step 1.\/} We claim that if $r>c_0$ and if the $\mc{P}$-norm of
$\frak{p}$ is bounded by $c_0^{-1}$, then an instanton $\frak{d}$ as
in (a) satisfies
\begin{equation}
\label{eqn:invoke3}
\frak{a}(\frak{c}_-) - \frak{a}(\frak{c}_+) < (c_0 + 2\pi L)r.
\end{equation}
Here and below, $\frak{a}(\frak{c}_\pm)$ denotes the sum of the action
functional \eqref{eqn:a} for $Y_\pm$ and the abstract perturbation
$\frak{p}_\pm$.  To prove \eqref{eqn:invoke3}, first note that the
same integration by parts that proves Lemma~\ref{lem:3.4x}(a),(b)
implies that
\begin{equation}
\label{eqn:acor}
\frak{a}(\frak{c}_-) > \frak{a}(\frak{c}_+) - c_0r
\end{equation}
if the $\mc{P}$-norm of $\frak{p}$ is bounded by $c_0^{-1}$.  (See
also the remark after \cite[Prop.\ 24.6.4]{km}.)  Since $\frak{d}$ has
index $0$, it follows from \eqref{eqn:af}, \eqref{eqn:if} and
\eqref{eqn:acor} that
\begin{equation}
\label{eqn:afcor}
\frak{a}^f(\frak{c}_-) > \frak{a}^f(\frak{c}_+) - c_0r.
\end{equation}
Here $f$ denotes the spectral flow function on $Y_\pm$ defined using
$\frak{p}_\pm$.  Meanwhile, by \cite[Prop.\ 4.11]{e1} we have
\begin{equation}
\label{eqn:invoke1}
\frak{a}^f(\frak{c}_+) = -\frac{1}{2}r \energy(\frak{c}_+)(1+o(1)).
\end{equation}
By this and \eqref{eqn:afcor} we have
$\frak{a}^f(\frak{c}_-)>-(c_0+2\pi L)r$.  Consequently \cite[Prop.\ 4.11]{e1}
can be invoked a second time to give
\begin{equation}
\label{eqn:invoke2}
\frak{a}^f(\frak{c}_-) = -\frac{1}{2}r\energy(\frak{c}_-)(1+o(1)).
\end{equation}
On the other hand, \cite[Lem.\ 2.3]{e4} implies that
$\energy(\frak{c}_-)>-c_0$.  This together with \eqref{eqn:invoke2} implies
that $\frak{a}^f(\frak{c}_-) < c_0r$.  Since $\frak{d}$ has index $0$,
it follows from this last inequality and \eqref{eqn:invoke1} that
\eqref{eqn:invoke3} holds.

{\em Step 2.\/} Now let $r$ be large, and assume to get a
contradiction that the conclusion of Proposition~\ref{prop:hol}(a) is
false.  Then there exist data $(J,\mu)$ and a
sequence of perturbations $\{\frak{p}_{k}\}_{k=1,2,\ldots}$ with
$\lim_{k\to\infty}\frak{p}_{k}=0$, for which there is an instanton
$\frak{d}_{k}$ with index $0$ and $\energy(\frak{c}_{k+})<2\pi L$ such
that assertion (i) or (ii) in Proposition~\ref{prop:hol}(a) fails.
Here $\frak{c}_{k\pm}$ denotes the $s\to\pm\infty$ limit of
$\frak{c}_{k}$.  By passing to a subsequence we may assume that
$\frak{c}_{k+}$ does not depend on $k$, so that we can denote it
by $\frak{c}_+$.

Now \eqref{eqn:invoke3} applies to
each $\frak{d}_{k}$ to give
\begin{equation}
\label{eqn:invoke4}
\frak{a}(\frak{c}_{k-}) - \frak{a}(\frak{c}_+) < (c_0 + 2\pi L)r.
\end{equation}
Consequently, \cite[Prop.\ 24.6.4]{km} implies that the sequence of
instantons $\{\frak{d}_{k}\}_{k=1,2,\ldots}$ has a subsequence that
converges in the sense of \cite[\S26]{km} to a broken trajectory, from
$\frak{c}_+$ to some generator $\frak{c}_-$, that is defined using the
equations \eqref{eqn:tsw4} on $\overline{X}$ and
\eqref{eqn:Tinstanton} on $\R\times Y_\pm$, without abstract
perturbations.  In particular, we can pass to a further subsequence so
that $\frak{c}_{k-}=\frak{c}_-$ for all $k$.  Let
$\{\frak{d}^n\}_{n=1,\ldots,N}$ denote the ordered set of instantons
that comprise the limiting broken trajectory.  Let $\frak{c}^n_\pm$
denote the $s\to\pm\infty$ limit of $\frak{d}^n$.  These limits
satisfy $\frak{c}^1_-=\frak{c}_-$, $\frak{c}^N_+=\frak{c}_+$, and
$\frak{c}^n_+ = \frak{c}^{n+1}_-$ for $1\le n < N$.

There is a unique $n_0$ such that $\frak{d}^{n_0}$ is an instanton on
$\overline{X}$.  By Lemma~\ref{lem:3.4x}(a),(b), we have
$\frak{a}(\frak{c}^{n_0}_-) > \frak{a}(\frak{c}^{n_0}_+) - c_0r$; and
by \cite[Lem.\ 3.4]{e4} we have $\frak{a}(\frak{c}^n_-) >
\frak{a}(\frak{c}^n_+)$ for all $n\neq n_0$.  These inequalities
together with \eqref{eqn:invoke4} imply that $\frak{a}(\frak{c}^n_-) -
\frak{a}(\frak{c}^n_+) \le c_0r$ for each $n$.  Consequently, if $r$
is sufficiently large, then Proposition~\ref{prop:9x} applies to
$\frak{d}^{n_0}$, and \cite[Prop.\ 5.5]{e4} applies to $\frak{d}^n$
for each $n\neq n_0$, to produce a broken $J$-holomorphic curve.
These propositions (or the existence of these broken $J$-holomorphic
curves) also imply that $\energy(\frak{c}_-)<2\pi L$ if $r$ is
sufficiently large.  The concatenation of these $N$ broken
$J$-holomorphic curves is a broken $J$-holomorphic curve from
$\Theta_+$ to $\Theta_-$, where $\Theta_\pm$ is determined by
$\frak{c}_\pm$ via Proposition~\ref{prop:Liso}(a).  It follows that if
$r$ is sufficiently large, then assertions (i) and (ii) in
Proposition~\ref{prop:hol}(a) are true for each $\frak{d}_{k}$.  This
is the desired contradiction.

(b) This is essentially the same as the proof of (a), the only
difference being that in Step 2, one now considers a sequence
$\{\frak{d}_{k}\}$ of instantons which solves the perturbed equations
\eqref{eqn:tsw4} for the data corresponding to some $t_{k}\in[0,1]$.
One then passes to a subsequence such that $\lim_{k\to
  \infty}t_{k}=t_*$.  The arguments in \cite[\S24]{km} can be used to
show that the sequence of instantons $\frak{d}_{k}$ has a subsequence
which converges to a broken trajectory for the data corresponding to
$t=t_*$.  Now the constants in Lemma~\ref{lem:3.4x} and
Proposition~\ref{prop:9x}, because they are stable, can be chosen to
work for the data corresponding to all $t\in[0,1]$.  Thus if $r$ is
sufficiently large (independently of the value of $t_*)$, then the
rest of the proof of (a) can be repeated verbatim to prove part (b).
\end{proof}

\begin{proof}[Proof of Proposition~\ref{prop:comp}.]
  We now consider Case 2 of the geometric setup in
  \S\ref{sec:9xprelim}.  Recall the $1$-form $\widetilde{\lambda}_R$
  on $\overline{X_R}$ defined in \eqref{eqn:lambdatilde}.  Define a
  $1$-form $\lambda'$ on $X$ to agree with $\widetilde{\lambda}_0$ on
  $[-\varepsilon,\varepsilon]\times Y_0$, where $\varepsilon$ was
  fixed in \S\ref{sec:9xprelim}, and to agree with $\lambda$ on the
  rest of $X$.  Note that the exact symplectic cobordism
  $(X,\lambda')$ from $(Y_+,\lambda_+)$ to $(Y_-,\lambda_-)$ is
  strongly homotopic to $(X,\lambda)$.  So by
  Corollary~\ref{cor:cobl}(c), if $r$ is sufficiently large then
  $\widehat{HM}^*_L(X,\lambda)=\widehat{HM}^*_L(X,\lambda')$.  Thus to
  prove Proposition~\ref{prop:comp}, it is enough to show that if $r$
  is sufficiently large then
\begin{equation}
\label{eqn:lastets}
\widehat{HM}^*_L(X,\lambda') = \widehat{HM}^*_L(X^-,\lambda^-) \circ
\widehat{HM}^*_L(X^+,\lambda^+).
\end{equation}

To prove \eqref{eqn:lastets}, we fix $r$ large and vary $R$ in Case 2.
Let $\frak{p}_\pm,\frak{p}_0$ be abstract perturbations as needed to
define the respective Seiberg-Witten Floer chain complexes on
$Y_\pm,Y_0$.  Extend these to abstract perturbations $\frak{p}^\pm$ on
$\overline{X^\pm}$ as needed to define chain maps that induce
$\widehat{HM}^*_L(X^\pm,\lambda^\pm)$; denote these chain maps by
$\Phi_\pm$.  The perturbations $\frak{p}^\pm$, with suitable cutoff
functions, then determine an abstract perturbation $\frak{p}_R$ on
$\overline{X}_R$, as explained in \cite[\S11, \S24.1]{km} (see
\cite[Eq.\ (24.1)]{km}). Let
\[
\Phi_R:\widehat{CM}^*_L(Y_+;\lambda_+,J_+,r) \longrightarrow
\widehat{CM}^*_L(Y_-;\lambda_-,J_-,r)
\]
denote the chain map defined by counting index $0$ instantons on
$\overline{X}_R$.  (One may need to perturb $\frak{p}_R$ to obtain
transversality, in which case the chain map will depend on this
perturbation.)  It follows as in the proof of
Proposition~\ref{prop:hol}(a) that if $r$ is sufficiently large, then
for any $R$, if the abstract perturbations are sufficiently small then
$\Phi_R$ is defined.

When $R=0$, the induced map on homology is
\[
(\Phi_0)_*=\widehat{HM}^*_L(X,\lambda'),
\]
because by construction
$\left(\overline{X}_0,\widetilde{\lambda}_0\right) =
\left(\overline{X},\widetilde{\lambda'}\right)$.  On the other hand,
for $R>0$ the manifold
$\left(\overline{X}_R,\widetilde{\lambda}_R\right)$ is obtained by
gluing $\left(\overline{X^-},\widetilde{\lambda^-}\right)$ with the
$s\ge R$ part of the positive end removed to
$\left(\overline{X^+},\widetilde{\lambda^+}\right)$ with the $s\le -R$
part of the negative end removed.  It then follows from \cite[Prop.\
26.1.6]{km} that there exists $R_0$ such that if $R\ge R_0$, then
$\Phi_R$ is defined without any further perturbation of $\frak{p}_R$,
and
\[
\Phi_R=\Phi_-\circ\Phi_+.
\]
So to complete the proof of \eqref{eqn:lastets}, it is enough to show
that the chain maps
\[
\Phi_0,\Phi_{R_0}:\widehat{CM}^*_L(Y_+;\lambda_+,J_+,r) \longrightarrow
\widehat{CM}^*_L(Y_-;\lambda_-,J_-,r)
\]
are chain homotopic.  To construct the desired chain homotopy one
counts index $-1$ instantons in the family $\{\overline{X_R}\mid
R\in[0,R_0]\}$ with a generic small family of abstract perturbations.
If $r$ is sufficiently large, then this chain homotopy will be well
defined as in the proof of Proposition~\ref{prop:hol}(b).
\end{proof}


\begin{thebibliography}{99}

\bibitem{aps1} M. Atiyah, V. Patodi, and I. Singer, {\em Spectral
    asymmetry and Riemannian geometry. I.\/}, Math. Proc. Cambridge
  Philos. Soc. {\bf 77} (1975), 43--69.

\bibitem{federer} H. Federer, {\em Geometric measure theory},
  Springer, 1969.

\bibitem{hummel} C. Hummel, {\em Gromov's compactness theorem for
    pseudo-holomorphic curves\/}, Progress in Mathematics 151,
  Birkh\"auser, 1997.

\bibitem{pfh2} M. Hutchings, {\em An index inequality for embedded
    pseudoholomorphic curves in symplectizations\/},
  J. Eur. Math. Soc. {\bf 4\/} (2002), 313--361.

\bibitem{ir} M. Hutchings, {\em The embedded contact homology index
    revisited\/}, New perspectives and challenges in symplectic field
  theory, 263--297, CRM Proc. Lecture Notes 49, Amer. Math. Soc.,
  2009.

\bibitem{icm} M. Hutchings, {\em Embedded contact homology and its
    applications\/}, in Proceedings of the 2010 ICM, vol. II, 1022-1041.

\bibitem{qech} M. Hutchings, {\em Quantitative embedded contact
    homology\/}, J.\ Diff.\ Geom. {\bf 88\/} (2011), 231--266.

\bibitem{bn} M. Hutchings, {\em Lecture notes on embedded contact homology\/}, arXiv:1303.5789.

\bibitem{t3} M. Hutchings and M. Sullivan, {\em Rounding corners of
  polygons and the embedded contact homology of $T^3$\/}, Geometry and
  Topology {\bf 10\/} (2006), 169--266.

\bibitem{obg1} M. Hutchings and C. H. Taubes, {\em Gluing
  pseudoholomorphic curves along branched covered cylinders I\/},
  J. Symplectic Geom. {\bf 5\/} (2007), 43--137.

\bibitem{obg2} M. Hutchings and C. H. Taubes, {\em Gluing
    pseudoholomorphic curves along branched covered cylinders II\/},
  J. Symplectic Geom. {\bf 7\/} (2009), 29--133.

\bibitem{wh} M. Hutchings and C. H. Taubes, {\em The Weinstein
    conjecture for stable Hamiltonian structures\/}, Geometry and
  Topology {\bf 13\/} (2009), 901--941.

\bibitem{cc} M. Hutchings and C. H. Taubes, {\em Proof of the Arnold
    chord conjecture in three dimensions I\/}, Math. Res. Lett. {\bf
    18} (2011), 295--313.

\bibitem{km} P.B. Kronheimer and T.S. Mrowka, {\em Monopoles and
    three-manifolds\/}, Cambridge University Press, 2008.

\bibitem{algebraic} J. Latschev and C. Wendl, {\em Algebraic torsion
    in contact manifolds\/}, GAFA {\bf 21\/} (2011), 1144-1195.

\bibitem{morgan} F. Morgan, {\em Geometric Measure Theory, a
    Beginner's Guide\/}, fourth edition, Elsevier/Academic Press,
  Amsterdam, 2009.

\bibitem{morrey} C. Morrey, {\em Multiple integrals in the calculus of
    variations\/}, Springer-Verlag, 1966.

\bibitem{swtogr} C. H. Taubes, {\em $SW\Rightarrow Gr$: from the
    Seiberg-Witten equations to pseudo-holomorphic curves}, in {\em
    Seiberg-Witten and Gromov Invariants for Symplectic
    4-manifolds\/}, International Press, Somerville MA 2000.

\bibitem{grtosw} C. H. Taubes, {\em $Gr\Rightarrow SW$: from
    pseudo-holomorphic curves to Seiberg-Witten solutions}, loc. cit.

\bibitem{gr=sw} C. H. Taubes, {\em $Gr=SW$: counting curves and
    connections}. loc. cit.

\bibitem{tw1} C. H. Taubes, {\em The Seiberg-Witten equations and the
    Weinstein conjecture\/}, Geom. Topol. {\bf 11\/} (2007),
  2117-2202.

\bibitem{tw2} C. H. Taubes, {\em The Seiberg-Witten equations and the
    Weinstein conjecture. II. More closed integral curves of the Reeb
    vector field\/}, Geom. Topol. {\bf 13} (2009), 1337-1417.

\bibitem{e1} C. H. Taubes, {\em Embedded contact homology and
    Seiberg-Witten Floer homology I\/}, Geometry and Topology {\bf
    14\/} (2010), 2497--2581.

\bibitem{e2} C. H. Taubes, {\em Embedded contact
  homology and Seiberg-Witten Floer homology II\/}, Geometry and Topology {\bf
    14\/} (2010), 2583--2720.

\bibitem{e3} C. H. Taubes, {\em Embedded contact
  homology and Seiberg-Witten Floer homology III\/}, 
Geometry and Topology {\bf
    14\/} (2010), 2721--2817.

\bibitem{e4} C. H. Taubes, {\em Embedded contact
  homology and Seiberg-Witten Floer homology IV\/}, 
Geometry and Topology {\bf
    14\/} (2010), 2819--2960.

\bibitem{e5} C. H. Taubes, {\em Embedded contact homology and
    Seiberg-Witten Floer homology V\/}, 
Geometry and Topology {\bf
    14\/} (2010), 2961--3000.

\bibitem{wolfson} J. Wolfson, {\em Gromov's compactness of
    pseudoholomorphic curves and symplectic geometry\/},
  J. Diff. Geom. {\bf 28} (1988), 383--405.

\bibitem{ye} R. Ye, {\em Gromov's compactness theorem for pseudo holomorphic curves\/}, Trans. Amer. Math. Soc. {\bf 342} (1994), 671--694.

\end{thebibliography}
\end{document}